\documentclass[english]{amsart}
\usepackage{amsmath}
\usepackage{amssymb}
\usepackage{amscd} \usepackage{tabularx}
\usepackage[all,cmtip,line]{xy}

\newcommand{\omegat}{\tilde{\omega}}

\newcommand{\ra}{\rightarrow}

\newcommand{\into}{\hookrightarrow}

\newcommand{\iso}{\stackrel{\sim}{\ra}}

\newcommand{\NN}{\mathbf{N}}

\newlength{\ownl}

\newcommand{\Wconj}{W^{\operatorname{explicit}}}
\newcommand{\WBT}{W^{\operatorname{BT}}}

\newcommand{\Dpst}{\operatorname{D_{pst}}}

\newcommand{\Art}{{\operatorname{Art}}}

\newcommand{\End}{{\operatorname{End}\,}}

\newcommand{\SW}{{\operatorname{sw}}}
\newcommand{\Frob}{{\operatorname{Frob}}}
\newcommand{\Gal}{{\operatorname{Gal}\,}}

\newcommand{\Hom}{{\operatorname{Hom}\,}}

\renewcommand{\Im}{{\operatorname{Im}\,}}

\newcommand{\ord}{{\operatorname{ord}}}

\newcommand{\WD}{{\operatorname{WD}}}

\newcommand{\Spec}{{\operatorname{Spec}\,}}

\newcommand{\ad}{{\operatorname{ad}\,}}

\newcommand{\gr}{{\operatorname{gr}\,}}

\newcommand{\rec}{{\operatorname{rec}}}

\newcommand{\rk}{{\operatorname{rk}\,}}
\newcommand{\wt}{\widetilde}

\newcommand{\lps}{[\![}
\newcommand{\rps}{]\!]}

\newcommand{\Wcris}{W^{\cris}}

\newcommand{\Wexplicit}{W^{\operatorname{explicit}}}

\newcommand{\cris}{{\operatorname{cris}}}

\newcommand{\loc}{{\operatorname{loc}}}

\newcommand{\ab}{{\operatorname{ab}}}

\newcommand{\univ}{{\operatorname{univ}}}

\newcommand{\A}{{\mathbb{A}}}

\newcommand{\C}{{\mathbb{C}}}

\newcommand{\F}{{\mathbb{F}}}

\newcommand{\Q}{{\mathbb{Q}}}

\newcommand{\T}{{\mathbb{T}}}

\newcommand{\Z}{{\mathbb{Z}}}

\renewcommand{\O}{{\mathcal{O}}}

\newcommand{\swsigma}{{\sigma}}

\newcommand{\pibar}{\overline{{\pi}}}

\newcommand{\tI}{\widetilde{{I}}}

\newcommand{\tQ}{\widetilde{{Q}}}

\newcommand{\tS}{\widetilde{{S}}}
\newcommand{\tT}{\widetilde{{T}}}

\newcommand{\tv}{{\widetilde{{v}}}}

\newcommand{\varepsilonbar    }{\overline{\varepsilon}}

\newcommand{\Qbar}{{\overline{\Q}}}
\newcommand{\Fbar}{{\overline{\F}}}

\newcommand{\To}{\longrightarrow}\newcommand{\m}{\mathfrak{m}}
\newcommand{\psibar}{\overline{\psi}}
\newcommand{\onto}{\twoheadrightarrow}
\newcommand{\isoto}{\stackrel{\sim}{\To}} 
 \newcommand{\St}{\operatorname{St}}

\newcommand{\bigO}{\mathcal{O}}

\newcommand{\sq}{\square}

\newcommand{\taubar}{\overline{\tau}}

\newcommand{\bb}{\mathbb} 
\newcommand{\mc}{\mathcal}
\newcommand{\mf}{\mathfrak}

\newcommand{\cC}{\mathcal{C}}

\newcommand{\cF}{\mathcal{F}}
\newcommand{\cG}{\mathcal{G}}

\newcommand{\cO}{\mathcal{O}}

\newcommand{\cS}{\mathcal{S}}
\newcommand{\cT}{\mathcal{T}}

\newcommand{\rhobar}{\overline{\rho}} 
\newcommand{\chibar}{\overline{\chi}} 
\newcommand{\rbar}{{\bar{r}}}

\newcommand{\rbarwtvL}{\rbar|_{G_{L_{\wt{v}}}}}

\newcommand{\GL}{\operatorname{GL}}

\newcommand{\PGL}{\operatorname{PGL}}
\newcommand{\HT}{\operatorname{HT}}

\newcommand{\glo}{\operatorname{gl}}

 \newcommand{\Qp}{\Q_p}
 \newcommand{\Zp}{\Z_p}
\newcommand{\Ql}{\Q_l} 
\newcommand{\Qpbar}{{\overline{\Q}_p}}
\newcommand{\Qpbartimes}{{\overline{\Q}^\times_p}}

\newcommand{\Fpbar}{{\overline{\F}_p}}

\newcommand{\PSL}{\operatorname{PSL}}

\newcommand{\Sym}{\operatorname{Sym}}

\newcommand{\dR}{\mathrm{dR}}

\usepackage{amsthm}
 \newtheorem{ithm}{Theorem}
\newtheorem{athm}{Theorem}[section]

\newtheorem{acor}[athm]{Corollary}
\newtheorem{aprop}[athm]{Proposition}
\newtheorem{alemma}{Lemma}[athm]
\newtheorem{thm}{Theorem}[subsection]

\newtheorem{cor}[thm]{Corollary}

\newtheorem{lem}[thm]{Lemma} 
\newtheorem{prop}[thm]{Proposition}
 \theoremstyle{definition}
\newtheorem{conjecture}[thm]{Conjecture} \theoremstyle{definition}
\newtheorem{defn}[thm]{Definition} \theoremstyle{remark}
\newtheorem{rem}[thm]{\bf Remark}
\newtheorem{remark}[thm]{\bf Remark}
\newtheorem{para}[thm]{\bf} 
\newtheorem{apara}[athm]{\bf}

\numberwithin{equation}{subsection}

\theoremstyle{definition}

\begin{document}
\title[The Breuil--M\'ezard Conjecture]{The Breuil--M\'ezard Conjecture for Potentially Barsotti--Tate representations.}

\author{Toby Gee} \email{toby.gee@imperial.ac.uk} \address{Department of
  Mathematics, Imperial College London} \author{Mark Kisin} \email{kisin@math.harvard.edu} \address{Department of
  Mathematics, Harvard University}   \thanks{The first author was
  partially supported by a Marie Curie Career Integration Grant, and by an
  ERC Starting Grant. The second author was partially supported
  by NSF grant DMS-0701123.}  \subjclass[2000]{11F33.}
\begin{abstract}We prove the Breuil--M\'ezard conjecture for
  2-dimensional potentially Barsotti--Tate representations of the
  absolute Galois group $G_K$, $K$ a finite extension of $\Qp$, for
  any $p>2$ (up to the question of determining precise values for the
  multiplicities that occur). In the case that $K/\Qp$ is unramified,
  we also determine most of the multiplicities. 
  We then apply these results to the weight part of
  Serre's conjecture, proving a variety of results including the
  Buzzard--Diamond--Jarvis conjecture.
\end{abstract}
\maketitle
\tableofcontents

\section*{Overview.}\label{sec:overview}

The Breuil--M\'ezard conjecture (\cite{breuil-mezard}) predicts 
the Hilbert--Samuel multiplicity of the special fibre of a deformation 
ring of a mod $p$ local Galois representation. Its motivation is to
give  
a local explanation for the multiplicities seen in Hecke algebras and spaces of modular forms; 
in that sense it can be viewed as an avatar for the hoped for $p$-adic local Langlands correspondence.
For $\GL_2(\Qp),$ the conjecture was mostly proved 
in \cite{kisinfmc}, using global methods and the
$p$-adic local Langlands correspondence  and was used 
in that paper to deduce modularity lifting theorems.

The conjecture was originally formulated for 2-dimensional
representations of $G_{\Qp}$ (the absolute Galois group of $\Qp$) with
a restriction on the Hodge--Tate weights of the deformations under
consideration, but the formulation extends immediately to the case of
unrestricted regular Hodge--Tate weights. In fact, 
there is a natural generalisation of the Breuil--M\'ezard conjecture for
continuous representations 
$\rbar:G_K\to\GL_2(\Fpbar)$ for any finite extension $K/\Qp$
- see \cite{kisinICM}, and section \ref{subsec: deformation rings and types} below. 
\footnote{In fact, it is possible to formulate a
  generalisation for representations of arbitrary dimension; see
  Section~4 of~\cite{emertongeerefinedBM}.}

There is currently no known generalisation of the $p$-adic local
Langlands correspondence to $\GL_2(K)$, $K\ne \Qp$, and it is
accordingly not possible to use the local methods of \cite{kisinfmc}
to prove the conjecture in greater generality.  
The main
idea of the present paper is that one can \emph{use} the modularity
lifting theorems proved in \cite{kis04} and \cite{MR2280776} (by a
completely different method, unrelated to $p$-adic local Langlands) to
prove the Breuil--M\'ezard conjecture for all potentially
Barsotti--Tate deformation rings. As a byproduct of these arguments,
we are also able to prove the Buzzard--Diamond--Jarvis conjecture
(\cite{bdj}) on the weight part of Serre's conjecture for Hilbert
modular forms, as well as its  generalisations to arbitrary totally real fields conjectured
in~\cite{MR2430440} and~\cite{gee061}.

\section{Introduction.}\label{sec:intro}
Fix finite extensions $K/\Qp$ and $E/\Qp$, the latter (which
will be our coefficient field) assumed
sufficiently large. Let $E$ have ring of integers $\cO$, uniformiser
$\pi$, and residue field $\F$, let $k$ be the residue field of $K$,
and fix a continuous representation $\rbar:G_K\to\GL_2(\F)$. Given a
Hodge type $\lambda$ and an inertial type $\tau$ (see
Section~\ref{subsec: deformation rings and types} below for the
precise definitions of these notions, and of the other objects
recalled without definition in this introduction), there is a
universal lifting $\cO$-algebra $R_\rbar^{\square,\lambda,\tau}$ for
potentially crystalline lifts of $\rbar$ of Hodge type $\lambda$ and
Galois type $\tau$. The Breuil--M\'ezard conjecture predicts the
Hilbert--Samuel multiplicity $e(R_\rbar^{\square,\lambda,\tau}/\pi)$
in terms of the representation theory of $\GL_2(\cO_K)$. More precisely, a
result of Henniart attaches to $\tau$ a smooth, irreducible,
finite-dimensional $E$-representation $\sigma(\tau)$ of $\GL_2(\cO_K)$
via the local Langlands correspondence, and there is also an algebraic
representation $W_\lambda$ of $\GL_2(\cO_K)$ associated to
$\lambda$. Let $L_{\lambda,\tau}\subset W_\lambda\otimes\sigma(\tau)$
be a $\GL_2(\cO_K)$-invariant lattice; then the general shape of the
Breuil--M\'ezard conjecture is that for all $\lambda$, $\tau$ we
have \[e(R_\rbar^{\square,\lambda,\tau}/\pi)=\sum_{\sigma}n_{\lambda,\tau}(\sigma)\mu_\sigma(\rbar),\]
where $\sigma$ runs over the irreducible mod $p$ representations of
$\GL_2(k)$, $n_{\lambda,\tau}(\sigma)$ is the multiplicity of $\sigma$
as a Jordan--H\"older factor of $L_{\lambda,\tau}/\pi$, and the
$\mu_\sigma(\rbar)$ are non-negative integers, depending only on
$\rbar$ and $\sigma$ (and not on $\lambda$ or $\tau$). 

One can view the conjecture as giving infinitely many equations
(corresponding to the different possibilities for $\lambda$ and
$\tau$) in the finitely many unknowns $\mu_\sigma(\rbar)$, and it is
easy to see that if the conjecture is completely proved, then the
$\mu_\sigma(\rbar)$ are completely determined (in fact, they are
determined by the equations for $\tau$ the trivial representation and
$\lambda$ ``small'', and are zero unless $\sigma$ is a predicted Serre
weight for $\rbar$ in the sense of \cite{gee061}). 
Our main theorem (see Corollary~\ref{cor: main result for pot BT}) 
is the following result which establishes the conjecture in the
potentially Barsotti--Tate case ($\lambda=0$ in the terminology above).
\begin{ithm}\label{thm: intro version of main abstract result}Suppose
  that $p>2$. Then there are uniquely determined non-negative integers
  $\mu_\sigma(\rbar)$ such that for all inertial types $\tau$, we
  have \[e(R_\rbar^{\square,0,\tau}/\pi)=\sum_{\sigma}n_{0,\tau}(\sigma)\mu_\sigma(\rbar).\]Furthermore,
  the $\mu_\sigma(\rbar)$ enjoy the following properties.
  \begin{enumerate}
  \item 
    $\mu_\sigma(\rbar)\ne 0$ if and only if $\sigma$ is a predicted Serre weight
    for $\rbar$.  
  \item If $K/\Qp$ is unramified and $\sigma$ is regular
    , then
    $\mu_\sigma(\rbar)=e(R_\rbar^{\square,\sigma}/\pi)$, where
    $R_\rbar^{\square,\sigma}$ is the crystalline lifting ring of Hodge type
    determined by $\sigma$. If
    furthermore $\sigma$ is Fontaine--Laffaille regular
    , then
    $\mu_\sigma(\rbar)=1$ if $\sigma$ is a predicted Serre weight for
    $\rbar$, and is $0$ otherwise.
  \end{enumerate}

\end{ithm} (See the introduction to~\cite{GLS13} for a discussion of the history
of the various definitions of predicted Serre weights for $\rbar$ and of their
equivalence, and see section~\ref{subsec: local Serre weights} below for the
precise definitions we are using.) We are able to apply this result and the techniques that we
use to prove it to the problem of the weight part of Serre's
conjecture. For any $\rbar:G_K\to\GL_2(\F)$ as above, we define $\WBT(\rbar)$ 
to be the set of weights $\sigma$ such that $\mu_\sigma(\rbar)>0$. It follows
from Theorem \ref{thm: intro version of main abstract result}(1) that
$\WBT(\rbar)$ is precisely the set of weights predicted by the 
Buzzard--Diamond--Jarvis conjecture and its generalisations 
(\cite{bdj}, \cite{MR2430440}, \cite{gee061}). We then prove the
following result (see Corollary \ref{cor: main BDJ theorem}),

\begin{ithm}\label{thm: intro version of BDJ}
 Let $p>2$ be
  prime, let $F$ be a totally real field, and let
  $\rhobar:G_F\to\GL_2(\Fpbar)$ be a continuous representation. Assume
  that $\rhobar$ is modular, that $\rhobar|_{G_{F(\zeta_p)}}$ is
  irreducible, and if $p=5$ assume further that the projective image
  of $\rhobar|_{G_{F(\zeta_p)}}$ is not isomorphic to $A_5$.

  For each place $v|p$ of $F$ with residue field $k_v$, let $\sigma_v$
  be a Serre weight of $\GL_2(k_v)$. Then $\rhobar$ is modular of
  weight $\otimes_{v|p}\sigma_v$ if and only if
  $\sigma_v\in\WBT(\rhobar|_{G_{F_v}})$ for all $v$.
\end{ithm} In the case that $p$ is unramified in $F$
this proves the Buzzard--Diamond--Jarvis conjecture (\cite{bdj}) for
$\rhobar.$ More generally, by Theorem~\ref{thm: intro version of main abstract result}(1) (which
relies on the main result of~\cite{GLS13}), it proves the generalisations of the
Buzzard--Diamond--Jarvis conjecture to arbitrary totally real fields conjectured
in~\cite{MR2430440} and~\cite{gee061}. In particular, 
it shows that the set of weights for which
$\rhobar$ is modular depends only on the restrictions of $\rhobar$ to
decomposition groups at places dividing $p$, which was not previously
known. We remark that, in the case of indefinite quaternion algebras, 
another proof of the BDJ conjecture (which, like the proof given in this paper, relies on the version of the
conjecture for unitary groups proven in \cite{blggU2}, \cite{GLS12},
\cite{GLS13}) is given in \cite{Newton}. 

We emphasize that in the above Theorem, the definition of $\rhobar$ being modular of
some weight is in terms of quaternion algebras as in the Buzzard--Diamond--Jarvis conjecture  (see \cite{bdj} Defn.~2.1, Conj.~3.14 and \cite{geebdj}), 
rather than in terms of unitary groups as in \cite{blggU2}. The former statement is more subtle since the set of {\em local} Serre weights 
in the Buzzard--Diamond--Jarvis conjecture is attached to $\bar r$ by considering crystalline liftings while, for a quaternion algebra, 
one cannot lift every global Serre weight to characteristic zero.

We now describe some of the techniques of this paper in more detail.
In \S 2,3 we adapt the patching arguments of \cite{kisinfmc} to the case of 
a rank two unitary group over a totally real field $F^+.$ In the presence of a modularity 
lifting theorem, these yield a relationship between a patched Hecke module and 
a tensor product of the local deformation rings which appear in the Breuil-M\'ezard 
conjecture. In particular, the modularity lifting theorems proved in \cite{kis04} and 
\cite{MR2280776} imply that such a relationship holds in the (two dimensional) potentially 
Barsotti-Tate case. More generally, we show that such a relationship holds for representations which 
are ``potential diagonalizable'' in the sense of \cite{BLGGT}.

These arguments produce not a solution in non-negative integers
to the systems of equations that we seek, but rather a solution to a
product, over the primes $\mathfrak p|p$ of $F^+,$ of these systems of equations. 
We deduce that a solution for each individual system exists
by showing that such a solution is unique, and applying some linear algebra.\footnote{It does not seem possible to avoid dealing with a product 
of systems of equations by, for example, choosing $F^+$ so that $p$ is inert in $F^+$:
our methods require that we realize the local representation $\rbar$ globally. 
We do this via the potential automorphy techniques of \cite{BLGGT} and \cite{frankII} (see Appendix A), 
and these methods cannot ensure that $p$ is inert in $F^+.$ (In the case that
$\rbar$ is irreducible or decomposable, it is presumably possible to
avoid this by making use of CM forms, but they cannot handle the case
that $\rbar$ is reducible and indecomposable.)} 
It follows from the construction that $\mu_{\sigma}(\rbar) \neq 0$ if
and only if there are modular forms of weight $\sigma$ for the unitary
group used in the construction. Using this, together
with the results on Serre's conjecture for unitary groups proved in 
 \cite{blggU2}, \cite{GLS12}, \cite{GLS13}, we deduce the other
 claims in Theorem \ref{thm: intro version of main abstract result}.

In order to prove Theorem \ref{thm: intro version of BDJ} 
and the Buzzard--Diamond--Jarvis conjecture, we 
repeat these constructions in \S 4, in the setting of the cohomology of Shimura
curves associated to division algebras.  The uniqueness of the $\mu_\sigma(\rbar)$, 
implies that the multiplicities computed globally in this setting agree with those computed 
via unitary groups. On the other hand, as for unitary groups, by construction these multiplicities 
are non-zero if and only if there are modular forms of weight $\sigma$ for the quaternion algebra 
used in the construction. Theorem \ref{thm: intro version of BDJ} follows from this, and implies 
the BDJ conjecture and its generalisations when combined with Theorem 
\ref{thm: intro version of main abstract result}(1).

The argument comparing multiplicities for unitary groups and quaternion algebras 
seems to us to be analogous to the use of the trace formula to prove instances of functoriality for inner forms. 
We view the left hand side of the
equality \[e(R_\rbar^{\square,0,\tau}/\pi)=\sum_{\sigma}n_{0,\tau}(\sigma)\mu_\sigma(\rbar)\]
as being the ``geometric'' side, and the right hand side as the
``spectral'' side. Then the geometric side is manifestly the same in
the unitary group or Shimura curve settings, from which we deduce
that the spectral sides are also the same, and thus transfer the
proof of the Buzzard--Diamond--Jarvis conjecture from the unitary
group context to the original setting of \cite{bdj}. As mentioned above, 
the proof of Theorem \ref{thm: intro version of main abstract result}(1)  
uses the results of \cite{blggU2}, \cite{GLS12}, \cite{GLS13} which, in turn, rely crucially on the 
fact that there is no parity restriction on the weights of the modular forms on unitary groups.
Thus, it does not seem possible to prove Theorem \ref{thm: intro version of main abstract result}(1) 
and its consequences for the BDJ conjecture by working only with quaternion algebras. 

We would like to thank Frank Calegari, Fred Diamond, Matthew Emerton, Guy Henniart,
Florian Herzig and David Savitt for helpful conversations. We would
also like to thank Matthew Emerton and Florian Herzig for their helpful comments on an
earlier draft of this paper. T.G. would
like to thank the mathematics department of Northwestern University
for its hospitality in the final stages of this project.

\subsection{Notation} Fix an algebraic closure $\Qbar$ of $\Q$, and an
algebraic closure $\Qpbar$ of $\Qp$ for each prime $p$. Fix also an
embedding $\Qbar\into\Qpbar$ for each $p$. If $M$ is a finite
extension of $\Q$ or $\Qp$, we let $G_M$ denote its absolute Galois
group. If $M$ is a finite extension of $\bb{Q}_p$ for some $p$, we write $I_M$ for the
inertia subgroup of $G_M$. If $F$ is a number field and $v$ is a finite place of $F$ then we let $\Frob_v$
denote a geometric Frobenius element of $G_{F_v}$.

If $R$ is a local ring we write
$\mf{m}_{R}$ for the maximal ideal of $R$.
We write all matrix transposes on the left; so ${}^tA$ is the transpose
of $A$. 

Let $\varepsilon$ denote the $p$-adic cyclotomic character, and
$\varepsilonbar=\omega$ the mod $p$ cyclotomic character. 

If $K$ is a finite extension
of $\Qp$ for some $p$, we let $\rec$ be the local Langlands
correspondence of \cite{ht},
so that if $\pi$ is an irreducible $\Qpbar$-representation of $\GL_n(K)$, then
$\rec(\pi)$ is a complex
Weil--Deligne representation of the Weil group $W_K$. Choose an isomorphism
$\imath:\Qpbar\xrightarrow{\sim}\C$, and set
$r_p(\pi):=\imath^{-1}\circ\rec\circ\imath(\pi\otimes|\det|^{(1-n)/2})$; this is independent of
the choice of $\imath$. We let $\Art_K:K^\times\to W_K^{ab}$
be the isomorphism provided by local class field theory, which we
normalise so that uniformizers correspond to geometric Frobenius
elements. 
We will sometimes
identify characters of $I_K$ and of $\cO_K^\times$ via $\Art_K$ without comment.
If $F$ is a number field, we write $\Art_F$ for the global Artin map, normalized to be
compatible with the $\Art_{F_v}$.

\section{Potentially crystalline deformation rings.}
\label{subsec: deformation  rings and types}  
In this section we recall the formulation of the Breuil--M\'ezard conjecture, or 
rather its generalisation to finite extensions of $\Q_p$ (cf.~\cite{kisinICM}). 

We fix a finite extension $K/\Q_p$ with ring of integers $\O_K$ and
residue field $k$. 
Let $E \subset \Qpbar$ be a finite extension of $\Qp$
with ring of integers $\cO$ and residue field $\F$.  In particular, we
may regard $\F$ as a subfield of $\Fpbar,$ the residue field of
$\Qpbar.$ We assume throughout the paper that $E$ is sufficiently
large; in particular, we assume that $E$ contains the image of every
embedding $K\into\Qpbar$, and that various $\Qpbar$-representations
$\tau$, $\sigma(\tau)$ that we consider are in fact defined over $E$.

\subsection{The Breuil--M\'ezard conjecture}

Let $B$ be a finite local $E$-algebra, and $V_B$ be a finite free $B$-module, 
with a continuous potentially semistable action of $G_K$. Then  
$$ D_{\dR}(V_B) = (B_{\dR}\otimes_{\Q_p}V)^{G_K}$$
is a filtered $B\otimes_{\Q_p}K$-module which is free of rank $\rk_BV_B,$ 
and whose associated graded is projective over $B\otimes_{\Q_p}K.$ 
For an embedding $\varsigma: K \hookrightarrow E$ we denote by 
$\HT_{\varsigma}(V_B)$ the multiset in which the integer $i$ appears 
with multiplicity $\rk_B \gr^i(D_{\dR}(V_B)\otimes_{B\otimes_{\Q_p}K,\varsigma}B).$ 
We call the elements of $\HT_{\varsigma}(V_B)$ the Hodge-Tate numbers of $V_B$ with respect to $\varsigma.$
Thus, for example $\HT_\varsigma(\varepsilon)=\{ -1\}$.

Let $\Z^2_+$ denote the set of pairs $(\lambda_1,\lambda_2)$ of
integers with $\lambda_1\ge \lambda_2$. Fix $\lambda\in(\Z^2_+)^{\Hom_{\Qp}(K,E)},$ 
and suppose $\rk_B V_B = 2.$  
We say that $V_B$ has \emph{Hodge type} $\lambda$ if for 
each $\varsigma :K\into E$, the Hodge--Tate weights of $V_B$ with respect to $\varsigma$
are $\lambda_{\varsigma,1}+1$ and $\lambda_{\varsigma,2}.$
If $V_B$ is potentially crystalline and has Hodge type $0$, we say that $V_B$ is \emph{potentially Barsotti--Tate}. 
\footnote{Note that this is a slight abuse of terminology; however, we will have no reason 
to deal with representations with non-regular Hodge-Tate weights, and
so we exclude them from 
consideration.} (Note that it is more usual in the literature to say that $V_B$
is potentially Barsotti--Tate if it is potentially crystalline, and
$V_B^\vee$ has Hodge type $0$; in particular, all of our definitions
are dual to those of~\cite{breuil-mezard}, but of course our results
can be translated into their setting by taking duals.)

 An {\em inertial type} is a representation
$\tau:I_K\to\GL_2(E)$ with open kernel, with the property that
(possibly after replacing $E$ with a finite extension) $\tau$ may be
extended to a representation of the Weil group $W_K$. In particular, it is
semisimple and factors through a finite quotient of $I_K$.
We say that $V_B$ is of \emph{Galois type} $\tau$ if the traces of
elements of $I_K$ acting on $\Dpst(V_B)$ and $\tau$ are equal.

Let $\rbar: G_K\to\GL_2(\F)$ be a continuous representation.  Let
$R^\square_\rbar$ be the universal $\cO$-lifting ring of $\rbar$, so
that $R^\square_\rbar$ pro-represents the functor which assigns to a
local Artin $\cO$-algebra $R$ with residue field $\F$ the set of liftings of
$\rbar$ to a representation $G_K\to\GL_2(R).$

The following is a special case of one of the main results of \cite{kisindefrings}.
\begin{prop}\label{defrings}
  There is a unique (possibly zero) $p$-torsion free
  quotient $R_\rbar^{\square,\lambda,\tau}$ of $R^\square_\rbar,$ such that for any finite
  local $E$-algebra $B,$ and any $E$-homomorphism $x:R^\square_\rbar\to B$, the
  $B$-representation of $G_K$ induced by $x$ is potentially
  crystalline of Galois type $\tau$ and Hodge type $\lambda,$ if and
  only if $x$ factors through $R_\rbar^{\square,\lambda,\tau}$.

Moreover, $\Spec R_\rbar^{\square,\lambda,\tau}[1/p]$
is formally smooth and everywhere of dimension $4+[K:\Qp]$.
\end{prop}
In the case that $\tau$ is the trivial representation, we will drop it
from the notation, and write $R_\rbar^{\square,\lambda}$ for
$R^{\square,\lambda,\tau}$.  
We will need the following simple lemma later.
\begin{lem}\label{lem: crystalline def ring unchanged by unramified twist}
  Let $\psibar:G_K\to\F^\times$ be an unramified character. Then the
  $\cO$-algebras  $R_\rbar^{\square,\lambda,\tau}$ and
  $R_{\rbar\otimes(\psibar\circ\det)}^{\square,\lambda,\tau}$ are isomorphic.
\end{lem}
\begin{proof}
  Let $\psi$ denote the Teichm\"uller lift of $\psibar$; since this
  character is unramified, it is crystalline with all Hodge--Tate
  weights equal to $0$. If $r^\univ:G_K\to \GL_2(R_\rbar^{\square,\lambda,\tau})$ is the
  universal lift of $\rbar$, then $r^\univ\otimes(\psi\circ\det):G_K\to
  \GL_2(R_{\rbar}^{\square,\lambda,\tau})$ is the
  universal lift of $\rbar\otimes(\psibar\circ\det)$, as required.
\end{proof}

Let $\tau:I_K\to\GL_2(\Qpbar)$ be an inertial type, as above. 
We have the following theorem of Henniart (see the appendix to
\cite{breuil-mezard}).
\begin{thm}\label{thm: Henniart existence of types}
  There is a finite-dimensional irreducible
  $\Qpbar$-representation $\sigma(\tau)$ of $\GL_2(\cO_K),$ such that
  for any 2-dimensional Frobenius semi-simple representation
  $\tilde{\tau}$ of $\WD_K$,
  $(r_p^{-1}(\tilde{\tau})^\vee)|_{\GL_2(\cO_K)}$
  contains $\sigma(\tau)$ if and only if $\tilde{\tau}|_{I_K}\sim\tau$
  and $N=0$ on $\tilde{\tau}$. Furthermore, for all $\tilde{\tau}$ we have
  \[\dim_{\Qpbar}\Hom_{\GL_2(\cO_K)}(\sigma(\tau),r_p^{-1}(\tilde{\tau})^\vee)\le
  1.\]If $|k|>2$ then $\sigma(\tau)$ is uniquely determined.
\end{thm}

\begin{para}\label{sec: BM conj}
For $\lambda \in (\Z^2_+)^{\Hom_{\Qp}(K,E)}$ we set 
$$ W_{\lambda} = \otimes_{\varsigma} \left(\left(\det{}^{\lambda_{\varsigma,2} }\otimes
\Sym^{\lambda_{\varsigma,1} - \lambda_{\varsigma,2}}
\O_K^2\right)\otimes_{\O_K,\varsigma}\O\right)$$ where $\varsigma$ runs over the embeddings
$K\into E$.

Now assume that $E$ is sufficiently large that $\sigma(\tau)$ is defined over $E$. Let $\pi$ be a
uniformiser of $\cO$. Fix a $\GL_2(\cO_K)$-stable $\cO$-lattice
$L_{\lambda,\tau}\subset W_\lambda\otimes_{\O} \sigma(\tau)$. If $A$ is a Noetherian local ring, we let $e(A)$ denote the
Hilbert--Samuel multiplicity of $A$.
\end{para}

\begin{conjecture}\label{BM conj}(Breuil--M\'ezard)
  There exist non-negative integers $\mu_\sigma(\rbar)$ for each
 irreducible mod $p$ representation $\sigma$ of $\GL_2(k),$ such that for any $\tau$,
  $\lambda$ as above, we have

\[e(R_\rbar^{\square,\lambda,\tau}/\pi)=\sum_\sigma n(\sigma)\mu_\sigma(\rbar),\] 
where 
$\sigma$ runs over isomorphism classes of irreducible mod $p$ representation of 
$\GL_2(k),$ and $n(\sigma) = n_{\lambda,\tau}(\sigma)$ 
is the multiplicity of $\sigma$ as a Jordan-H\"older factor of $L_{\lambda,\tau}\otimes_\cO\F,$ so that 
\[(L_{\lambda,\tau}\otimes_\cO\F)^{ss}\isoto\oplus_\swsigma
  \swsigma^{n(\swsigma)}.\]
\end{conjecture}
\begin{rem}
  When $|k|=2$, so that the type $\sigma(\tau)$ is not necessarily unique, it
  follows from Proposition~4.2 of~\cite{breuildiamond} that the quantities
  $n(\sigma)$ are independent of the choice of $\sigma(\tau)$.
\end{rem}

\subsection{Local Serre weights}\label{subsec: local Serre weights}
The following definitions will be useful to us later, in order to give
more explicit information about the $\mu_\sigma(\rbar)$ in some cases.

\begin{para}\label{defnofserreweights}
By a  \emph{(local) Serre weight} we mean an absolutely irreducible representation of $\GL_2(k)$ 
on an $\F$-vector space, up to isomorphism. Let $(\Z^2_+)^{\Hom(k,\F)}_{\SW} \subset (\Z^2_+)^{\Hom(k,\F)}$ be the 
subset consisting of elements $a$ such that 
\[p-1\ge a_{\varsigma,1}-a_{\varsigma,2} \] for each $\varsigma\in\Hom(k,\F).$ 
Then 
$$ \swsigma_a = \otimes_{\varsigma} \det{}^{a_{\varsigma,2} }\otimes \Sym^{a_{\varsigma,1} - a_{\varsigma,2}} k^2\otimes_{k,\varsigma}\F$$ 
where $\varsigma$ runs over the embeddings $k \hookrightarrow \F,$ is a Serre weight, and every Serre weight is of this form. 
We say $a, a' \in(\Z^2_+)^{\Hom(k,\F)}_{\SW}$ are \emph{equivalent} if $\swsigma_a \cong \swsigma_{a'}.$ 
If $\swsigma_a \cong \swsigma_{a'}$ then  $a_{\varsigma,1} - a_{\varsigma,2} = a'_{\varsigma,1} - a'_{\varsigma,2}$ 
for all $\varsigma.$

\end{para}

\begin{para}\label{weightremarks}
 We have a natural surjection $\Hom_{\Qp}(K,E)\onto\Hom(k,\F)$, which is a
bijection if and only if $K/\Qp$ is unramified. Suppose that $K/\Qp$
has ramification degree $e$. For each $\varsigma\in\Hom(k,\F)$, we choose
an element $\tau_{\varsigma,1}$ in the preimage of $\varsigma$, and denote
the remaining elements of the preimage by
$\tau_{\varsigma,2},\dots,\tau_{\varsigma,e}$. Now, given
$a\in(\Z^2_+)^{\Hom(k,\F)}$, we define
$\lambda_a\in(\Z^2_+)^{\Hom_{\Qp}(K,E)}$ as follows: For $i=1,2,$
$\lambda_{a_{\tau_{\varsigma,1},i}}=a_{\varsigma,i}$, and
$\lambda_{a_{\tau_{\varsigma,j},i}}=0$ if $j>1$. When $K/\Qp$ is
unramified, we will sometimes write $a$ for $\lambda_a$.  

By definition, we have $W_{\lambda_a}\otimes_{\cO}\F\cong \swsigma_a$, and we
write $R^{\square,a}$ for $R^{\square,\lambda_a}$. Note that in the
case that $K/\Qp$ is ramified this definition depends on the choice of
the places $\tau_{\varsigma,1}$, as does Definition \ref{defn: lift of
  Hodge type a} below (at least \emph{a priori}). However, our only use of
this definition in the ramified case will be to prove that in certain
cases $R^{\square,\lambda_a}$ is nonzero (cf.\ Lemma \ref{lem:
  multiplicity positive implies Serre weight and thus crystalline
  lift} below), and in these cases our argument will in fact show that
this holds for any choice of the $\tau_{\varsigma,1}$. 

If $a,a' \in(\Z^2_+)^{\Hom(k,\F)}_{\SW}$ with $\swsigma_a\cong
\swsigma_{a'},$ then there is a crystalline character $\psi_{a,a'}$ of $G_K$ with trivial reduction
mod $p$ and Hodge--Tate weights given by
$\HT_{\tau_{\varsigma,1}}(\psi_{a,a'})=a_{\varsigma,2}-a'_{\varsigma,2}$, and
$\HT_{\tau_{\varsigma,j}}(\psi_{a,a'})=0$ if $j>1$. The corresponding
universal deformation to $R^{\square,a}$ is obtained from that to
$R^{\square,a'}$ by twisting by $\psi_{a,a'},$ which induces an isomorphism $R^{\square,a} \cong R^{\square,a'}.$

If $\sigma$ is an irreducible
$\F$-representation of $\GL_2(k)$ and $\sigma\cong \swsigma_a$, we will write
$R^{\square,\sigma}$ for $R^{\square,a}$. The following definitions
will be needed in order to state our main results.
\end{para}
\begin{defn}
  \label{defn: lift of Hodge type a}We say that $\rbar$ has a
  \emph{lift of Hodge type} $\sigma$ if (with notation as above)
  $R^{\square,\sigma}\ne 0$.
\end{defn}
\begin{defn}
  \label{defn:regular weight local}Suppose that $K/\Qp$ is
  unramified. We say that a Serre weight $\sigma$ is
  \emph{regular} if $\sigma \cong \sigma_a$ for some $a \in (\Z^2_+)^{\Hom(k,\F)},$ 
with  $p-2\ge a_{\varsigma,1}-a_{\varsigma,2}$ for each $\varsigma\in\Hom(k,\F).$
We say that it is \emph{Fontaine--Laffaille regular} if for each
  $\varsigma\in\Hom(k,\F)$ we have $p-3\ge a_{\varsigma,1}-a_{\varsigma,2}.$
The remarks above show that these conditions depend only on $\sigma$ 
and not on the choice of $a.$

\end{defn}
\begin{para}\label{predicted} 
We remark that this is a significantly less restrictive definition
than the definition of regular weights in \cite{geebdj}. 

In order to make use of the results of \cite{GLS11}, \cite{GLS12}, \cite{GLS13} we
need to recall the notion of a predicted Serre weight. Beginning with
the seminal work of \cite{bdj}, various definitions have been
formulated of conjectural sets of Serre weights for two-dimensional
global mod $p$ representations (cf.\ \cite{MR2430440},
\cite{gee061}). These sets are defined purely locally, and the
relationship between the different local definitions is important in
proving the weight part of Serre's conjecture; see Section 4 of
\cite{blggU2} for a thorough discussion. In particular, given a
continuous representation $\rbar:G_K\to\GL_2(\F)$, sets of weights
$\Wconj(\rbar)$ and $\Wcris(\rbar)$ are defined in \emph{loc. cit.}
The set $\Wconj(\rbar)$ is defined by an explicit recipe that
generalises those of \cite{bdj}, \cite{MR2430440} and \cite{gee061},
whereas the set $\Wcris(\rbar)$ is the set of weights $\sigma$ for which
$\rbar$ has a crystalline lift of Hodge type $\sigma$. It is shown in
\cite{blggU2} that $\Wexplicit(\rbar)\subset\Wcris(\rbar)$, and
conjectured that equality holds; this has now been proved \cite{GLS13}. We make the following
definition.
\end{para}
\begin{defn}
  \label{defn:predicted Serre weight}If $\rbar:G_K\to\GL_2(\F)$ is a
  continuous representation, we say that $\sigma$ is a \emph{predicted
    Serre weight} for $\rbar$ if $\sigma \in\Wconj(\rbar)$,
  where $\Wconj(\rbar)$ is the set of Serre weights defined in Section
  4 of \cite{blggU2}.
\end{defn}

\section{Algebraic automorphic forms and Galois representations}
\label{sec: Unitary groups}
\subsection{Unitary groups and algebraic automorphic forms}
Let $p>2$ be a prime, and let $F$ be an imaginary CM field with
maximal totally real field subfield $F^+$. We assume throughout this
paper that:
\begin{itemize}
\item $F/F^+$ is unramified at all finite places.
\item Every place $v|p$ of $F^+$ splits in $F$.
\end{itemize}
By class field theory, the set of places $v$ of $F^+$ such that $-1$ is not in the image of 
the local norm map $(F\otimes_{F^+}F^+_v)^\times \rightarrow
F^{+,\times}_v$ has even cardinality. Since we are assuming that 
$F/F^+$ is unramified at all finite places, this implies that $[F^+:\Q]$ is even.

\begin{para}\label{para: models} We now define the unitary groups with which we shall work. It will be convenient 
to define these as groups over $\O_{F^+}.$

Let $c \in \Gal(F/F^+)$ be the complex conjugation. The map $g \mapsto ({}^tg^c)^{-1}$ is an involution of 
$\GL_{2/\O_F}$ which covers the action of $c$ on $\O_F.$ Since $\O_F$ is unramified over $\O_{F^+},$ 
it follows by \'etale descent that there is a reductive group $G$ over $\O_{F^+}$ 
such that for any $\cO_{F^+}$-algebra $R,$ one has 
\[G(R)=\{g\in \GL_2(\O_F\otimes_{\cO_{F^+}} R) :{}^tg^c g=1\}.\] Thus $G$ is a
unitary group which is definite at infinite places, and which is quasi-split at
finite places, because $-1$ is in the image of the local norm map at finite
places. ($G$ is automatically quasi-split if $n$ is odd, and if $n$ is even,
this is equivalent to the standard Hermitian form being a sum of hyperbolic
planes, by (for example) the discussion of groups of type ${}^2A_n$ on p.\ 55
of~\cite{MR0224710}. If $N(a)=-1$, then the equation $z_1z_1^c + z_2z_2^c = 0$
has solution vectors of the form $(z,az)$ and $(z, a^cz)$, so the condition is
satisfied for $n$ even.)

By construction, $G$ is equipped with an isomorphism 
 $\iota:G_{/\O_F} \isoto \GL_{2/\O_F}$ such that  $\iota\circ c \circ \iota^{-1}(g) = ({}^tg^c)^{-1}.$
If $v$ is a place of $F^+$ which splits as $ww^c$ over $F$, then $\iota$ induces 
isomorphisms
\[\iota_w:G(\cO_{F_v^+})\isoto \GL_2(\cO_{F_w}),\ \iota_{w^c}:G(\cO_{F_v^+})\isoto \GL_2(\cO_{F_{w^c}}) \]
which satisfy
$\iota_w\circ\iota^{-1}_{w^c}(g)=({}^tg^c)^{-1}$.  These extend to
isomorphisms $\iota_w:G(F^+_v)\isoto\GL_2(F_w)$ and $\iota_{w^c}:G(F^+_v)\isoto\GL_2(F_{w^c})$ with the same
properties.

\end{para}

\begin{para}\label{sec: coefficients}
  Continue to let $E/\Qp$ be a sufficiently large extension with ring
  of integers $\cO$ and residue field $\F$. Let $S_p$ denote the set
  of places of $F^+$ lying over $p$, and for each $v\in S_p$ fix a
  place $\tv$ of $F$ lying over $v$. Let $\tilde{S}_p$ denote the set
  of places $\tilde{v}$ for $v\in S_p$. Write $F^+_p=F^+\otimes_\Q\Qp$,
  and $\cO_{F^+_p}$ for the $p$-adic completion of $\cO_{F^+}$.

Let $W$ be an $\cO$-module with an action of $G(\cO_{F^+_p})$, and let
$U\subset G(\A_{F^+}^\infty)$ be a compact open subgroup with the
property that for each $u\in U$, if $u_p$ denotes the projection of
$u$ to $G(F_p^+)$, then $u_p\in G(\cO_{F^+_p})$.

Let $S(U,W)$ denote
the space of algebraic modular forms on $G$ of level $U$ and weight
$W$, i.e. the space of functions \[f:G(F^+)\backslash
G(\A_{F^+}^\infty)\to W\] with $f(gu)=u_p^{-1}f(g)$ for all $u\in U$.

For any compact open subgroup $U$  of $G(\A_{F^+}^\infty)$ as above, we may write
$G(\A_{F^+}^\infty)=\coprod_i G(F^+)t_i U$ for some finite set
$\{t_i\}$. Then there is an isomorphism \[S(U,W)\isoto\oplus_i W^{U\cap
  t_i^{-1}G(F^+)t_i}\]given by $f\mapsto (f(t_i))_i$. 
We say that
$U$ is \emph{sufficiently small} if for some finite place $v$ of $F^+$
the projection of $U$ to $G(F^+_v)$ contains no element of finite
order other than the identity. Suppose that $U$ is sufficiently
small. Then for each $i$ as above we have
$U\cap t_i^{-1}G(F^+)t_i=\{1\}$, so we see
that for any $\cO$-algebra $A$ and $\cO$-module $W$, we
have \[S(U,W\otimes_\cO A)\cong S(U,W)\otimes_\cO A.\]
\end{para}

\begin{para}\label{para: weights}
  Let $\tI_p$ denote the set of embeddings $F\into E$ giving rise to a
  place in $\tS_p$. For any $v \in S_p$, let $\tI_\tv$ denote the set
  of elements of $\tI_p$ lying over $\tv$. Note that $|\tI_\tv|=[F^+_v:\Qp]=[F_\tv:\Qp]$. Let $\Z^2_+$ be as in
  Section \ref{subsec: deformation rings and types}.  For any
  $\lambda\in(\Z^2_+)^{\tI_\tv}$, let $W_\lambda$ be the $E$-vector
  space with an action of $\GL_2(\cO_{F_\tv})$ given
by 
\[W_\lambda:=\otimes_{\varsigma\in\tI_\tv}\det{}^{\lambda_{\varsigma,2}}\otimes
\Sym^{\lambda_{\varsigma,1}-\lambda_{\varsigma,2}}F_\tv^2\otimes_{F_\tv,\varsigma}E.\]
We give this an action of $G(\cO_{F^+_v})$ via $\iota_\tv$. 

For any $\lambda\in(\Z^2_+)^{\tI_p}$ and $v \in S_p$, let $\lambda_v
\in (\Z^2_+)^{\tI_\tv}$ denote the tuple of pairs in $\lambda$ indexed
by $\tI_{\tv},$ and let $W_\lambda$ be the $E$-vector space with an
action of $G(\cO_{F^+_p})$ given by \[W_\lambda:=\otimes_{v \in S_p}
W_{\lambda_v}.\]

For each $v \in S_p$, let $\tau_v$ be an inertial type for $G_{F_v^+}$, so
that (since $E$ is assumed sufficiently large) there is an absolutely
irreducible $E$-representation $\sigma(\tau_v)$ of
$\GL_2(\cO_{F_{\tilde v}})$ 
associated to $\tau_v$ by Theorem \ref{thm: Henniart existence of
  types}.  Write $\sigma(\tau)$ for the tensor product of the
$\sigma(\tau_v)$, regarded as a representation of $G(\cO_{F^+_p})$ by
letting $G(\cO_{F^+_p})$ act on $\sigma(\tau_v)$ via $\iota_\tv$.  Fix a
$G(\cO_{F^+_p})$-stable $\O$-lattice $L_{\lambda,\tau} \subset
W_{\lambda}\otimes_\cO\sigma(\tau),$ and for any $\cO$-module $A$,
write
$$ S_{\lambda,\tau}(U,A):=S(U,L_{\lambda,\tau}\otimes_\cO A).$$
\end{para}

\subsection{Hecke algebras and Galois representations}\label{sec: Galois repns}

\begin{defn}\label{defn: good subgroup} We say that a compact open subgroup
of $G(\A_{F^+}^\infty)$ is \emph{good} if $U=\prod_vU_v$ with $U_v$ a
compact open subgroup of $G(F^+_v)$ such that:
\begin{itemize}
\item $U_v\subset G(\bigO_{F^+_v})$ for all $v$ which split in $F$;
   \item $U_v=G(\bigO_{F^+_v})$ if $v|p$ or $v$ is inert in $F.$
\end{itemize}
\end{defn}

\begin{para} 
Let $U$ be a good compact open subgroup of $G(\A_{F^+}^\infty)$. Let $T$ be a finite set of finite places of $F^+$ which split in $F$,
containing $S_p$ and all the places $v$ which split in $F$ for which
$U_v\neq G(\bigO_{F^+_v})$. We let $\mathbb{T}^{T,\univ}$ be the
commutative $\bigO$-polynomial algebra generated by formal variables
$T_w^{(j)}$ for $j = 1,2$, and $w$ a place of $F$ lying over a
place $v$ of $F^+$ which splits in $F$ and is not contained in $T$.
For any $\lambda\in (\Z^2_+)^{\tI_p}$, the algebra
$\mathbb{T}^{T,\univ}$ acts on $S_{\lambda,\tau}(U,\cO)$ via the
 Hecke operators
  \[ T_{w}^{(j)}:=  \iota_{w}^{-1} \left[ GL_2(\mc{O}_{F_w}) \left( \begin{matrix}
      \varpi_{w}1_j & 0 \cr 0 & 1_{2-j} \end{matrix} \right)
GL_2(\mc{O}_{F_w}) \right] 
\] for $w\not \in T$ and $\varpi_w$ a uniformiser in
$\mc{O}_{F_w}$.

We denote by $\mathbb{T}^T_{\lambda,\tau}(U,\cO)$ the image of
$T^{T,\univ}$ in $\End_\cO(S_{\lambda,\tau}(U,\cO))$.
\end{para}

\begin{para}\label{para: defn G2}

Let $\mf{m}$
be a maximal ideal of $\mathbb{T}^{T,\univ}$ with residue field $\F.$ 
We say that $\mf{m}$ is 
\emph{automorphic} if $S_{\lambda,\tau}(U,\cO)_\m\ne 0$ for some
$(\lambda,\tau)$ as above. If $\rbar:G_F\to\GL_2(\F)$ is an absolutely
irreducible continuous representation, we say that $\rbar$ is
\emph{automorphic} if there are $U$, $T$ as above and an automorphic maximal ideal $\m$ of
$\mathbb{T}^{T,\univ}$ such that for all places $v\notin T$ of $F^+$
which split as $v=ww^c$ in $F$, $\rbar|_{G_{F_w}}$ is unramified, and $\rbar(\Frob_w)$ has characteristic
polynomial equal to the image of
$X^2-T^{(1)}_wX+(\mathbf{N}w)T^{(2)}_w$ in $\F[X]$. Note that if
$\rbar$ is automorphic, it is necessarily the case that
$\rbar^c\cong\rbar^\vee\varepsilonbar^{-1}$, because we have
$T_{w^c}^{(j)}=(T_w^{(2)})^{-1}T_w^{(2-j)}$ in $\mathbb{T}^T_{\lambda,\tau}(U,\cO)$.

Let $\cG_2$ be the group scheme over $\Z$ defined to be the semidirect
product of $\GL_2\times\GL_1$ by the group $\{1,j\}$, which acts on
$\GL_2\times\GL_1$ by \[j(g,\mu)j^{-1}=(\mu{}^tg^{-1},\mu).\]We have a
homomorphism $\nu:\cG_2 \to \GL_1,$ sending $(g,\mu)$ to $\mu$ and $j$ to $-1$.

Assume that $\rbar$ is absolutely irreducible and automorphic with
corresponding maximal ideal $\m$. By Lemma 2.1.4 of \cite{cht} and the
main result of \cite{belchen} we can and do extend $\rbar$ to a
representation $\rhobar:G_{F^+}\to\cG_2(\F)$ with
$\nu\circ\rhobar=\varepsilonbar^{-1}$ and
$\rhobar|_{G_F}=(\rbar,\varepsilonbar^{-1})$. 
By Lemma 2.1.4 of
\cite{cht}, the $\F^\times$-conjugacy classes of such extensions are
a torsor under $\F^\times/(\F^\times)^2.$ In particular, any two extensions $\rhobar$ 
become conjugate if we replace $\F$ by a quadratic extension. 
We fix a choice of $\rhobar$ from now on.

In the rest of the paper we will make a number of arguments that will
be vacuous unless $S_{\lambda,\tau}(U,\cO)_\m\ne 0$ for the specific
$(\lambda,\tau)$ at hand, but for technical reasons we do not assume
this. Let $G_{F^+,T}:=\Gal(F(T)/F^+)$,
  $G_{F,T}:=\Gal(F(T)/F)$, where $F(T)$ is the maximal extension of
  $F$ unramified outside of places lying over $T$.
\end{para}

\begin{thm}\label{existence of repns}For any $(\lambda,\tau)$ there is a unique continuous
lift \[\rho_\mf{m}:G_{F^+,T}\to\cG_2(\mathbb{T}^T_{\lambda,\tau}(U,\cO)_\mf{m})\]of
$\rhobar$, which satisfies
\begin{enumerate}
\item $\rho_{\mf{m}}^{-1}((\GL_2\times\GL_1)(\mathbb{T}^T_{\lambda,\tau}(U,\cO)_\mf{m}))=G_{F,T}$.
\item $\nu\circ\rho_\mf{m}=\varepsilon^{-1}$.
\item\label{item:char poly of universal Hecke deformation} $\rho_\mf{m}$ is unramified outside $T$. If $v\notin T$ splits
  as $ww^c$ in $F$ then $\rho_\mf{m}(\Frob_w)$ has characteristic
  polynomial \[X^2-T_w^{(1)}X+(\mathbf{N}w)T_w^{(2)}.\]
\item \label{item:local at l=p behaviour of universal Hecke
    deformation} For each place $v\in S_p$, and each homomorphism $x:\mathbb{T}^T_{\lambda,\tau}(U,\cO)_\mf{m}\to\Qpbar$,
  $x\circ\rho_\mf{m}|_{G_{F_\tv}}$ is potentially crystalline of Hodge
  type $\lambda_v$ and Galois type $\tau_v$.

\end{enumerate}
\end{thm}
\begin{proof} This may be proved in the same way as Proposition 3.4.4
  of \cite{cht}, making use of Corollaire 5.3 of \cite{labesse},
  Theorem 1.1 of \cite{BLGGT11}, and Theorem \ref{thm: Henniart
    existence of types} above. (More specifically, Corollaire 5.3 of
  \cite{labesse} is used in order to transfer our automorphic forms to
  $\GL_n$, in place of the arguments made in the proof of Proposition
  3.3.2 of~\cite{cht}. The argument of Proposition 3.4.4
  of \cite{cht} then goes through unchanged, except that we have to
  check property (4) above; but by Theorem \ref{thm: Henniart
    existence of types}, this is a consequence of local-global 
compatibility at places dividing $p$, which is a special case of Theorem 1.1 of \cite{BLGGT11}.)

\end{proof}

\subsection{Global Serre weights}\label{subsec: globalserreweights}
 A \emph{global Serre weight} (for $G$) is an absolutely irreducible mod $p$ representation 
of $G(\cO_{F^+_p})$ considered up to equivalence.
For $v \in S_p$ denote by $k_v$ the residue field of $v.$
Let $a = (a_v)_{v \in S_p},$ where $a_v \in (\Z^2_+)^{\Hom(k_v,\F)}_{\SW}.$ 
We set 
\[\swsigma_a=\otimes_\F \swsigma_{a_v},\] where for $v \in S_p,$ 
$\swsigma_{a_v}$ is the representation of
$\GL_2(k_v) = \GL_2(k_\tv)$ defined in section~\ref{subsec: deformation rings and
  types}. We let $G(\cO_{F^+_p})$ act on $\swsigma_{a_v}$ by the composite
of $\iota_\tv,$ and reduction modulo $p.$ This makes $\swsigma_a$ an
irreducible $\F$-representation of $G(\cO_{F^+_p}),$ and any 
irreducible $\F$-representation of $G(\cO_{F^+_p})$ is equivalent to 
$\sigma_a$ for some $a.$

We say that two such tuples $a = (a_v)_{v \in S_p}, a' = (a'_v)_{v \in S_p},$ are \emph{equivalent} if 
$\swsigma_a\cong \swsigma_{a'}$; this implies that
$a_{v,\varsigma,1}-a_{v,\varsigma,2}=a'_{v,\varsigma,1}-a'_{v,\varsigma,2}$
for each $v \in S_p,$ and $\varsigma\in \Hom(k_v,\F).$

For $v \in S_p,$ $\tI_{\tv}$ naturally surjects onto $\Hom(k_v, \F).$ 
Fixing once and for all a splitting of each of these surjections, we obtain, 
as in section~\ref{subsec: deformation rings and types},
a Hodge type $\lambda_{a_v} \in  (\Z^2_+)^{\tI_{\tv}},$ and hence an element 
$\lambda_a \in   (\Z^2_+)^{\tI_p}.$

\section{The patching argument}\label{sec: patching}

\subsection{Hecke algebras}\label{subsec: Hecke algebras} 
In this section we employ the Taylor--Wiles--Kisin patching method,
following the approaches of \cite{kisinfmc} and \cite{blgg} (which in
turn follows \cite{cht}). In particular, in the actual implementation
of the patching method we follow \cite{blgg} very closely.

\begin{para}\label{subsubsec: assumptions on rbar}
Continue to assume
that $F$ is an imaginary CM field with maximal totally real field
subfield $F^+$ such that:
\begin{itemize}
\item $F/F^+$ is unramified at all finite places.
\item Every place $v|p$ of $F^+$ splits in $F$.
\item $[F^+:\Q]$ is even.
\end{itemize}
Let $G_{/\cO_{F^+}}$ be the algebraic group defined in section
\ref{sec: Unitary groups}.

Fix an absolutely irreducible representation $\rbar:G_F\to\GL_2(\F)$. Assume that
\begin{itemize}
\item $\rbar$ is automorphic in the sense of Section \ref{sec: Galois repns} (so
  that in particular $\rbar^c\cong\rbar^\vee\varepsilonbar^{-1}$).
\item $\rbar$ is unramified at all primes $v\nmid p.$
\item $\zeta_p\notin F$.
\item $\rbar|_{G_{F(\zeta_p)}}$ is absolutely irreducible.
\item The projective image of $\rbar$ is not isomorphic to $A_4$.
\end{itemize}

Let $U$ be a good compact open subgroup of $G(\A_{F^+}^\infty)$ (see
Definition \ref{defn: good subgroup}) such that if $U_v \subset G(\cO_{F_v})$ is not maximal
for some $v\nmid p,$ then
\begin{itemize}
\item $U_v$ is the preimage of the upper triangular unipotent matrices under 
\[G(\cO_{F^+_{v}})\to G(k_{v}) \underset{\iota_{w}}\iso \GL_2(k_{v})\]
where $w$ is a place of $F$ over $v.$
\item $v$ does not split completely in $F(\zeta_p);$ that is $(\mathbf{N}v) \neq 1$ mod $p.$
\item The ratio of the eigenvalues of $\rhobar(\Frob_{v})$
is not equal to $(\mathbf{N}v)^{\pm 1}$.
\end{itemize}

Finally, we assume that $U_{v_1}$ is not maximal for some place $v_1\nmid
p$ of $F^+$ such that
\begin{itemize}
\item for any non-trivial root of unity $\zeta$ in a
quadratic extension of $F$, $v_1$ does not divide $\zeta+\zeta^{-1}-2$.
\end{itemize}
Note that under these assumptions, $U$ is sufficiently small. Note
also that by Lemma 4.11 of \cite{MR1605752}, it is always possible to
choose a place $v_1$ which satisfies these hypotheses. 
\end{para}

\begin{para}
  Continue to let $E$ be a sufficiently large finite extension of
  $\Qp$ with ring of integers $\cO$ and residue field $\F$, and assume
  in particular that $E$ is large enough that $\rbar$ is defined over
  $\F$. As in section \ref{sec: Unitary groups}, we have a fixed set
  of places $\tS_p$ of $F$ dividing $p$, and we let $\tI_p$ denote the
  set of embeddings $F\into E$ giving rise to an element of $\tS_p$.
  Let $R$ denote the set of places $v$ of $F^+$ for which $U_v\ne
  G(\cO_{F^+_v})$, write $T=S_p\coprod R$, and define the Hecke
  $\cO$-algebra $\mathbb{T}^{T,\univ}$ as
  above. 

Fix a weight $\lambda\in(\Z^2_+)^{\tI_p}$, and for each
place $v\in S_p$, fix an inertial type $\tau_v$ of $I_{F_\tv}$.

We note that our assumptions on $\rbar$ and on $U_v$ at the places $v$ at which $U_v$ is not maximal 
imply that $S_{\lambda,\tau}(U,\cO)_\m\otimes_{\Z_p}\Q_p$ is either 0 or is locally free over 
$\mathbb{T}^T_{\lambda,\tau}(U,\cO)_\m[1/p]$ of rank $2^{|R|}$
(cf.\ Lemma 1.6(2) of \cite{tay-fm2}; the requisite multiplicity one result is
given by Theorems 5.4 and 5.9 of~\cite{labesse}).
\end{para}

\subsection{Deformations to $\cG_2$}
Let $\cG_2$ be as in section \ref{para: defn G2} and extend $\rbar$  
to a representation $\rhobar:G_{F^+}\to\cG_2(\F)$
with $\nu\circ\rhobar=\varepsilonbar^{-1}$ 
and $\rhobar|_{G_F}=(\rbar,\varepsilonbar^{-1})$ as in section \ref{para: defn G2}. 

\begin{para}
Let $\cC_\cO$ denote the category of complete local Noetherian
$\cO$-algebras with residue field isomorphic to $\F$ via the structure
map. Let $S$ be a set of places of $F^+$ which split in $F$,
containing all places dividing $p$. 
Regard $\rhobar$
as a representation of $G_{F^+,S}$. As in Definition 1.2.1 of \cite{cht}, we define
\begin{itemize}
\item a \emph{lifting} of $\rhobar$ to an object $A$ of
  $\mc{C}_{\mc{O}}$ to be a continuous homomorphism $\rho : G_{F^+,S}
  \rightarrow \mc{G}_2(A)$ lifting $\rhobar$ and with $\nu \circ \rho =
  \varepsilon^{-1}$;
\item two liftings $\rho$, $\rho^{\prime}$ of $\rhobar$ to $A$ to be
  \emph{equivalent} if they are conjugate by an element of
  $\ker(\GL_2(A)\rightarrow \GL_2(\F))$;
\item a \emph{deformation} of $\rhobar$ to an object $A$ of
  $\mc{C}_{\mc{O}}$ to be an equivalence class of liftings.
\end{itemize}
Similarly, if $T\subset S$, we define
\begin{itemize}
\item a \emph{$T$-framed lifting} of $\rhobar$ to $A$ to be a tuple
 $(\rho,\{ \alpha_v\}_{v \in T})$ where $\rho$ is a lifting of $\rhobar$ and
 $\alpha_v \in \ker(\GL_2(A)\rightarrow \GL_2(\F))$ for $v \in T$;
\item two $T$-framed liftings  $(\rho,\{ \alpha_v\}_{v \in T})$,
 $(\rho^{\prime},\{ \alpha^{\prime}_v\}_{v \in T})$ to be
 \emph{equivalent} if there is an element $\beta \in
 \ker(\GL_2(A)\rightarrow \GL_2(\F))$ with $\rho^{\prime}=\beta \rho
 \beta^{-1}$ and $\alpha_v^{\prime}=\beta \alpha_v$ for $v \in T$;
\item a \emph{$T$-framed deformation} of $\rhobar$ to be an equivalence
 class of $T$-framed liftings.
\end{itemize}
\end{para}

\begin{para}
For each place $v\in T$ we choose a place $\tv$ of $F$ above $v$,
extending the choices made for $v\in S_p$. Let $\tT$ denote the set of
places $\tv$, $v\in T$. For each $v\in T$, we let $R_\tv^\square$
denote the universal $\cO$-lifting ring of $\rbar|_{G_{F_\tv}}$. For
each $v\in S_p$, write $R_\tv^{\square,\lambda_v,\tau_v}$ for
$R_{\rbar|_{G_{F_\tv}}}^{\square,\lambda_v,\tau_v}$. 

We now recall from sections 2.2 and 2.3 of \cite{cht} the notion of a
\emph{deformation
  problem} \[\cS':=(L/L^+,T',\tT',\cO,\rbar,\chi,\{R_{\cS',\tv}\}_{v\in T'}).\]
  This data consists of
  \begin{itemize}
  \item an imaginary CM field $L$ with maximal totally
  real subfield $L^+$.
\item  a finite set of finite places $T'$ of $L^+$,
  each of which splits in $L$.
\item a finite set of finite places $\tT'$ of $L$, consisting of exactly one
  place lying over each place in $T'$.
\item the ring of integers $\cO$ of a finite extension $E$ of $\Qp$
  (assumed sufficiently large).
\item $\rbar:G_{L^+,T'}\to\cG_2(\F)$  a continuous homomorphism such
  that $\rbar^{-1}(\GL_2(\F)\times\GL_1(\F))=G_{L,T'}$, and
  $\rbar|_{G_{L,T'}}$ is irreducible.
\item $\chi:G_{L^+,T'}\to\cO^\times$  a continuous character lifting $\nu\circ\rbar$.
\item for each place $v\in T'$,  a quotient $R_{\cS',\tv}$ of
  $R_\tv^\square$ by an $1+M_2(\m_{R_\tv^\square})$-invariant ideal.
  \end{itemize}
For any deformation problem $\cS'$ as above, there is a universal
deformation $\cO$-algebra $R_{\cS'}^\univ$ and a universal deformation
$r_{\cS'}^{\univ}:G_{L^+,T'}\to\cG_2(R_{\cS'}^{\univ})$ of
$\rbar$, which is universal for deformations $r$ of $\rbar$ with
$\nu\circ r=\chi$ which satisfy the additional property that for each
$v\in T'$, the point of $\Spec R_\tv^\square$ corresponding to
$r|_{G_{F_{\tv}}}$ is a point of $\Spec R_{\cS',\tv}$. 
The $1+M_2(\m_{R_\tv^\square})$-invariance of $R_{\cS',\tv}$ implies that 
this condition does not depend on the choice of $r$ in its equivalence class. 
For any $T\subset T'$, we
also consider the universal $T$-framed lifting ring $R_{\cS'}^{\square_T}$,
which is universal for liftings of type $\cS'$ together with choices of basis at
the places in $T$ (see Definitions 2.2.1 and 2.2.7 of~\cite{cht}).

  Returning to our specific situation, consider the deformation problem
\[\cS:=(F/F^+,T,\tT,\cO,\rhobar,\varepsilon^{-1},\{R_\tv^\square\}_{v\in
  R}\cup\{R_\tv^{\square,\lambda_v,\tau_v}\}_{v\in S_p}).\] (The
quotients $R_\tv^{\square,\lambda_v,\tau_v}$ satisfy the condition
above by Lemma 3.2.3 of \cite{gg}.) There is a corresponding universal
deformation $\rho_\cS^{\univ}:G_{F^+,T}\to\cG_2(R_\cS^{\univ})$ of
$\rhobar$.

The lifting of Theorem \ref{existence of repns} and the universal property of 
$\rho_\cS^{\univ}$ gives an $\cO$-homomorphism \[R_\cS^{\univ}\onto
\mathbb{T}^T_{\lambda,\tau}(U,\cO)_\mf{m},\]which is surjective by Theorem
\ref{existence of repns}(\ref{item:char poly of universal Hecke deformation}).
\end{para}

\subsection{Patching}\label{subsec:patching}

\begin{para} In order to apply the Taylor--Wiles--Kisin method, and in
  particular to choose the auxiliary primes used in the patching
  argument, it is
  necessary to make an assumption on the image of the global mod $p$
  representation $\rbar$. In our setting, it will be convenient for us to use
  the notion of an \emph{adequate} subgroup of $\GL_2(\Fpbar)$, which is
defined in \cite{jack}. We will not need to make use of the actual
definition; instead, we recall the following classification.
\end{para}
\begin{prop} 
\label{prop:adequacy for n=2} Suppose that $p>2$ is a prime, and 
that $G\subset \GL_2(\Fbar_p)$ is a finite subgroup which acts irreducibly on $\Fbar_p^2$. 
Then precisely one of the following is true:
\begin{itemize}
\item We have $p=3$, and the image of $G$ in $\PGL_2(\Fbar_3)$ is  conjugate to $\PSL_2(\F_3)$.
\item We have $p=5$, and the image of $G$ in $\PGL_2(\Fbar_5)$ is  conjugate to $\PSL_2(\F_5)$.
\item $G$ is adequate.
\end{itemize}
\end{prop}
\begin{proof}
  This is Proposition A.2.1 of \cite{blggU2}.
\end{proof}

\begin{para}\label{subsubsec:assumption on adequate}
Assume from now on that
\begin{itemize}
\item $\rbar(G_{F(\zeta_p)})$ is adequate.
\end{itemize}

We wish to consider auxiliary sets of primes in order to apply the
Taylor--Wiles--Kisin patching method. Let $(Q,\tQ,\{\psibar_\tv\}_{v\in
  Q})$ be a triple where
\begin{itemize}
\item $Q$ is a finite set of finite places of $F^+$ which is disjoint
  from $T$ and consists of places which split in $F$;
\item $\tQ$ consists of a single place $\tv$ of $F$ above each place
  $v$ of $F^+$;
\item for each $v\in Q$,
  $\rbar|_{G_{F_\tv}}\cong\psibar_\tv\oplus\psibar'_\tv$ where
  $\psibar_\tv\ne\psibar'_\tv$, and $\mathbf{N}v\equiv 1\pmod{p}$.
\end{itemize}For each $v\in Q$, let $R^{\psibar_\tv}_v$ denote the
quotient of $R^\square_v$ corresponding to lifts
$r:G_{F_\tv}\to\GL_2(A)$ which are
$\ker(\GL_2(A)\to\GL_2(\F))$-conjugate to a lift of the form
$\psi\oplus\psi'$, where $\psi$ is a lift of $\psibar_\tv$ and $\psi'$
is an unramified lift of $\psibar'_\tv$. We let $\cS_Q$ denote the
deformation problem 
\[\cS_Q=(F/F^+,T\cup Q,\tT\cup
\tQ,\cO,\rhobar,\varepsilon^{-1},\{R_\tv^\square\}_{v\in
  R}\cup\{R_\tv^{\square,\lambda_v,\tau_v}\}_{v\in
  S_p}\cup\{R^{\psibar_\tv}\}_{v\in Q}).\] We let $R_{\cS_Q}^{\univ}$
denote the corresponding universal deformation ring, and we let
$R^{\square_T}_{\cS_Q}$ denote the corresponding universal $T$-framed deformation
ring.

We define 
\[ R^{\loc} := \left(\widehat{\otimes}_{v \in
  S_p}R_\tv^{\square,\lambda_v,\tau_v}\right)\widehat{\otimes} \left(\widehat{\otimes}_{v \in
  R}{R}^{\square}_{\tv}\right) \]
where all completed tensor products are taken over $\mc{O}$. 
\end{para}

\begin{rem}\label{rem: formally smooth def rings} Let $v\in R$. Since we have assumed that the ratio of the
  eigenvalues of $\rhobar(\Frob_v)$ is not equal to
  $(\mathbf{N}v)^{\pm 1}$, and $(\mathbf{N}v) \neq 1$ mod $p,$ 
  ${R}^{\square}_{\rbarwtvL}$ is formally smooth of relative dimension
  $4$ over $\cO$ (this may be checked by computing the dimension of the reduced
  tangent space by the usual Galois cohomology calculation; see Lemma 2.4.9 of~\cite{cht}), and in particular all deformations of $\rbarwtvL$ are
  unramified. Applying Proposition \ref{defrings} above and Lemma
  3.3 of \cite{blght}, we see that  $R^{\loc}$ is equidimensional of dimension $1 + 4 |T| +
[F^+:\bb{Q}]$, and $R^\loc[1/p]$ is formally smooth.
\end{rem}

\begin{para} For each finite place $w$ of $F$, let $U_0(w)$ be the subgroup of
$\GL_2(\cO_{F,w})$ consisting of matrices congruent to $
\begin{pmatrix}
  *&*\\0&*
\end{pmatrix}$ modulo $w$, and let  $U_1(w)$ be the subgroup of
$\GL_2(\cO_{F,w})$ consisting of matrices congruent to $
\begin{pmatrix}
  *&*\\0&1
\end{pmatrix}$ modulo $w$. For $i=0,1,$ let 
$U_i(Q)=\prod_v U_i(Q)_v$ be the compact open subgroups of 
$G(\A_{F^+}^\infty)$ defined by  $U_i(Q)_v=U_v$ if $v\notin Q$, and
$U_i(Q)_v=\iota_{\tv}^{-1}U_i(\tv)$ if $v\in Q$. We have natural maps 
\[ \bb{T}_{\lambda,\tau}^{T\cup Q}(U_{1}(Q),\mc{O}) \onto \bb{T}_{\lambda,\tau}^{T\cup Q}(U_{0}(Q),\mc{O}) \onto \bb{T}_{\lambda,\tau}^{T\cup Q}(U,\mc{O}) \into \bb{T}_{\lambda,\tau}^T(U,\mc{O}).\]
Note that $\bb{T}_{\lambda,\tau}^{T \cup Q}(U,\mc{O})_{\mf{m}} = \bb{T}_{\lambda,\tau}^T(U,\mc{O})_{\mf{m}}$ by the proof of Corollary 3.4.5 of \cite{cht}.
By taking its image, $\mf{m}$ determines maximal ideals of the first three algebras in this sequence which we denote by $\mf{m}_{Q}$ for the first two and $\mf{m}$ for the third. 
\end{para}

\begin{para} Let $\rho_{\m_Q}:G_{F^+,T\cup Q}\to\GL_2(\bb{T}^{T\cup
  Q}_{\lambda,\tau}(U_1(Q),\mc{O})_{\mf{m}_{Q}})$ be the representation defined in Theorem
\ref{existence of repns}. For each $v \in Q$ choose an element
$\varphi_{\wt{v}} \in G_{F_{\wt{v}}}$ lifting the geometric Frobenius
element of $G_{F_{\tv}}/I_{F_{\tv}}$ and
let $\varpi_{\wt{v}} \in \mc{O}_{F_{\wt{v}}}$ be the uniformiser with
$\Art_{F_{\wt{v}}}\varpi_{\wt{v}}=\varphi_{\wt{v}}|_{F_{\wt{v}}^{\ab}}$. Let
$P_{\wt{v}}(X) \in \bb{T}^{T\cup
  Q}_{\lambda,\tau}(U_1(Q),\mc{O})_{\mf{m}_{Q}} [X]$ denote the characteristic
polynomial of $\rho_{\mf{m}_{Q}}(\varphi_{\wt{v}})$. Since $\overline{\psi}_{\wt{v}}(\varphi_{\wt{v}})\ne\overline{\psi}'_{\wt{v}}(\varphi_{\wt{v}})$,
by Hensel's lemma we can factor 
$P_{\wt{v}}(X)=(X-A_{\wt{v}})(X-B_{\wt{v}})$ where $A_{\wt{v}}, B_\tv \in \bb{T}^{T\cup Q}_{\lambda,\tau}(U_1(Q),\mc{O})_{\mf{m}_{Q}}$
lift $\overline{\psi}_{\wt{v}}(\varphi_{\wt{v}})$ and $\overline{\psi}'_{\wt{v}}(\varphi_{\wt{v}})$ respectively. 

For
$i=0,1$ and $\alpha \in F_{\wt{v}}^\times$ of non-negative valuation, consider the Hecke operator
\[V_{\alpha}:= \iota_{\wt{v}}^{-1}\left( \left[ U_{i}(\wt{v})  \left( \begin{matrix}
      1 & 0 \cr 0 & \alpha \end{matrix} \right)
  U_{i}(\wt{v}) \right] \right) \]
on $S_{\lambda,\tau}(U_i(Q),\mc{O})$. 
Denote by $\bb{T}^{T\cup  Q}_{\lambda,\tau}(U_i(Q),\mc{O})' 
\subset \End_\cO(S_{\lambda,\tau}(U_i(Q),\mc{O}))$
the $\cO$-subalgebra  generated by $\bb{T}^{T\cup
  Q}_{\lambda,\tau}(U_i(Q),\mc{O})$ and the $V_{\varpi_{\wt{v}}}$ for
$v\in Q$. We denote by $\mf{m}'_Q$ the maximal ideal of 
$\bb{T}^{T\cup Q}_{\lambda,\tau}(U_i(Q),\mc{O})'$ 
generated by $\mf{m}_Q$ and the $V_{\varpi_{\tv}}-A_{\wt{v}}.$ Write
$\T_{i,Q}:=\bb{T}^{T\cup
  Q}_{\lambda,\tau}(U_i(Q),\mc{O})'_{\mf{m}'_Q}$. 

Let 
  $\Delta_{Q}$ denote the maximal $p$-power
order quotient of $U_0(Q)/U_1(Q)$. Let $\mf{a}_{Q}$ denote the kernel of
the augmentation map $\mc{O}[\Delta_{Q}] \rightarrow \mc{O}$.
Exactly as in the proof of the sublemma of Theorem 3.6.1 of
 \cite{blgg}, we have:
\begin{enumerate}
\item The natural map
\[ \prod_{v\in Q}(V_{\varpi_{\tv}}-B_{\tv}): S_{\lambda,\tau}(U,\mc{O})_{\mf{m}} \rightarrow  
S_{\lambda,\tau}(U_0(Q),\mc{O})_{\mf{m}'_{Q}} \]
is an isomorphism.
\item $ S_{\lambda,\tau}(U_1(Q),\mc{O})_{\mf{m}'_{Q}}$ is free over $\mc{O}[\Delta_{Q}]$ with
\[  S_{\lambda,\tau}(U_1(Q),\mc{O})_{\mf{m}'_{Q}}/{\mf{a}_{Q}} \isoto  
S_{\lambda,\tau}(U_0(Q),\mc{O})_{\mf{m}'_Q} \isoto  
S_{\lambda,\tau}(U,\mc{O})_{\mf{m}}. \]
\item For each $v \in Q$, there is a character with
  open kernel $V_{\wt{v}} : F_{\wt{v}}^{\times} \rightarrow
  \bb{T}_{1,Q}^{\times}$ so that
  \begin{enumerate}
  \item For each element $\alpha \in F_{\wt{v}}^\times$ of non-negative valuation,
    $V_{\alpha}= V_{\wt{v}}(\alpha)$ on $S_{\lambda,\tau}(U_1(Q),\mc{O})_{\mf{m}'_{Q}}.$
  \item 
$(\rho_{\mf{m}_{Q}}\otimes_{\bb{T}^{T\cup Q}_{\lambda,\tau}(U_1(Q),\mc{O})_{\mf{m}_{Q}}} \bb{T}_{1,Q})|_{W_{F_{\wt{v}}}}
    \cong \psi' \oplus (V_{\wt{v}}\circ \Art_{F_{\wt{v}}}^{-1})$ with
    $\psi'$ an unramified lift of $\psibar'_{\wt{v}}$ and $(V_{\wt{v}}\circ \Art_{F_{\wt{v}}}^{-1})$ lifting $\overline{\psi}_{\wt{v}}$.
  \end{enumerate}
\end{enumerate}
\end{para}

\begin{para} The above shows, in particular, that 
the lift $\rho_{\mf{m}_{Q}} \otimes \bb{T}_{1,Q}$ of $\rhobar$ is of
type $\mc{S}_{Q}$ and gives rise to a surjection
$R^{\univ}_{\mc{S}_{Q}} \onto \bb{T}_{1,Q}$. We think of 
$S_{\lambda,\tau}(U_1(Q),\mc{O})_{\mf{m}_{Q}}$ as an $R^{\univ}_{\mc{S}_{Q}}$-module 
via this map.

Thinking of $\Delta_{Q}$ as the image of the product of the inertia subgroups in the 
maximal abelian $p$-power order quotient of
$\prod_{v \in Q} G_{F_{\wt{v}}}$, the determinant of any choice of
universal deformation $r^{\univ}_{\mc{S}_{Q}}$ gives rise to a
homomorphism $\Delta_{Q} \rightarrow (R^{\univ}_{\mc{S}_{Q}})^{\times}$. We
thus have homomorphisms 
$$\mc{O}[\Delta_{Q}] \rightarrow 
R^{\univ}_{\mc{S}_{Q}} \rightarrow R^{\square_T}_{\mc{S}_{Q}}$$ 
and natural isomorphisms $R^{\univ}_{\mc{S}_{Q}}/\mf{a}_{Q} \cong
R^{\univ}_{\mc{S}}$ and $R^{\square_T}_{\mc{S}_{Q}}/\mf{a}_{Q}
\cong R^{\square_T}_{\mc{S}}$. (This follows from (3)(b) above, which shows that
for each place $v\in Q$,
the ramification of $r^{\univ}_{\mc{S}_{Q}}$ at $\tv$ is given by the character $(V_{\wt{v}}\circ \Art_{F_{\wt{v}}}^{-1})$.)
\end{para}

\begin{para} 
  We have assumed that $\rbar(G_{F(\zeta_p)})$ is adequate, or
  equivalently (because we are considering 2-dimensional
  representations) big in the terminology of \cite{cht}. By
  Proposition 2.5.9 of \cite{cht}, this implies that we can (and do)
  choose an integer $q\ge[F^+:\Q]$ and for each $N\ge 1$ a tuple
  $(Q_N,\tQ_N,\{\psibar_\tv\}_{v\in Q_N})$ as above such that
\begin{itemize}
\item $\# Q_N = q$ for all $N$;
\item $\mathbf{N}v \equiv 1 \mod p^N$ for $v \in Q_N$;
\item the ring $R^{\square_T}_{\mc{S}_{Q_N}}$ can be topologically
  generated over $R^{\loc}$ by $q-[F^+:\Q]$ elements.
\end{itemize}
We will apply the above constructions to each of these tuples 
$(Q_N,\tQ_N,\{\psibar_\tv\}_{v\in
  Q_N})$.

Choose a lift $r^{\univ}_{\mc{S}} : G_{F^+,S} \rightarrow
\mc{G}_2(R^{\univ}_{\mc{S}})$ representing the universal
deformation. Let
\[ \mc{T} = \mc{O}[[X_{v,i,j}:v \in T, i,j = 1,2]].\] 
The tuple $(r^{\univ}_{\mc{S}},(1_2+X_{v,i,j})_{v \in T})$ (where
$1_2+X_{v,i,j}$ is a $2\times 2$ matrix) gives rise to an isomorphism
$R^{\square_T}_{\mc{S}} \isoto
R^{\univ}_{\mc{S}}\widehat{\otimes}_{\mc{O}} \mc{T}.$ (Note that the
action of $j$ in the group $\cG_2$ implies that this tuple has no non-trivial
scalar endomorphisms.)
For each $N$, choose a lift $r^{\univ}_{\mc{S}_{Q_N}} : G_{F^+}
\rightarrow \mc{G}_2(R^{\univ}_{\mc{S}_{Q_N}})$ representing the
universal deformation, with $r^{\univ}_{S_{Q_N}} \mod
\mf{a}_{Q_N}=r^{\univ}_{\mc{S}}$. This gives rise to an isomorphism
$R^{\square_T}_{\mc{S}_{Q_N}} \isoto
R^{\univ}_{\mc{S}_{Q_N}}\widehat{\otimes}_{\mc{O}}\mc{T}$ which
reduces modulo $\mf{a}_{Q_N}$ to the isomorphism
$R^{\square_T}_{\mc{S}}\isoto
R^{\univ}_{\mc{S}}\widehat{\otimes}_{\mc{O}} \mc{T}$. 

We let
\begin{eqnarray*}
  M & = & S_{\lambda,\tau}(U,\mc{O})_{\mf{m}}
  \\
  M_{N} & = & S_{\lambda,\tau}(U_1(Q_N),\mc{O})_{\mf{m}_{Q_N}'}
  \otimes_{R^{\univ}_{\mc{S}_{Q_N}}} R^{\square_T}_{\mc{S}_{Q_N}}.
\end{eqnarray*}
Then $M_{N}$ is a finite free
$\mc{T}[\Delta_{Q_N}]$-module with 
$M_N/\mf{a}_{Q_N} \cong M\otimes_{R^{\univ}_{\mc{S}}}R^{\square_T}_{\mc{S}} ,$
compatibly with the isomorphism
$R^{\square_T}_{\mc{S}_{Q_N}}/\mf{a}_{Q_N} \cong R^{\square_T}_{\mc{S}}$.

Fix a filtration by $\F$-subspaces \[0=L_0\subset
L_1\subset\dots\subset L_s=L_{\lambda,\tau}\otimes_{\cO} \F\]such that each
$L_i$ is $G(\cO_{F^+_p})$-stable, and for each $i=0,1,\dots,s-1$,
$\sigma_i:=L_{i+1}/L_i$ is an absolutely irreducible representation of 
$G(\cO_{F^+_p}).$ This in turn
induces a filtration on $S_{\lambda,\tau}(U,\cO)_\mf{m}\otimes_\cO\F$
(respectively
$S_{\lambda,\tau}(U_1(Q_N),\cO)_{\mf{m}_{Q_N}}\otimes_\cO\F$) whose
graded pieces are the finite-dimensional $\F$-vector spaces
$S(U,\sigma_i)_\mf{m}$ (respectively the finite free
$\F[\Delta_{Q_N}]$-modules $S(U_1(Q_N),\sigma_i)_{\mf{m}_{Q_N}}$). By
extension of scalars we obtain a filtration on 
$M_{N}^{}\otimes_\cO\F.$ We denote these filtrations by \[0=M^{0}\subset
M^{1}\subset\dots\subset
M^{s}=M^{}\otimes_\cO
\F\]and \[0=M_{N}^{0}\subset
M_{N}^{1}\subset\dots\subset
M_{N}^{s}=M_{N}^{}\otimes_\cO \F.\]

Let $g=q-[F^+:\Q]$ and let
\begin{eqnarray*}
  \Delta_{\infty} & = & \bb{Z}_p^q, \\
  R_{\infty} & = & R^{\loc}[[x_1,\ldots,x_g]], \\
  R'_{\infty} & = & \left(\widehat{\otimes}_{v \in
  T}R_\tv^{\square}\right)[[x_1,\ldots,x_g]], \\
  S_{\infty} & = & \mc{T}[[\Delta_{\infty}]], 
\end{eqnarray*}
and let $\mf{a}$ denote the kernel of the $\mc{O}$-algebra homomorphism $S_{\infty}
\rightarrow \mc{O}$ which sends each $X_{v,i,j}$ to 0 and each element
of $\Delta_{\infty}$ to 1. Note that $S_\infty$ is  formally smooth
over $\bigO$ of relative dimension $q+4|T|$, and that $R_\infty$ is a
quotient of $R'_\infty$. For each $N$, choose a surjection
$\Delta_{\infty} \onto \Delta_{Q_N}$ and let $\mf{c}_N$ denote the
kernel of the corresponding homomorphism $S_{\infty} \onto \mc{T}[\Delta_{Q_N}]$.
For each $N \geq 1$, choose a surjection of $R^{\loc}$-algebras
\[ R_{\infty} \onto R^{\square_T}_{\mc{S}_{Q_N}}.\]
We regard each $R^{\square_T}_{\mc{S}_{Q_N}}$ as an
$S_{\infty}$-algebra via $S_{\infty}\onto \mc{T}[\Delta_{Q_N}]
\rightarrow R^{\square_T}_{\mc{S}_{Q_N}}$. In particular,
$R^{\square_T}_{\mc{S}_{Q_N}}/\mf{a} \cong R^{\univ}_{\mc{S}}$.

Now a patching argument as in \cite{kisinfmc} 2.2.9 shows that there exists 
\begin{itemize}
\item an $\cO$-module homomorphism $S_\infty\to R_\infty$, and an
  $R_\infty$-module $M_\infty$ which is finite free as an $S_{\infty}$-module,
\item a filtration by $R_{\infty}$-modules 
\[0=M_{\infty}^{0}\subset M_{\infty}^{1}\subset\dots\subset M_{\infty}^{s}=M_{\infty}^{}\otimes_\cO \F\] 
whose graded pieces are finite free $S_{\infty}/\pi S_{\infty}$-modules,
\item a surjection of $R^{\loc}$-algebras $R_{\infty}/\mf{a}R_{\infty}
  \rightarrow R^{\univ}_{\mc{S}}$, and
\item an isomorphism of $R_{\infty}$-modules
  $M_{\infty}/\mf{a}M_{\infty} \iso M$ which identifies $M^i$ with $M^i_{\infty}/\mf{a}M^i_{\infty}$.
\end{itemize}

We claim that we can make the above construction so that for $i=1,2,\dots s,$ 
the $(R'_{\infty},S_{\infty})$-bimodule $M^i_{\infty}/M^{i-1}_\infty$ and the isomorphism 
$M^i_{\infty}/(\mf{a}M^i_{\infty}+M^{i-1}_\infty) \iso M^i/M^{i-1}$ depends only on ($U,\mf{m}$ and) the isomorphism class 
of $L_i/L_{i-1}$ as a $G(\O_{F^+_p})$-representation, but not on $(\lambda,\tau).$ For any finite 
collection of pairs $(\lambda,\tau)$ this follows by the same finiteness argument used during patching. 
Since the set of $(\lambda,\tau)$ is countable, the claim follows from a diagonalization argument.

For $\sigma$ a global Serre weight, we denote by $M^\sigma_{\infty}$ the $R_{\infty}/\pi R_{\infty}$-module constructed above 
when $L_i/L_{i-1} \iso \sigma,$ and we set 
$$ \mu_\sigma'(\rbar) = 2^{-|R|}e_{R_{\infty}/\pi}(M^\sigma_\infty).$$
\end{para}

\begin{lem}\label{key patching lemma} For each $\sigma,$ $\mu'_\sigma(\rbar)$ is a non-negative integer. 
Moreover, the following conditions are equivalent.
\begin{enumerate}
\item The support of $M$ meets every irreducible component of $\Spec R^{\loc}[1/p].$
\item $M_{\infty}\otimes_{\Z_p}\Q_p$ is a faithfully flat $R_{\infty}[1/p]$-module 
which is locally free of rank $2^{|R|}.$
\item 
$R^{\univ}_{\mc{S}}$ is a finite $\O$-algebra and $M\otimes_{\Z_p}\Q_p$ is a faithful $R^{\univ}_{\mc{S}}[1/p]$-module.  
\item 
\[e(R_{\infty}/\pi R_{\infty}) =  \sum_{i=1}^s 2^{-|R|}e_{R_{\infty}/\pi}(M^{\sigma_i}_{\infty}) = \sum_{i=1}^s \mu_{\sigma_i}'(\rbar)\]
where $\sigma_i$ is a global Serre weight with $L_i/L_{i-1} \iso \sigma_i.$ 
\end{enumerate}
\end{lem}
\begin{proof} We argue in a similar fashion to the proof of Lemma~2.2.11
  of~\cite{kisinfmc}. By Remark \ref{rem: formally smooth def rings},
  $R_{\infty}[1/p]$ is formally smooth of dimension $q+4|T| = \dim
  S_{\infty}[1/p].$ Since $M_{\infty}$ is free over $S_{\infty}$, the module
  $M_{\infty}\otimes_{\Z_p}\Q_p$ has depth $q+4|T|$ at every maximal ideal of
  $R_{\infty}[1/p]$ in its support. By the Auslander--Buchsbaum formula,
  $M_{\infty}\otimes_{\Z_p}\Q_p$ has projective dimension $0$, and as it is
  finite over $S_\infty[1/p]$ it is also finite and therefore finite flat over
  $R_{\infty}[1/p].$

If $Z \subset \Spec R_{\infty}[1/p]$ is an irreducible component in the support of $M_{\infty},$ then $Z$ is 
finite over $\Spec S_{\infty}[1/p]$ and of the same dimension. 
Hence the map $Z \rightarrow \Spec S_{\infty}[1/p]$ 
is surjective. As $(M_{\infty}/\mf{a}M_{\infty})[1/p] = M[1/p]$ has rank $2^{|R|}$ over any point of $R^{\univ}_{\mc{S}}[1/p]$ 
in its support,  $M_{\infty}\otimes_{\Z_p}\Q_p$ has rank $2^{|R|}$ over $Z.$ This shows that (1) and (2) are equivalent.

Let $\Spec A \subset \Spec R_{\infty}$ denote the closure of the
support of $M_{\infty},$ on $\Spec R_{\infty}[1/p].$ By what we have just seen, there exists 
a map $A^{2^{|R|}} \rightarrow M_{\infty},$ which is an isomorphism at the generic points of 
$\Spec A.$ Using Proposition 1.3.4 of \cite{kisinfmc} we see that 
$$ 2^{-|R|} e_{R_\infty/\pi}(M_{\infty}/\pi M_{\infty}) = 
2^{-|R|} e_{A/\pi}(M_{\infty}/\pi M_{\infty}) = e(A/\pi)$$
is an integer. If $\sigma = \sigma_a$ is a global Serre weight, as in Section~\ref{subsec: globalserreweights}, 
then applying the above with $(\lambda,\tau) = (\lambda_a,1)$ shows that 
$\mu'_{\sigma}(\rbar)$ is an integer.

We also have 
\[e(R_{\infty}/\pi R_{\infty}) \geq 2^{-|R|} e_{R_\infty/\pi}(M_{\infty}/\pi M_{\infty}) 
= \sum_{i=1}^s 2^{-|R|}e_{R_{\infty}/\pi}(M^{\sigma_i}_{\infty})\]
 with equality if and only if $M_{\infty}$ is a faithful
 $R_{\infty}$-module or, equivalently, if and only if 
the support of $M$ meets every irreducible component of $R_{\infty}[1/p].$ 
So (1) and (4) are equivalent. 

Finally, if $M_{\infty}$ is a faithful $R_{\infty}$-module, then
$R_{\infty}$ is finite over $S_{\infty},$ and so $R^{\univ}_{\mc{S}},$
which is a quotient of $R_{\infty}/\mf{a},$ is a finite
$\O$-module. This shows that (2) implies (3). The converse is a consequence of
the Khare--Wintenberger argument, cf.\ Thm.\ 3.3 of~\cite{kw}.  To be precise, assuming
$R^{\univ}_{\mc{S}}$ is a finite $\O$-algebra, we see that the image of
$$ \Spec R^{\univ}_{\mc{S}} \rightarrow \Spec R^{\loc}$$ 
meets every component of $\Spec R^{\loc}[1/p],$ because it follows from Proposition
1.5.1(2) of \cite{BLGGT}  that  the quotient $R^{\univ}_{\mc{S}}$ 
corresponding to any particular component of $\Spec R^{\loc}[1/p]$ is a finite $\cO$-algebra of dimension
at least $1$, and therefore has $\Qpbar$-points. Hence if $M\otimes_{\Z_p}\Q_p$ is a faithful $R^{\univ}_{\mc{S}}$-module, 
then the support of $M$ meets every component of $\Spec R^\loc[1/p]$, which is (1). 
\end{proof}

Let $\sigma = \otimes_{v \in S_p} \sigma_v$ be a global Serre weight.  
The following lemma will be useful in order to determine the
$\mu'_\sigma(\rbar)$ more precisely in some situations. 
\begin{lem}
  \label{lem: multiplicity positive implies Serre weight and thus
    crystalline lift}The multiplicity $\mu'_\sigma(\rbar)$ is nonzero if
  and only if $S(U,\sigma)_\mathfrak{m}\ne 0$. If this holds, then for
  each place $v|p$ of $F$, $\rbar|_{G_{F_v}}$ has a crystalline lift of Hodge
  type $\sigma_v$ in the sense of section \ref{defn: lift of Hodge type a}.
\end{lem}
\begin{proof}
  By definition, $\mu'_\sigma(\rbar)\ne 0$ if and only if $M^\sigma_\infty\ne
  0$. Moreover, $M^\sigma_\infty\ne 0$ if and only if
  $M^\sigma_\infty/\mathfrak{a}M^\sigma_\infty\ne 0$, and
  $M^\sigma_\infty/\mathfrak{a}M^\sigma_\infty\cong S(U,\sigma)_\mathfrak{m}$ by
  definition, so we indeed have $\mu'_\sigma(\rbar)\ne 0$ if and only if
  $S(U,\sigma)_\mathfrak{m}\ne 0$.

 For the second part, let $a = (a_v)_{v \in S_p}$ with 
$a_v \in (\Z^2_+)^{\Hom(k_v,\F)}_{\SW}$ and $\sigma_v \cong \sigma_{a_v}.$
Note that since $U$ is sufficiently small, we have $S(U,\sigma)_\mathfrak{m}\ne 0$ if and only if
$S_{\lambda_a,1}(U,\cO)\ne 0,$ where $\lambda_a$ is defined in section
\ref{subsec: globalserreweights}, and $1$ denotes the
trivial type. The result then follows at
once from Theorem \ref{existence of repns}(4).
\end{proof}

\subsection{Potential diagonalizability}\label{subsec: pot diag}We now
use the methods of \cite{BLGGT} to show that the equivalent conditions
of Lemma \ref{key patching lemma} are frequently achieved. We begin by
recalling the definition of \emph{potential diagonalizability}, a 
notion defined in \cite{BLGGT}. We will use this definition here for
convenience, as it allows us to make use of certain results from
\cite{blggU2}, and to easily argue simultaneously in the potentially
Barsotti--Tate and Fontaine--Laffaille cases. 

Suppose that $K/\Qp$ is a finite extension 
with residue field $k,$ that $E/\Qp$ is a finite extension with ring of integers $\cO$ and
residue field $\F$,
and that $\rho_1,\rho_2:G_K\to\GL_2(\cO)$ are two continuous
representations. We say that $\rho_1$ \emph{connects} to $\rho_2$ if
all of the following hold:
\begin{itemize}
\item $\rho_1$ and $\rho_2$ are both crystalline of the same Hodge
  type $\lambda$,
\item $\rhobar_1\cong\rhobar_2$, and
\item  $\rho_1$ and $\rho_2$ define points on the same irreducible
  component of $R^{\square,\lambda}_{\rhobar_1}\otimes_{\cO}\Qpbar$.
\end{itemize}
We say that $\rho:G_K\to\GL_2(\cO)$ is \emph{diagonal} if it is a
direct sum of crystalline characters, and we say that $\rho$ is
\emph{diagonalizable} if it connects to some diagonal
representation. Finally, we say that $\rho$ is \emph{potentially
  diagonalizable} if there is a finite extension $L/K$ such that
$\rho|_{G_L}$ is diagonalizable. We say that a representation
$G_K\to\GL_2(E)$ is potentially diagonalizable if the representation
on some $G_K$-invariant lattice is potentially diagonalizable; this is
independent of the choice of lattice by Lemma 1.4.1 of \cite{BLGGT}.

The following two lemmas, which rely on our earlier papers, are the
key to our applications of Lemma \ref{key patching lemma} to the
Breuil--M\'ezard conjecture for potentially Barsotti--Tate representations.
\begin{lem}
  \label{lem: pot BT implies pot diag}If $\rho:G_K\to\GL_2(E)$ is
  potentially Barsotti--Tate, then it is potentially diagonalizable.
\end{lem}
\begin{proof}Choose a finite extension $L/K$ such that $L$ contains a
  primitive $p$-th root of unity, and $\rhobar|_{G_L}$ is trivial.
  Then $\rhobar|_{G_L}$ has a decomposable ordinary crystalline lift
  of Hodge type $0$, namely
  $\rho_1:=1\oplus\varepsilon^{-1}$. Extending $L$ if necessary, we
  may also assume that $\rhobar|_{G_L}$ has a decomposable
  non-ordinary crystalline lift of Hodge type $0$, say $\rho_2$ (this is simply
  a direct sum of Lubin--Tate characters). 
By Proposition 2.3 of
  \cite{MR2280776} and Corollary 2.5.16 of \cite{kis04} we see that
  $\rho|_{G_L}$ connects to one of $\rho_1,\rho_2$, and in either case
  $\rho$ is potentially diagonalizable by
  definition.\end{proof}

\begin{lem}
  \label{lem: unramified regular weight implies pot diag} 
Let $a \in (\Z^2_+)^{\Hom(k_v,\F)}_{\SW}$ and $\sigma = \sigma_a$ the corresponding Serre weight. 
\begin{enumerate}
\item If  $\sigma$ is
  not a predicted Serre weight for $\rbar$, then
  $R_{\rbar}^{\square,\sigma}=0$. In particular, it is vacuously the case
  that every crystalline representation $\rho:G_K\to\GL_2(E)$ of Hodge
  type $a$ which lifts $\rhobar$ is potentially diagonalizable.
\item If $K/\Qp$ is unramified and $\sigma$ is regular in the sense of
  Definition \ref{defn:regular weight local} then every crystalline
  representation $\rho:G_K\to\GL_2(E)$ of Hodge type $a$ which lifts
  $\rhobar$ is potentially diagonalizable.
  \end{enumerate}

\end{lem}
\begin{proof}In the case that $\sigma$ is not a predicted Serre weight for
  $\rhobar$, the main result of \cite{GLS13} shows that there are no
  crystalline lifts of $\rhobar$ of Hodge type $a$, and the result
  follows.

 If $K/\Qp$ is unramified and $\sigma$ is regular, then the result follows
  immediately from the main theorem of \cite{GaoLiu12}. 
\end{proof}
We will apply these results by using the following corollary of the
methods of \cite{BLGGT}. We maintain the notations and assumptions of
the rest of this section.
\begin{cor}
  \label{cor: main patching lemma in the potentially diagonalizable
    case}Suppose that $\rbar:G_F\to\GL_2(\F)$ satisfies the 
  assumptions of \S\ref{subsubsec: assumptions on rbar} and \S\ref{subsubsec:assumption on adequate}, and that for each place $v|p$, every lift of
  $\rbar|_{G_{F_{\tv}}}$ of Hodge type $\lambda_v$ and Galois type
  $\tau_v$ is potentially diagonalizable. Then the equivalent
  conditions of Lemma \ref{key patching lemma} hold.
\end{cor}
\begin{proof}We will show that condition (1) of Lemma \ref{key patching lemma}
  holds, i.e. that the support of $M$ meets every
  irreducible component of $\Spec R^{\loc}[1/p]$.  By the
  correspondence between our algebraic automorphic forms and
  automorphic forms on $\GL_2$ (which is explained in detail in
  section 2 of \cite{blggU2}), this is equivalent to the
  statement that, if we make for each place $v|p$ of $F^+$ any choice
  of component $\Spec R_\tv$ of $\Spec R_\tv^{\square,\lambda_v,\tau_v}[1/p]$,
  there is a continuous lift $r:G_F\to\GL_2(\Qpbar)$ of $\rbar$ such
  that
   \begin{itemize}
   \item $r^c\equiv r^\vee\varepsilon^{-1}$,
   \item $r$ is unramified at all places not dividing $p$ (note that this will
     be automatic at the places in $R$ by Lemma~\ref{rem: formally smooth def rings}),
   \item for each place $v|p$ of $F^+$, $r|_{G_{F_\tv}}$ corresponds
     to a point of $R_\tv$, and
   \item $r$ is automorphic in the sense of \cite{BLGGT}.
   \end{itemize}

 By Lemma 3.1.1 of
   \cite{blggU2}, we may choose a solvable extension $F_1/F$ of CM
   fields such that 
 \begin{itemize}
 \item $F_1$ is linearly disjoint from
   $\bar{F}^{\ker\ad\rbar}(\zeta_p)$ over $F$.
 \item there is a continuous lift $r':G_{F_1}\to\GL_2(\Qpbar)$ of
   $\rbar|_{G_{F_1}}$ such that $r'$ is automorphic, and for each place
   $w|p$ of $F_1$, $r'|_{G_{F_{1,w}}}$ is (potentially) diagonalizable.
 \end{itemize} 

  The existence of $r$ now follows by applying Theorem A.4.1 of
  \cite{blggU2} (which is another variant of the Khare--Wintenberger argument;
  again, cf.\ Thm.\ 3.3 of~\cite{kw}) with the representation $r_{l,\imath}(\pi')$ in the statement of 
{\it loc.~cit} equal to $r'.$ It is here that we use the assumption that every lift of
  $\rbar|_{G_{F_{\tv}}}$ of Hodge type $\lambda_v$ and Galois type
  $\tau_v$ is potentially diagonalizable, as we need to know that the
  points of $\Spec R_\tv$ correspond to potentially diagonalizable
  lifts in order to satisfy the hypotheses of Theorem A.4.1. of {\it loc.~cit.}
 (note that since $r|_{G_{F_1}}$ is automorphic and
 the extension $F_1/F$ is solvable, $r$ is automorphic by Lemma 1.4 of
 \cite{blght}).
\end{proof}

\subsection{Local results}\label{subsec: main local results}We will
now combine Corollary \ref{cor: main patching lemma in the potentially diagonalizable
    case} with the local-to-global results of Appendix \ref{sec:local
    to global} to prove our main local results. We begin with a lemma
  from linear algebra, for which we need to establish some
  notation. 
Given a vector space $V$ over $\Q$ with a choice of basis, we let
  $V_{\ge 0}$ denote the cone spanned by nonnegative linear
  combinations of the basis elements. If $V,W$ are vector spaces over $\Q$ with
  choices of bases, then we will choose the corresponding tensor basis
  for $V\otimes W$, and define $(V\otimes W)_{\ge 0}$ accordingly. For
  any set $I$, we
  set $\Z^I_{\ge 0} = \Z^I\cap\Q^I_{\ge 0}.$ In particular,  
$(\Z^m)^{\otimes n}_{\ge 0} = (\Z^m)^{\otimes n}\cap (\Q^m)^{\otimes n}_{\ge 0}.$

  \begin{lem}
    \label{lem: linear algebra lemma}Let $k$ be a field.
    \begin{enumerate}
    \item If for $i=1,\dots,n$ we have injective linear maps
      $\alpha_i:V_i\into W_i$ between $k$-vector spaces, then
      $\alpha_1\otimes\cdots\otimes\alpha_n:V_1\otimes\cdots\otimes
      V_n\to W_1\otimes\cdots\otimes W_n$ is also injective.
    \item If for $i=1,\dots,n$ we have linear maps $\alpha_i:V_i\to
      W_i$ between $k$-vector spaces and nonzero elements $w_i\in W_i$
      such that \[w_1\otimes\cdots\otimes
      w_n\in\Im(\alpha_1\otimes\cdots\otimes\alpha_n),\] then for each
      $i$, $w_i\in\Im(\alpha_i)$.
    \item Let $I$ be a (possibly infinite) set and
      $\alpha:\Z_{\ge 0}^m\to\Z_{\ge 0}^I$ a map that extends to an
      injective linear map $\alpha:\Q^m\to\Q^I$. Suppose that
      $v\in\Q^m$,  with $\alpha(v)\in\Z_{\ge 0}^I$ and that for some $n\ge 1$, $v$ satisfies
      $v^{\otimes n}\in(\Z^m)^{\otimes n}_{\ge 0}.$  
      Then $v\in\Z_{\ge 0}^m$.
    \end{enumerate}

  \end{lem}
  \begin{proof}
  (1) By induction, it suffices to treat the case $n=2$. Then we
      can factor $\alpha_1\otimes \alpha_2$ as the
      composite \[V_1\otimes V_2\to W_1\otimes V_2\to W_1\otimes W_2\]
      of two maps which are each injective. 

    (2) For each $j\ne i$, choose $\varphi_j\in W_j^*$ with
      $\varphi_j(w_j)=1$. Identifying $V_j$ and $W_j$ with
      $k\otimes\cdots\otimes k\otimes V_j\otimes\cdots\otimes k$ and
      $k\otimes\cdots\otimes k\otimes W_j\otimes\cdots\otimes k$
      respectively, we see that if \[w_1\otimes\cdots\otimes
      w_n=(\alpha_1\otimes\cdots\otimes\alpha_n)(\vec{v}),\]then \[w_i=\alpha_i((\varphi_1\alpha_1\otimes\cdots\otimes\varphi_{i-1}\alpha_{i-1}\otimes
      1\otimes\varphi_{i+1}\alpha_{i+1}\otimes\cdots\otimes\varphi_n\alpha_n)(\vec{v})),\]as
      required.
 
(3)  
By explicitly examining the entries of $v^{\otimes n}$ in the standard basis, one sees easily that 
$v^{\otimes n}\in (\Z^m)^{\otimes n}_{\ge 0}$ implies that either $v$ or $-v$ is in $\Z^m_{\ge 0}.$ 
If $-v \in \Z^m_{\ge 0}$ then $\alpha(v), -\alpha(v) \in \Z^I_{\ge 0},$ which implies $\alpha(v) = 0,$ 
and $v=0$ as $\alpha$ is injective. 

  \end{proof}

In order to apply this result, we will make use of the following
lemma. The last assertion (allowing the determinant of $\tau$ to run over
tame characters)
will be used in Section \ref{sec:BDJ}.
\begin{lem}\label{lem: pot BT equations are well determined}
  Let $p$ be a prime, let $K/\Qp$ be a finite extension with residue
  field $k$, and let $\rbar:G_K\to\GL_2(\F)$ be a continuous
  representation. Then the system of
  equations \[e(R_\rbar^{\square,0,\tau}/\pi)=\sum_{\sigma}n_{0,\tau}(\sigma)\mu_\sigma(\rbar),\] 
in the unknowns $\mu_\sigma(\rbar)$ has at most one
solution. Equivalently,  
the linear map which sends  $(\mu_{\sigma})_{\sigma}$ to
  $(\sum_{\sigma}n_{0,\tau}(\sigma)\mu_\sigma)_{0,\tau}$
  is injective. 
In fact, this is true even if we restrict to the set of types $\tau$
for which $\det\tau$ is tame.
\end{lem}
\begin{proof}
  If $L$ is a topological field, and $G$ is a topological group, let $R_L(G)$ denote the
 Grothendieck group of continuous $L$-representations of $G$. 

The composite of reduction
  mod $p$ and semisimplification yields a homomorphism
  $\pi:R_E(\GL_2(\cO_K))\to R_{\F}(\GL_2(k))$. Since by definition
  this homomorphism takes $\sigma(\tau)=W_0\otimes\sigma(\tau)$ to
  $\sum n_{0,\tau}(\sigma)\sigma$, it is enough to check that the
  $\pi(\sigma(\tau))$ span $R_\F(\GL_2(k))$. The surjection
  $\GL_2(\cO_K)\onto\GL_2(k)$ gives an injection $R_E(\GL_2(k))\into
  R_E(\GL_2(\cO_K))$, and by Theorem 33 of Chapter 16 of
  \cite{MR0450380}, the homomorphism $R_E(\GL_2(k))\to
  R_{\F}(\GL_2(k))$ is surjective; so $\pi$ is certainly surjective.

  We recall the explicit classification of irreducible
  $E$-representations of $\GL_2(k)$, cf.\ Section 1 of
  \cite{MR2392355}. There are the one-dimensional representations
  $\chi\circ\det$, the twists $\St_\chi$ of the Steinberg
  representation, the principal series representations
  $I(\chi_1,\chi_2)$, and the cuspidal representations
  $\Theta(\xi)$. By the explicit construction in Henniart's appendix
  to \cite{breuil-mezard}, all but the representations $\St_\chi$ occur as a
$\sigma(\tau)$ for some tame type $\tau$ (the principal series
representations occur for tamely ramified types of niveau one, and the
cuspidal types for tamely ramified types of niveau two), so in order to complete
  the proof, it is enough to check that the $\pi(\St_\chi)$ are in the
  span of the $\pi(\sigma(\tau))$, where $\tau$ runs over the
  representations with tame determinant. 

  To see this, note that the reduction mod $p$ of $\St_\chi$ is just
  the irreducible representation $\chibar\circ\det\otimes
  \sigma_{p-1,\dots,p-1}$ of $\GL_2(k)$. Let $\psi:\cO_K^\times\to
  E^\times$ be a non-quadratic ramified character with trivial reduction, let
$\omegat$ denote the Teichm\"uller lift of $\omega$, and
  consider the element
  $$\sigma_\chi:=\sigma(\chi\psi\oplus\chi\psi^{-1})-\sigma(\chi\psi\omegat\oplus\chi\psi^{-1}\omegat^{-1})+\sigma(\chi\omegat\oplus\chi\omegat^{-1})-\sigma(\chi\oplus\chi)$$
  of $R_E(\GL_2(\cO_K))$. By
  Proposition 4.2 of \cite{breuildiamond}, 
  $\pi(\sigma_{\chi^{-1}})=\chibar\circ\det\otimes \sigma_{p-1,\dots,p-1}$, as
\item   required. 
\end{proof}

\begin{cor}  \label{remark: all multiplicities are 0 unless det condition holds}
Suppose a solution to the equations in Lemma \ref{lem: pot BT equations are well determined} exists.  
\begin{itemize}
\item If $\sigma|_{k^\times} \neq (\varepsilonbar\det\rbar)^{-1}\circ\Art_K,$ then $\mu_{\sigma}(\rbar) = 0.$ 
\item The $\mu_{\sigma}(\rbar)$ for which $\sigma|_{k^\times} = (\varepsilonbar\det\rbar)^{-1}\circ\Art_K$ 
are determined by the equations corresponding to the types $\tau$ with determinant
$\widetilde{\varepsilonbar\det\rbar}$.  
\end{itemize}
\end{cor}
\begin{proof} 
By 
Lemma \ref{lem: pot BT equations are well determined}, we certainly
need only consider the equations with tame determinant. Now, if
$\det\tau$ does not lift $\varepsilonbar\det\rbar$, then certainly
$R_\rbar^{\square,0,\tau}=0$. It is also easy to check that
$n_{0,\tau}(\sigma)=0$ unless the central character of $\sigma$ is
$(\det\overline{\tau})^{-1}\circ\Art_K$, so that if we set
$\mu_\sigma(\rbar)=0$ when the central character of $\sigma$ is not
equal to $(\varepsilonbar\det\rbar)^{-1}\circ\Art_K$, then all the
equations for types $\tau$ with
$\det\overline{\tau}\ne\varepsilonbar\det\rbar$ will automatically be
satisfied.  

Furthermore, none of the equations with
$\det\tau=\widetilde{\varepsilonbar\det\rbar}$ involve any of these
values of $\mu_\sigma(\rbar)$, so since the equations have a unique
solution by Lemma \ref{lem: pot BT equations are well determined}, it
follows that we must have $\mu_\sigma(\rbar)=0$ if the central
character of $\sigma$ is not equal to
$(\varepsilonbar\det\rbar)^{-1}\circ\Art_K$, and that the remaining
values of $\mu_\sigma (\rbar)$ are determined by the $\tau$ with
$\det\tau=\widetilde{\varepsilonbar\det\rbar}$. Note that if $p>2$, then it
follows from twisting that if the equations hold for all $\tau$ with
$\det\tau=\widetilde{\varepsilonbar\det\rbar}$, then in fact they hold for all
$\tau$ with $\det\overline{\tau}=\varepsilonbar\det\rbar$.
\end{proof}
\begin{rem}
  \label{rem: ref remark about central characters}A referee has pointed out
  that in fact the obvious variant of Lemma~\ref{lem: pot BT equations are well
    determined} where the types and weights have fixed central character may be
  proved via Propositions~4.2 and~4.3 of~\cite{breuildiamond} (or, in the case
  $p>2$, it can be deduced directly from the previous remark by twisting).
\end{rem}
  We may now combine Corollary \ref{cor: main patching lemma in the
    potentially diagonalizable case} with Lemmas \ref{lem: linear
    algebra lemma} and \ref{lem: pot BT equations are well determined}
  and Corollary \ref{cor: the final local-to-global
    result} to prove our main local result for potentially
  diagonalizable representations.
  \begin{thm}
    \label{thm: local BM for pot diag}Let $p>2$ be prime, let $K/\Qp$
    be a finite extension with residue field $k$, and let
    $\rbar:G_K\to\GL_2(\F)$ be a continuous representation. Then there
    are uniquely determined non-negative integers $\mu_\sigma(\rbar)$,
    $\sigma$ running over local Serre weights, with the following property: for any pair
    $(\lambda,\tau)$ with the property that every potentially
    crystalline lift of $\rbar$ of Hodge type $\lambda$ and Galois
    type $\tau$ is potentially
    diagonalizable, 
\[e(R_\rbar^{\square,\lambda,\tau}/\pi)=\sum_{\sigma}n_{\lambda,\tau}(\sigma)\mu_\sigma(\rbar).\]
  \end{thm}
  \begin{proof}
    By Corollary \ref{cor: the final local-to-global result} and
    Corollary \ref{cor: main patching lemma in the potentially
      diagonalizable case}, we see that all of the equivalent
    conditions of Lemma \ref{key patching lemma} hold in a case where
    each $F_\tv\cong K$ for each prime $v|p$ of $F,$ and each $\rbar|_{G_{F_\tv}}$ is an unramified
    twist of a representation isomorphic to our local $\rbar$. In
    particular, condition (4) of Lemma \ref{key patching lemma}
    holds. In this case, by Lemma \ref{lem: crystalline def ring
      unchanged by unramified twist} we see that 
\[e(R_\infty/\pi    R_\infty)=\prod_{v|p}e(R_\rbar^{\square,\lambda_v,\tau_v}/\pi)
= \sum_{\{\sigma_v\}_{v|p}} (\prod_{v|p} n_{\lambda_v,\tau_v}(\sigma_v))\mu'_{\sigma_{\glo}}(\rbar)
,\]
where in the final expression the sum runs over tuples $\{\sigma_v\}_{v|p}$ of equivalence classes of 
Serre weights, and $\sigma_{\glo} \cong \otimes_{v|p} \sigma_v.$

Choose an ordering of the set of equivalence classes of (local) Serre weights, 
and denote its cardinality by $m.$ Let $I$ be the set of pairs $(\lambda,\tau)$ 
in the statement of the theorem. We will apply Lemma \ref{lem: linear algebra lemma} 
to the map $\alpha: \Q^m \rightarrow \Q^I$ which sends $(\mu_{\sigma})_{\sigma}$ 
to $(\sum_{\sigma}n_{\lambda,\tau}(\sigma)\mu_\sigma)_{\lambda,\tau}.$ 
Note that all the pairs $(0,\tau)$ are in $I$ by Lemma \ref{lem: pot BT implies pot diag}, so that 
$\alpha$ is an injective map by Lemma \ref{lem: pot BT  equations are well determined}, 
and induces a map $\Z^m_{\geq 0} \rightarrow \Z^I_{\geq 0}$ since the $n_{\lambda,\tau}(\sigma)$ 
are non-negative integers.

Let $w = (e(R_\rbar^{\square,\lambda,\tau}/\pi))_{\lambda,\tau} \in \Z^I_{\geq 0},$
and take $n$ to be the number of primes of $F^+$ lying over $p.$ 
By what we saw above, there exists $v_n \in (\Z^m)^{\otimes n}_{\geq 0}$ 
such that $\alpha^{\otimes n}(v_n) = w^{\otimes n}.$ 
Hence $w = \alpha(v)$ for some $v \in \Q^m$ by  Lemma 
\ref{lem: linear algebra lemma}(2). 
As $\alpha^{\otimes n}$ is injective  Lemma \ref{lem: linear algebra lemma}(1)
implies that
we must have $v_n = v^{\otimes n},$ so that $v \in \Z^m_{\geq 0}$ by 
 Lemma \ref{lem: linear algebra lemma}(3). Defining the
 $\mu_\sigma(\rbar)$ by  $v = (\mu_{\sigma}(\rbar))_{\sigma},$ 
the theorem follows.
\end{proof}

\begin{cor}\label{cor: main result for pot BT}
   (Theorem \ref{thm: intro version of main abstract result})
    Let $p>2$ be prime, let $K/\Qp$ be a finite extension with residue
    field $k$, and let $\rbar:G_K\to\GL_2(\F)$ be a continuous
    representation. Then there are uniquely determined non-negative
    integers $\mu_\sigma(\rbar)$ such that for all inertial types
    $\tau$, we
    have \[e(R_\rbar^{\square,0,\tau}/\pi)=\sum_{\sigma}n_{0,\tau}(\sigma)\mu_\sigma(\rbar).\]
 Furthermore, the $\mu_\sigma(\rbar)$ enjoy the following properties.
\begin{enumerate}
  \item 
    $\mu_\sigma(\rbar)\ne 0$ if and only if $\sigma$ is a predicted Serre weight
    for $\rbar$.  
  \item If $K/\Qp$ is unramified and $\sigma$ is regular
    , then
    $\mu_\sigma(\rbar)=e(R_\rbar^{\square,\sigma}/\pi)$, where
    $R_\rbar^{\square,\sigma}$ is the crystalline lifting ring of Hodge type
    determined by $\sigma$. If
    furthermore $\sigma$ is Fontaine--Laffaille regular
    , then
    $\mu_\sigma(\rbar)=1$ if $\sigma$ is a predicted Serre weight for
    $\rbar$, and is $0$ otherwise.
  \end{enumerate}
\end{cor}
\begin{proof}
  The initial part is an immediate consequence of Theorem \ref{thm:
    local BM for pot diag} and Lemma \ref{lem: pot BT implies pot
    diag}.  
  The numbered
  properties also follow easily from the results above, as we now
  show. Recall that in the proof of Theorem \ref{thm: local BM for pot
    diag}, we worked with a global representation $\rbar$ with the
  property that for each $v|p,$ $\rbar|_{G_{F_\tv}},$ was an unramified twist
  of our local $\rbar$. We will continue to use this global
  representation $\rbar$. Let $\sigma$ be an equivalence class of
  local Serre weights, and let $a$ be a global Serre weight such that
  $\sigma_a\cong\sigma\otimes\cdots\otimes\sigma$. By the proof of Theorem
  \ref{thm: local BM for pot diag}, we see that
  $\mu'_\sigma(\rbar)=\mu_\sigma(\rbar)^n$, where there are $n$ places of
  $F^+$ lying over $p$. In particular, we have $\mu'_\sigma(\rbar)\ne 0$ if
  and only if $\mu_\sigma(\rbar)\ne 0$.

 (1) 
  Lemma \ref{lem: multiplicity positive implies Serre weight and thus
    crystalline lift} shows that $\mu_\sigma(\rbar)\ne 0$ if and
  only if (the global representation) $\rbar$ is modular of weight
  $\sigma\otimes\cdots\otimes\sigma$. The main result of \cite{blggU2}
  shows that whenever $\sigma$ is a predicted Serre weight for
  $\rbar$, then $\rbar$ is modular of weight
  $\sigma\otimes\cdots\otimes\sigma$. The converse holds by the main result of \cite{GLS13}.

(2) That $\mu_\sigma(\rbar)=e(R_\rbar^{\square,\sigma}/\pi)$
    whenever $\sigma$ is regular is an immediate consequence of
    Theorem \ref{thm: local BM for pot diag} and Lemma \ref{lem:
      unramified regular weight implies pot diag}(2). Now suppose that 
$\sigma$ is Fontaine--Laffaille regular. 
By Fontaine--Laffaille theory (or as a special case of the main
    result of \cite{GLS12}),
    $R_\rbar^{\square,\sigma}\ne 0$ if and only if $\sigma$ is a
    predicted Serre weight for $\rbar$, so to complete the proof it is
    enough to show that if $\sigma$ is Fontaine--Laffaille regular and
    $R_\rbar^{\square,\sigma}\ne 0$, then
    $e(R_\rbar^{\square,\sigma}/\pi)=1$. By Lemma 2.4.1 of \cite{cht}
    (see also Definition 2.2.6 of \emph{op. cit.})
    $R_\rbar^{\square,\sigma}$ is formally smooth over $\cO$, so the
    result follows.
\end{proof}
\begin{rem}
  \label{rem: if all crystalline repns are potentially diag, then multiplicities
  are the crystalline ones}If (as seems plausible) it is always the case that
any crystalline lift of $\rbar$ of Hodge type determined by $\sigma$ is in fact
potentially diagonalizable, then it would follow that
$\mu_\sigma(\rbar)=e(R_\rbar^{\square,\sigma}/\pi)$ (without any assumption on
$K$ or $\sigma$).
\end{rem}
In the next section, it will be helpful to have the following definition.
\begin{defn}
  \label{defn: WBT}Let $\WBT(\rbar)$ be the set of Serre weights
  $\sigma$ for which $\mu_\sigma(\rbar)>0$.
\end{defn}Recall from Section \ref{subsec: deformation  rings and
  types} that sets of Serre weights $\Wconj(\rbar)$ and
$\Wcris(\rbar)$ are defined in \cite{blggU2}, and that
$\Wcris(\rbar)$ is simply the set of Serre weights $\sigma$ for
which $\rbar$ has a crystalline lift of Hodge type $\sigma$. (Recall
that in the ramified case this notion depends on the choice of
particular embeddings $K\into E$; the following corollary is true for
any such choice.)
\begin{cor}
  \label{cor: relationship of WBT to Wexplicit} We have equalities
  $\Wconj(\rbar)=\WBT(\rbar)=
 \Wcris(\rbar)$.
\end{cor}
\begin{proof}The equality $\Wconj(\rbar)=\WBT(\rbar)$ is an
  immediate consequence of Corollary \ref{cor: main result for pot BT}(1) and
  the definitions of $\Wconj(\rbar)$ and $\WBT(\rbar)$. 
   By the main result of~\cite{GLS13} we have
  $\Wconj(\rbar)=\Wcris(\rbar)$, as required. 
\end{proof}

\section{The Buzzard--Diamond--Jarvis
  Conjecture}\label{sec:BDJ}

\subsection{Types}\label{subsec:types}
We now apply the machinery
developed above to the weight part of Serre's conjecture for inner
forms of $\GL_2$ (as opposed to the outer forms treated in
\cite{blggU2}, \cite{GLS11}, \cite{GLS12}, \cite{GLS13}).

We briefly recall some results from Henniart's
appendix to \cite{breuil-mezard} which will be useful to us in the
sequel. Let $l\ne p$ be prime, and let $L$ be a finite extension of
$\Ql$ with ring of integers $\cO_L$ and residue field $k_L$. Let
$\tau:I_L\to\GL_2(\Qpbar)$ be an inertial type. Then Henniart defines an irreducible
finite-dimensional representation $\sigma(\tau)$ of $\GL_2(\cO_L)$
with the following property.
\begin{itemize}
\item If $\pi$ is an infinite dimensional smooth irreducible $\Qpbar$-representation of
  $\GL_2(L)$, then $\Hom_{\GL_2(\cO_L)}(\sigma(\tau),\pi^\vee)\ne 0$ if and
  only if $r_p(\pi^\vee)|_{I_L}\cong\tau$, in which case
  $\Hom_{\GL_2(\cO_L)}(\sigma(\tau),\pi^\vee)$ is one-dimensional.
\end{itemize} (Note that in contrast to the case $l=p$ covered in
Theorem \ref{thm: Henniart existence of types}, we make no
prescription on the monodromy. The only difference between the two
definitions is for scalar inertial types, where we replace a twist of
the trivial representation with a twist of the small Steinberg representation.)
 
For later use, we need to understand the basic properties of the
reductions modulo $p$ of the $\sigma(\tau)$. We would like to thank
Guy Henniart for his assistance with the proof of the following lemma.
\begin{lem}
  \label{lem: mod p homs to types vanish implies zero} Let $\pibar$ be
  an infinite dimensional, admissible smooth, irreducible $\Fpbar$-representation of $\GL_2(L)$. Then
  there is an inertial type $\tau$ such that for any
  $\GL_2(\O_L)$-stable $\bar \Z_p$-lattice $L_{\tau} \subset \sigma(\tau)$ 
we have $(L_{\tau}\otimes_{\bar \Z_p}\pibar)^{\GL_2(\O_L)} \neq 0.$
\end{lem}
\begin{proof}By Corollaire 13 of \cite{MR1026328}, $\pibar$ may be
  lifted to an admissible smooth irreducible $\Qpbar$-representation
  $\pi$ of $\GL_2(L)$. Set $\tau=r_p(\pi^{\vee})|_{I_L}.$ Then we have
  $\Hom_{\GL_2(\cO_L)}(\sigma(\tau),\pi^\vee)\ne 0$, and hence 
$$ (\sigma(\tau)\otimes \pi)^{\GL_2(\O_L)} = \Hom_{\GL_2(\cO_L)}(\pi^\vee,\sigma(\tau))\ne 0$$ 
as the category of finite dimensional, smooth $\bar \Q_p$-representations of $\GL_2(\O_L)$ is semi-simple.
The Lemma follows.
\end{proof}
\begin{para}\label{nonsplittypes} We will also need the analogue of this result for (nonsplit)
  quaternion algebras. Let $D$ be the quaternion algebra with centre
  $L$, and let $\pi$ be a smooth irreducible (so finite-dimensional)
  $\Qpbar$-representation of $D^\times$. Let $\cO_D$ be the maximal
  order in $D$, so that $L^\times\cO_D^\times$ has index two in
  $D^\times$. Thus $\pi|_{\cO_D^\times}$ is either irreducible or a  
  sum of two distinct irreducible representations which are conjugate under a
  uniformiser in $D^\times$, and we easily see that if $\pi'$ is
  another smooth irreducible representation of $D^\times$, then $\pi$
  and $\pi'$ differ by an unramified twist if and only if
  $\pi|_{\cO_D^\times}\cong\pi'|_{\cO_D^\times}$.

Let $\tau$ be as above, and assume that $\tau$ is either irreducible
or scalar. Then there is an irreducible smooth representation $\pi_\tau$ of
$D^\times$ such that
$r_p(\operatorname{JL}(\pi_\tau))|_{I_L}\cong\tau$,
where $\operatorname{JL}$ denotes the Jacquet--Langlands
correspondence, and any two such representations differ by an
unramified twist (cf.\ Section A.1.3 of Henniart's appendix to \cite{breuil-mezard}). Define $\sigma_D(\tau)$ to be an irreducible
constituent of  $\pi_\tau^\vee|_{\cO_D^\times}$; then by the above discussion,
we have the following property. \begin{itemize}
\item If $\pi$ is a smooth irreducible $\Qpbar$-representation of
  $D^\times$ then $\Hom_{\cO_D^\times}(\sigma_D(\tau),\pi^\vee)$ is non-zero if
  and only if
  $r_p(\operatorname{JL}(\pi))|_{I_L}\cong\tau$,
  in which case $\Hom_{\cO_D^\times}(\sigma_D(\tau),\pi^\vee)$ is
  one-dimensional.
\end{itemize}
We also have the following analogue of Lemma \ref{lem: mod p homs to
  types vanish implies zero}. 
\end{para}
\begin{lem}
  \label{lem: mod p homs to types vanish implies zero quaternion
    algebra version} Let $\pibar$ be
  an admissible smooth, irreducible $\Fpbar$-representation of $D^\times$. Then
  there an inertial type $\tau$ such that for any 
$\O_D^\times$-stable $\bar \Z_p$-lattice $L_{\tau} \subset \sigma_D(\tau)$ 
we have $(L_{\tau}\otimes\pibar)^{\O_D^\times} \neq 0.$
\end{lem}
\begin{proof} By Th\'eor\`eme 4 of \cite{MR1009804}, $\pibar$ may be lifted to an
  admissible smooth, irreducible $\Qpbar$-representation $\pi$ of
  $D^\times$, and the result follows as in the proof of Lemma
  \ref{lem: mod p homs to types vanish implies zero}.
  \end{proof}

  \subsection{Deformation rings} Assume from now on that $p>2$. We now carry out our global
  patching argument. Since the arguments are by now rather standard,
  and in any case extremely similar to those of Section \ref{sec:
    patching}, we will sketch the construction, giving the necessary
  definitions and explaining the differences from the arguments of
  Section \ref{sec: patching}. For a detailed treatment of the
  Taylor--Wiles--Kisin method for Shimura curves, the reader could
  consult \cite{breuildiamond}.

\begin{para}
 For technical reasons, we will fix the
  determinants of all the deformations we will consider, and we now
  introduce some notation to allow this. We will also need to consider
  $p$-adic representations of the absolute Galois groups of $l$-adic
  fields. Accordingly, let $E/\Qp$ be a finite extension with ring of
  integers $\cO$ and residue field $\F$, let $l$ be a prime (possibly
  equal to $p$), let $K/\Ql$ be a finite extension, and let
  $\rbar:G_K\to\GL_2(\F)$ be a continuous representation. We will
  always assume that $E$ is sufficiently large that all
  representations under consideration are defined over $E$. 
Fix a finite-order character $\psi:G_K\to E^\times$ such that
  $\det\rbar=\varepsilonbar^{-1}\psibar$.  Let $\tau:I_K\to\GL_2(E)$ be
  an inertial type such that $\det\tau=\psi|_{I_K}$. Recall that we
  have the universal $\cO$-lifting ring $R_{\rbar}^\square$ of
  $\rbar$, and let $R_{\rbar}^{\square,\psi}$ denote the universal
  $\cO$-lifting ring for liftings of determinant $\psi\varepsilon^{-1}$.

Suppose firstly that $l=p$. Let
$R^{\square,\lambda,\tau,\psi}_{\rbar}$ be the unique quotient of
$R^{\square,\lambda,\tau}_{\rbar}$ with the property that the
$\Qpbar$-points of $R^{\square,\lambda,\tau,\psi}_{\rbar}$ are
precisely the $\Qpbar$-points of $R^{\square,\lambda,\tau}_{\rbar}$
whose associated Galois representations have determinant $\psi\varepsilon^{-1}$.

Suppose now that $l\ne p$. Then by Theorem 2.1.6 of
\cite{gee061}, there is a unique (possibly zero) $p$-torsion free
reduced quotient $R_{\rbar}^{\square,\tau,\psi}$ of
$R_{\rbar}^{\square}$ whose $\Qpbar$-points are precisely those points
of $R_{\rbar}^{\square}$ which correspond to liftings of $\rbar$ which
have determinant $\psi\varepsilon^{-1}$ and Galois type $\tau$, in the sense that the
restriction to $I_K$ of the corresponding Weil--Deligne
representations are isomorphic to $\tau$. 
\end{para}

\begin{rem}\label{rem: fixing det doesn't change e}By Lemma
4.3.1 of \cite{emertongeerefinedBM} we have an
isomorphism \[R^{\square,\lambda,\tau}_{\rbar}\cong
R^{\square,\lambda,\tau,\psi}_{\rbar}[[X]],\]so that in particular we have $e(R^{\square,\lambda,\tau,\psi}_{\rbar}/\pi)=
e(R^{\square,\lambda,\tau}_{\rbar}/\pi)$.  
\end{rem}

\begin{para}
We begin with a generalisation of Corollary 3.1.7 of \cite{gee061}
(see also Th\'eor\`eme 3.2.2 of \cite{breuildiamond}), which 
produces modular liftings of mod $p$ Galois representations
with prescribed local properties, and will be used in place of
Corollary \ref{cor: main patching lemma in the potentially
  diagonalizable case} in this setting. The proof is very similar to
that of Corollary \ref{cor: main patching lemma in the potentially
  diagonalizable case}. We firstly define some notation.

 Let $F$ be a totally real field, and let
  $\rhobar:G_F\to\GL_2(\F)$ be a continuous representation. We will
  say that $\rhobar$ is \emph{modular} if it is isomorphic to the reduction
  mod $p$ of the Galois representation associated to a Hilbert modular
  eigenform of parallel weight two. 
  Fix a totally even, finite order character $\psi:G_F\to E^\times$ with the
  property that $\det\rhobar=\psibar\varepsilonbar^{-1}$. Let $S$ be a
  finite set of finite places of $F$, including all places dividing
  $p$ and all places at which $\rhobar$ or $\psi$ is ramified. For
  each place $v\in S$, fix an inertial type $\tau_v$ of $I_{F_v}$ such
  that $\det\tau_v=\psi|_{I_{F_v}}$. For each place $v|p$, let $R_v$
  be a quotient of
  $R_{\rhobar|_{G_{F_v}}}^{\square,0,\tau_v,\psi|_{G_{F_v}}}$
  corresponding to a choice of an irreducible component of
  $\Spec R_{\rhobar|_{G_{F_v}}}^{\square,0,\tau_v,\psi|_{G_{F_v}}}[1/p]$,
  and for each place $v\in S$ with $v\nmid p$, let $R_v$ be a quotient
  of $R_{\rhobar|_{G_{F_v}}}^{\square,\tau_v,\psi|_{G_{F_v}}}$
  corresponding to a choice of an irreducible component of
  $\Spec R_{\rhobar|_{G_{F_v}}}^{\square,\tau_v,\psi|_{G_{F_v}}}[1/p]$.
\end{para}

\begin{para} We assume from now on that $\bar\rho$ satisfies the following conditions. 
  \begin{itemize}
  \item $\rhobar$ is modular,
  \item $\rhobar|_{G_{F(\zeta_p)}}$ is absolutely irreducible, and
  \item if $p=5$, the projective image of $\rhobar|_{G_{F(\zeta_p)}}$ is not isomorphic
    to $A_5$.
  \end{itemize}
\end{para}

\begin{lem}\label{lem: existence of lifts of prescribed type for HMFs} 
For $\bar\rho$ and $R_v$ as above, there is a continuous lift $\rho:G_F\to\GL_2(E)$ of $\rhobar$ such that:
\begin{itemize}
\item $\det\rho=\psi\varepsilon^{-1}$.
\item $\rho$ is unramified outside of $S$.
\item For each place $v\in S$, $\rho|_{G_{F_v}}$ arises from a point of $R_v[1/p]$.
\item $\rho$ is modular.
\end{itemize}
\end{lem}
\begin{proof} This is essentially an immediate
  consequence of Corollary 3.1.7 of \cite{gee061}, given the main
  result of \cite{blggordII}. Indeed, in the case that for all places
  $v|p$ the component $R_v$ corresponds to non-ordinary
  representations, the lemma is a special case of Corollary 3.1.7 of \cite{gee061}, and
  the only thing to be checked in general is that the hypothesis (ord) of
  Proposition 3.1.5 of \cite{gee061} is satisfied. This hypothesis
  relates to the existence of ordinary lifts of $\rhobar$, and a
  general result on the existence of such lifts is proved in
  \cite{blggordII}. 

In fact, examining
  the proof of Proposition 3.1.5 of \cite{gee061}, we see that it is
  enough to check that there is some finite solvable extension $L/F$
  of totally real fields with the property that $L$ is linearly
  disjoint from $\overline{F}^{\ker\rhobar}(\zeta_p)$ over $F$, and
  $\rhobar|_{G_L}$ has a modular lift which is potentially
  Barsotti--Tate and ordinary at all places $v|p$. In order to see
  that such an $L$ exists, simply choose a solvable totally real
  extension $L/F$ such that $L$ is linearly disjoint from
  $\overline{F}^{\ker\rhobar}(\zeta_p)$ over $F$ and
  $\rhobar|_{G_{L_w}}$ is reducible for each place $w|p$ of $L$. The
  existence of the required lift of $\rhobar|_{G_L}$ is then immediate
  from the main result of \cite{blggordII}.
\end{proof}

\begin{para}\label{par: def ring setup} Now let $\Sigma, \Sigma' \subset S$ be disjoint subsets, not containing 
any places dividing $p.$ Suppose that if $v\in \Sigma$, then $\tau_v$ is either irreducible or scalar, 
and if $v \in \Sigma'$ then $\tau_v$ is scalar. 

For each $v\in S$, $v\nmid p$ we define
$R_{\rhobar|_{G_{F_v}}}^{\square,\tau_v,\psi,*}$ as follows:
$R_{\rhobar|_{G_{F_v}}}^{\square,\tau_v,\psi,*}=R_{\rhobar|_{G_{F_v}}}^{\square,\tau_v,\psi}$
unless $v\in \Sigma\cup \Sigma'$ and $\tau_v$ is scalar. If $\tau_v$ is scalar, 
and $v \in \Sigma$ (resp. $v \in \Sigma'$) then we define 
$R_{\rhobar|_{G_{F_v}}}^{\square,\tau_v,\psi,*}$ as the quotient of
$R_{\rhobar|_{G_{F_v}}}^{\square,\tau_v,\psi}$
corresponding to the components of
$R_{\rhobar|_{G_{F_v}}}^{\square,\tau_v,\psi}[1/p]$ whose 
$\Qpbar$-points are not all potentially unramified 
(resp. are all potentially unramified). Note that at this 
stage we could have
$R_{\rhobar|_{G_{F_v}}}^{\square,\tau_v,\psi,*}=0.$ For example if $v \in \Sigma$ and $\tau_v$
is scalar and $R_{\rhobar|_{G_{F_v}}}^{\square,\tau_v,\psi,*}\ne 0$,
then $\rhobar|_{G_{F_v}}$ is necessarily a twist of an extension of 
the trivial character by the mod $p$ cyclotomic character. 

By \cite{pillonidef} Theorems 4.1.1, 4.1.3, 
$\Spec R_{\rhobar|_{G_{F_v}}}^{\square,\tau_v,\psi}[1/p]$ is the union of formally 
smooth irreducible components, and two such components $C_1,$ $C_2$ can have a non-trivial 
intersection only in the following situation: $\tau_v$ is scalar, 
and there is a character $\gamma: G_{F_v} \rightarrow E^\times$ such that (up to exchanging $C_1, C_2$) 
the representations parameterized by $x \in C_1(\bar \Q_p)$ are unramified after twisting by $\gamma^{-1},$ 
and the representations parameterized by $x \in C_2(\bar \Q_p)$ are extensions of $\gamma$ by $\gamma(1),$ which 
are ramified if $x \in C_2\backslash C_1(\bar \Q_p).$ 

We set $R^\psi_S=\widehat{\otimes}_{v\in S,\cO}
R_{\rhobar|_{G_{F_v}}}^{\square,\psi}$, $R_p^{\tau,\psi}=\widehat{\otimes}_{v|p,\cO}
R_{\rhobar|_{G_{F_v}}}^{\square,0,\tau_v,\psi}$, $R_{S\setminus
  p}^{\tau,\psi}=\widehat{\otimes}_{v\nmid p,v\in S,\cO}
R_{\rhobar|_{G_{F_v}}}^{\square,\tau_v,\psi,*}$ and
$R_S^{\tau,\psi}=R_p^{\tau,\psi}\widehat{\otimes}_\cO R_{S\setminus
  p}^{\tau,\psi}$. (Note that $R_S^{\tau,\psi}$ is analogous to the ring
$R^\loc$ defined in \S\ref{subsubsec:assumption on adequate}.)
\end{para}

\subsection{Modular forms}\label{subsec:modularforms}

 We continue to use the notation and assumptions introduced in the last subsection.

\begin{para} Let $D$ be quaternion algebra with centre $F$. We assume that either $D$ is
ramified at all infinite places (the definite case) or that $D$ is split at
precisely one infinite place (the indefinite case), and that the set of finite primes at which $D$ 
is ramified is precisely the set $\Sigma$ introduced above. If $F=\Q$, assume
further that $D\ne M_2(\Q)$. (We make this assumption only in order to give a
uniform definition of our spaces of modular forms; with the usual modifications
to handle the non-compactness of modular curves, the case $D=M_2(\Q)$ could also
be treated by our methods. Since our main results are all well-known in this
case, we do not comment further on this assumption below.)

Fix a maximal order
$\cO_D$ of $D$, and an isomorphism $(\cO_D)_v\isoto M_2(\cO_{F_v})$ at
each finite place $v\notin\Sigma$. For each finite place $v$ of $F$ we
let $\pi_v\in F_v$ denote a uniformiser.

By a {\em global Serre weight for} $D^\times$ we mean an absolutely irreducible mod $p$ representation $\sigma$
of $\prod_{v|p} (\O_D)_v^\times$ considered up to equivalence. Using the isomorphisms $(\cO_D)_v\isoto M_2(\cO_{F_v})$ 
for $v| p,$ each such $\sigma$ has the form $\sigma \cong \otimes_{v|p}\sigma_v,$ where 
$\sigma_v$ is a local Serre weight for $\GL_2(k_v),$ $k_v$ the residue field of $v.$
For the remainder of this section when we say ``global Serre weight'' we mean a global Serre weight for $D^\times.$
\end{para}

\begin{para} 
We now define the spaces of modular forms
with which we will work, following section 3 of \cite{breuildiamond},
to which we refer for a more detailed treatment of the facts that we
need, and for further references to the literature. 

By Lemma  4.11 of \cite{MR1605752} we can and do choose a finite place $v_1\notin S$ of $F$ such
that $\rho$ is unramified at $v_1$, $\mathbf{N}v_1\not\equiv
1\pmod{p}$, and the ratio of the eigenvalues of $\rhobar(\Frob_{v_1})$
is not equal to
$(\mathbf{N}v_1)^{\pm 1}$. In particular, we have
$H^0(G_{F_{v_1}},\ad\rhobar(1))=(0)$, and it follows that any lifting
of $\rhobar|_{G_{F_{v_1}}}$ is necessarily unramified. Furthermore, we
can and do assume that the residue characteristic of $v_1$ is
sufficiently large that for any non-trivial root of unity $\zeta$ in a
quadratic extension of $F$, $v_1$ does not divide $\zeta+\zeta^{-1}-2$.

Fix from now on the compact
open subgroup $U=\prod_v U_v\subset
(D\otimes_\Q\A^\infty)^\times$ where $U_v=(\cO_D)_{v}^\times$ for $v\ne
v_1$, and $U_{v_1}$ is the subgroup of $\GL_2(\cO_{F,{v_1}})$
consisting of elements which are upper triangular unipotent modulo
$v_1$. As in the case of unitary groups, the assumption on $v_1,$ 
is used (implicitly) below, to guarantee that $M_{\infty}$ is a free $S_{\infty}$ module
(cf.~e.g the condition (2.1.2) in \cite{kisinfmc}). 

For each place $v\in S$, we have an inertial type $\tau_v$, and thus a
finite-dimensional irreducible $E$-representation $\sigma(\tau_v)$ of
$(\cO_D)^\times_v$. (When $v|p$ the representation $\sigma(\tau_v)$ is
defined by Theorem \ref{thm: Henniart existence of types}, and when
$v\nmid p$ it is defined in Section \ref{subsec:types}, where it was
denoted $\sigma_{D_v}(\tau_v)$ in the case that $D_v$ is ramified.)

Now define a representation of $(\cO_D)^\times_v$ on a finite free $\cO$-module $L_{\tau_v},$ 
as follows: If $v \notin \Sigma',$ then we choose  $L_{\tau_v}$ to be a 
$(\cO_D)^\times_v$-stable $\cO$-lattice in $\sigma(\tau_v).$ If $v \in \Sigma',$ then $\tau_v$ 
is scalar, by assumption, and 
we set $L_{\tau_v} = \O,$ equipped with an action of $(\cO_D)^\times_v$ given by $\tau_v^{-1}\circ\Art_{F_v}\circ\det.$
We write $L_{\tau^p}=\otimes_{v\in S,v\nmid
  p,\cO}L_{\tau_v}$, $L_{\tau_p}=\otimes_{v|p,\cO}L_{\tau_v}$. 
\end{para}

\begin{para}\label{para:modularformsnotn} 
Let $\theta$ denote a finite $\cO$-module with a
continuous action of 
$\prod_{v|p}U_v$, with the property that the action of
$\prod_{v|p}\cO_{F_v}^\times\subset \prod_{v|p}U_v$ on $\theta$ is given by
$\psi\circ\Art_F$. Then $\theta\otimes_\cO L_{\tau^p}$ has an action of
$U$ via the projection onto $\prod_{v\in S}U_v$. We extend this action
to an action of $U(\A_F^\infty)^\times$ by letting
$(\A_F^\infty)^\times$ act via the composition of the projection
$(\A_F^\infty)^\times\to (\A_F^\infty)^\times/F^\times$ and
$\psi\circ\Art_F$. (This action is well-defined by our assumptions on
$\psi$ and $\theta$.) 

Let $V\subset U$ be a compact open subgroup, and 
suppose firstly that we are in the indefinite case. Then there is a smooth projective algebraic
curve $X_V$ over $F$ associated to $V$, and a local system
$\cF_{\theta\otimes_\cO L_{\tau^p}}$ on $X_{V}$ corresponding to
$\theta\otimes_\cO L_{\tau^p}$, and we
set \[S_{\tau^p}(V,\theta):=H^1(X_{U,\Qbar},\cF_{\theta\otimes_\cO
  L_{\tau^p}}).\] 

If we are in the definite case, then we let $S_{\tau^p}(V,\theta)$ be the set of continuous functions \[f:D^\times\backslash
(D\otimes_F\A_F^\infty)^\times\to \theta\otimes_\cO L_{\tau^p}\]such that we have
$f(gu)=u^{-1}f(g)$ for all $g\in (D\otimes_F\A_F)^\times$, $u\in V(\A_F^\infty)^\times$.

In the special case that
$\theta=L_{\tau_p}$, we write $S_\tau(V,\cO)$ for
$S_{\tau^p}(V,L_{\tau_p})$.

For the precise relationship between these spaces and automorphic
forms on $D^\times$, see Section 3.1.14 of~\cite{kis04} (in the
definite case) and Section 3.5 of~\cite{breuildiamond} (in the
indefinite case).
\end{para}

\begin{para} We now define the Hecke algebras that we will use. Let $T$ be a finite
set of finite places of $F$ containing $ S,$ and such that 
$(\O_D)_v^\times \subset V$ for $v \notin T.$ Let $\T^{T,\univ}$ be
the commutative $\cO$-polynomial algebra generated by formal variables
$T_v,S_v$ for each finite place $v\notin T\cup\{v_1\}$ of $F$. Then
$\T^{T,\univ}$ acts on $S_{\tau^p}(V,\theta)$  via
the Hecke operators \[T_v=\left[\GL_2(\cO_{F,v})
\begin{pmatrix}
  \pi_v& 0\\0&1
\end{pmatrix}\GL_2(\cO_{F,v})\right],\]  \[S_v=\left[\GL_2(\cO_{F,v})
\begin{pmatrix}
  \pi_v& 0\\0&\pi_v
\end{pmatrix}\GL_2(\cO_{F,v})\right].\] 
We denote the image of $\T^{T,\univ}$ in $\End_\cO(S_\tau(V,\cO))$ by
$\T_\tau^T(V,\cO)$. 

Let $\m$ be the maximal ideal of
$\T^{S,\univ}$ with residue field $\F$ with the property that for each
finite place $v\notin S\cup\{v_1\}$ of $F$, the characteristic
polynomial of $\rhobar(\Frob_v)$ is equal to the image of
$X^2-T_vX+(\mathbf{N}v)S_v$ in $\F[X]$. Note that if $S_\tau(U,\cO)_\m\ne 0$, then
$S_\tau(U,\cO)_\m\otimes_{\Zp}\Qp$ is locally free of rank $2^m$ over
$\T_\tau^S(U,\cO)_\m[1/p]$, where $m=1$ in the definite case and $2$
in the indefinite case. (This follows from our assumptions on 
$v_1$ and $U$, and the multiplicity one property of the
$\sigma(\tau_v)$ explained in Theorem \ref{thm: Henniart existence of
  types} and Section \ref{subsec:types}.)
\end{para}

\begin{para}  Let $R_{F,S}^{\psi}$ be the universal deformation ring for deformations of $\bar\rho$ with determinant 
$\psi\varepsilon^{-1}$ which are unramified outside $S$. Then $S_\tau(U,\cO)_\m$ is naturally a $R^{\psi}_{F,S}$-module. In particular, a Hecke eigenform in 
$S_\tau(U,\cO)_\m$ gives rise to a deformation $\rho: G_{F,S} \rightarrow \GL_2(\O)$ of $\bar\rho,$ with determinant $\psi\varepsilon^{-1}.$
\end{para}

\begin{lem}\label{lemma:existenceofquaternionicforms} For each $v \in  S$ and $v| p$ (resp $v\nmid p$) 
let $R_v$ be a quotient of $R_{\rhobar|_{G_{F_v}}}^{\square,0,\tau_v,\psi}$ 
corresponding to an irreducible component of $\Spec R_{\rhobar|_{G_{F_v}}}^{\square,0,\tau_v,\psi}[1/p]$ 
(resp. $\Spec R_{\rhobar|_{G_{F_v}}}^{\square,\tau_v,\psi,*}$). 
(In particular we are assuming that the set of such components is non-empty for each $v \in S$).
Then there is a continuous lift $\rho:G_F\to\GL_2(\O)$ of $\rhobar$ such that:
\begin{itemize}
\item $\det\rho=\psi\varepsilon^{-1}$.
\item $\rho$ is unramified outside of $S$.
\item For each place $v\in S$, $\rho|_{G_{F_v}}$ arises from a point of $R_v[1/p]$.
\item $\rho$ arises from a Hecke eigenform in $S_{\tau}(U,\O).$
\end{itemize}
\end{lem}
\begin{proof} By Lemma \ref{lem: existence of lifts of prescribed type for HMFs} 
(note that we are already assuming the hypotheses of that lemma are satisfied), there exists 
a $\rho$ with the first three properties, which is modular in the sense that it arises from a Hilbert 
modular form $\pi.$ The conditions on $\rho$ at $v|p$ imply that this form is parallel of weight $2.$

We now check that $\pi$ is discrete series at all $v \in \Sigma,$ so that $\pi$ transfers to an automorphic form $\pi^D$ on $D.$ 
Suppose that $\pi_v$ is not discrete series for some $v \in \Sigma.$ 
Since we are assuming that $\tau_v$ is either scalar or irreducible, $\pi_v$ must be a twist of 
an unramified principal series. By local-global compatibility, this implies that a twist of 
$\rho|_{G_{F_v}}$ is unramified and pure (i.e satisfies the Ramanujan conjecture) by~\cite{MR2327298} Thm 1. 
But, this contradicts the assumption that $\rho|_{G_{F_v}}$ arises from a point of 
$\Spec R_{\rhobar|_{G_{F_v}}}^{\square,\tau_v,\psi,*}$ (cf.\ the description in \ref{par: def ring setup}).

Finally local-global compatibility for $D$ implies that $\pi^D$ corresponds to an eigenform in  $S_{\tau}(U,\O).$
\end{proof}

\subsection{Patching}
\begin{para} 
Let $R^{\square,\psi}_{F,S}$ denote the
complete local $\cO$-algebra which pro-represents the functor which
assigns to a local Artinian $\cO$-algebra $A$ the set of equivalence
classes of tuples $(\rho,\{\alpha_v\}_{v\in S})$ where $\rho$ is a
lifting of $\rhobar$ as a $G_{F,S}$-representation with
determinant $\psi\varepsilon^{-1}$,
$\alpha_v\in\ker(\GL_2(A)\to\GL_2(\F))$, and two such tuples
$(\rho,\{\alpha_v\}_{v\in S})$ and $(\rho',\{\alpha'_v\}_{v\in S})$
are equivalent if there is an element
$\beta\in\ker(\GL_2(A)\to\GL_2(\F))$ with $\rho'=\beta\rho\beta^{-1}$
and $\alpha'_v=\beta\alpha_v$ for all $v\in S$, so that
$R^{\square,\psi}_{F,S}$ is naturally an $R_S^\psi$-algebra. We
define $R^{\square,\tau,\psi}_{F,S}=R^{\square,\psi}_{F,S}\otimes_{R_S^\psi}R^{\tau,\psi}_S$. We also have the corresponding
universal deformation ring $R^{\tau,\psi}_{F,S},$ defined as the image of $R^{\psi}_{F,S}$ 
in $R^{\square,\tau,\psi}_{F,S}.$

Choose a lift $\rho_S^{\univ}:G_{F,S}\to\GL_2(R_{F,S}^{\tau,\psi})$
representing the universal deformation. Fix a place $w\in S$, and let
\[ \mc{T} = \mc{O}[[X_{v,i,j}:v \in S, i,j = 1,2]]/(X_{w,1,1}).\] The
tuple $(\rho^{\univ}_{S},(1_2+X_{v,i,j})_{v \in S})$ induces an
isomorphism $R^{\square,\tau,\psi}_{F,S} \isoto
R^{\tau,\psi}_{F,S}\widehat{\otimes}_{\mc{O}} \mc{T}.$ 
(Note that $1_2+X_{v,i,j}$ is a $2\times 2$ matrix, and the
fact that $X_{w,1,1}=0$ in $\cT$ implies that this tuple has no
non-trivial scalar endomorphisms.) 

We let $M=
S_{\tau}(U,\mc{O})_{\mf{m}}$. 
Fix a filtration by $\F$-subspaces \[0=L_0\subset
L_1\subset\dots\subset L_s=(\otimes_{v|p,\cO}L_{\tau_v})\otimes_\cO \F\]such that each
$L_i$ is $\prod_{v|p}U_v$-stable, and for each $i=0,1,\dots,s-1$,
$\sigma_i:=L_{i+1}/L_i$ is absolutely irreducible. This in turn
induces a filtration 
\[0=M^{0}\subset
M^{1}\subset\dots\subset
M^{s}=M^{}\otimes_\cO
\F.\]
on $M^{}\otimes_\cO\F,$ whose
graded pieces are the finite-dimensional $\F$-vector spaces
$S_{\tau^p}(U,\sigma_i)_\mf{m}$. 

Let $q \geq [F:\Q]$ be an integer, and set $g=q-[F:\Q]+|S|-1,$ 
\begin{eqnarray*}
  \Delta_{\infty} & = & \bb{Z}_p^q, \\
  R_{\infty} & = & R^{\tau,\psi}_S[[x_1,\ldots,x_g]], \\
  R'_{\infty} & = & R^{\psi}_S[[x_1,\ldots,x_g]], \\
  S_{\infty} & = & \mc{T}[[\Delta_{\infty}]], 
\end{eqnarray*}
and let $\mf{a}$ denote the kernel of the $\mc{O}$-algebra homomorphism $S_{\infty}
\rightarrow \mc{O}$ which sends each $X_{v,i,j}$ to 0 and each element
of $\Delta_{\infty}$ to 1. Note that $S_\infty$ is formally smooth
over $\bigO$ of relative dimension $q+4|S|-1$ and that $R_\infty$ also
has relative dimension $q+4|S|-1$ over $\cO.$ 

Now a patching argument as in Section \ref{subsec:patching} shows for some $q \geq [F:\Q],$ there exists 
\begin{itemize}
\item  an $\cO$-module homomorphism $S_\infty\to R_\infty$, and an
  $R_\infty$-module $M_\infty$ which is finite free as an $S_{\infty}$-module,
\item a filtration by $R_{\infty}$-modules 
\[0=M_{\infty}^{0}\subset M_{\infty}^{1}\subset\dots\subset M_{\infty}^{s}=M_{\infty}^{}\otimes \F\] 
whose graded pieces are finite free $S_{\infty}/\pi S_{\infty}$-modules,
\item a surjection of $R_S^{\tau,\psi}$-algebras $R_{\infty}/\mf{a}R_{\infty}
  \rightarrow R^{\tau,\psi}_{F,S}$, and
\item an isomorphism of $R_{\infty}$-modules $M_{\infty}/\mf{a}M_{\infty} \iso M$ which identifies 
 $M^i$ with $M^i_{\infty}/\mf{a}M^i_{\infty}.$
\end{itemize}

As in Section \ref{subsec:patching}, we may and do assume that for
$i=1,2,\dots s,$ the $(R'_{\infty},S_{\infty})$-bimodule
$M^i_{\infty}/M^{i-1}_\infty$ and the isomorphism 
$M^i_{\infty}/(\mf{a}M^i_{\infty}+M^{i-1}_\infty) \iso M^i/M^{i-1}$
depends only on ($U,\mf{m}$, the types $\tau_v$ for $v\nmid p$ and)
the isomorphism class of $L_i/L_{i-1}$ as a
$\prod_{v|p}U_v$-representation, but not on the choice of types
$\tau_v$ for $v|p$. 
We denote this $(R'_{\infty},S_{\infty})$-bimodule 
by $M^{\sigma}_{\infty},$ where $\sigma \iso L_i/L_{i-1}.$
\end{para}

\begin{para} Assume that for $v\in S$, $v\nmid p$ the $\tau_v$ and the set $\Sigma'$ have been chosen
such that $R_{\rhobar|_{G_{F_v}}}^{\square,\tau_v,\psi,*}\ne 0$, and
that for $v|p$ we have
$\det\tau_v=\widetilde{\varepsilon\det\rhobar|_{G_{F_v}}}.$ 
We set 
$$ \mu_\sigma'(\rhobar) = \frac{1}{2^m\prod_{v\in S,v\nmid
    p}e(R_{\rhobar|_{G_{F_v}}}^{\square,\tau_v,\psi,*}/\pi)}
e_{R_{\infty}/\pi}(M^\sigma_\infty).$$
We need the following variant of Lemma \ref{key patching lemma}.
\end{para}

\begin{lem}\label{key patching lemma II} We have 
\begin{enumerate}
\item The support of $M$ meets every irreducible component of $\Spec R^{\tau,\psi}_S[1/p].$
\item $M_{\infty}\otimes_{\Z_p}\Q_p$ is a finite $R_{\infty}[1/p]$-module, which 
which is locally free of rank $2^m$ over the formally smooth locus in 
$\Spec R_{\infty}[1/p].$
\item 
$R^{\tau,\psi}_{F,S}$ is a finite $\O$-algebra and $M\otimes_{\Z_p}\Q_p$ is a faithful $R^{\tau,\psi}_{F,S}$-module.  
\item 
\[e(R_{\infty}/\pi R_{\infty}) =  \sum_{i=1}^s 2^{-m}e_{R_{\infty}/\pi}(M^{\sigma_i}_{\infty}) \]
where $\sigma_i$ is a global Serre weight with $L_i/L_{i-1} \iso \sigma_i.$ 
\end{enumerate}
\end{lem}
\begin{proof} (1) follows from Lemma \ref{lemma:existenceofquaternionicforms}. 

For (2), note that $M_{\infty}$ is a finite $R_{\infty}$-module, since it is a finite $S_{\infty}$-module, 
and the same argument as in  Lemma~\ref{key patching lemma} shows that $M_{\infty}$ is locally 
free over the formally smooth locus of $\Spec R_{\infty}[1/p].$ Thus to show (2), we have to show that 
any irreducible component of $Z \subset \Spec R_{\infty}[1/p]$ contains a closed point in the smooth locus 
of $\Spec R_{\infty}[1/p]$ at which $M_{\infty}$ has rank $2^m.$ 

By (1), there exists $\mathfrak p \in Z$ in the support of $M.$ We claim that the image of 
$\mathfrak p$ is in the smooth locus of $\Spec R^{\tau,\psi}_S[1/p].$ Assuming this, we see that 
$\mathfrak p$ is in the smooth locus of $\Spec R_{\infty}[1/p],$ and in particular $M_{\infty}$ is locally free (of positive rank) at $\mathfrak p.$ 
Hence the maps 
$$ R_{\infty}/\mf{a}R_{\infty}[1/p] \rightarrow R^{\tau,\psi}_{F,S}[1/p] \rightarrow \T_\tau^S(U,\cO)_\m[1/p]$$
become isomorphisms after localising at $\mathfrak p.$ Since $M\otimes \Q_p$ has rank $2^m$ over 
$\T_\tau^S(U,\cO)_\m[1/p]$ it follows that $M_{\infty}$ has rank $2^m$ at $\mathfrak p.$

By the description of $R_{\rhobar|_{G_{F_v}}}^{\square,\tau_v,\psi,*}$ given in \ref{par: def ring setup}, to show that the image of
$\mathfrak p$ is in the smooth locus of $\Spec R^{\tau,\psi}_S[1/p],$ it suffices to show that if $v \in S\backslash \Sigma,$ $v\nmid p,$ and 
$\tau_v$ is scalar, then the image of $\mathfrak p$ lies on exactly one irreducible component of $\Spec R^{\square,\tau_v,\psi}_{\rhobar|_{G_{F_v}}}[1/p].$
Let $\rho_{\mathfrak p}$ denote the representation of $G_{F,S}$ corresponding to the image of $\mathfrak p$ in $\Spec R^{\tau,\psi}_{F,S}.$
Since $\mathfrak p$ is in the support of $M,$ it corresponds to an automorphic representation $\pi$ of $D^\times,$ and $\pi_v$ 
is a twist either of an unramified principal series, or of a Steinberg representation.  
In the former case, one sees as in the proof of Lemma \ref{lemma:existenceofquaternionicforms} that $\rho_{\mathfrak p}|_{G_{F_v}}$ 
is the twist of an unramified representation which satisfies the Ramanujan conjecture. So $\mathfrak p$ is contained in a single component of  
$\Spec R^{\square,\tau_v,\psi}_{\rhobar|_{G_{F_v}}}[1/p],$ and this component parameterizes potentially unramified representations. 
When $\pi_v$ is a twist of a Steinberg representation, then it follows from \cite{carayol86}, Thm~A, that the image of $\rho_{\mathfrak p}|_{I_{F_v}}$ 
is infinite, and again $\mathfrak p$ lies on exactly one component of $\Spec R^{\square,\tau_v,\psi}_{\rhobar|_{G_{F_v}}}[1/p].$

This proves (2) and also shows that $M\otimes \Q_p$ is a finite free $R_{\infty}/\mf{a}R_{\infty}[1/p]$-module of rank $2^m.$ 
In particular, the map $R_{\infty}/\mf{a}R_{\infty}[1/p] \rightarrow R^{\tau,\psi}_{F,S}[1/p]$ is an isomorphism, so $M\otimes \Q_p$
is a finite free $R^{\tau,\psi}_{F,S}[1/p]$-module, which proves (3).

Finally, by (2) $M_{\infty}$ is locally free of rank $2^m$ at the generic points of $R_{\infty}.$ By \cite{kisinfmc} Lemma 1.3.4. this implies that 
$e(R_{\infty}/\pi) = 2^{-m} e_{R_{\infty}/\pi}(M_{\infty}/\pi M_{\infty}),$ and (4) follows as in the proof of Lemma \ref{key patching lemma}.
\end{proof}

\begin{prop}\label{key patching prop for BDJ} For each global Serre weight $\sigma,$ we have
  $\mu'_\sigma(\rhobar)=\prod_{v|p}\mu_{\sigma_v}(\rhobar|_{G_{F_v}})$,
  where $\sigma \cong\otimes_{v|p}\sigma_v$, and the
  $\mu_{\sigma_v}(\rhobar|_{G_{F_v}})$ are as in Theorem~\ref{thm: local BM for pot diag}.
In particular $\mu'_\sigma(\rhobar)$ is a non-negative integer.

\end{prop}
\begin{proof} By Lemma \ref{key patching lemma II} we have 
    \[e(R_{\infty}/\pi R_{\infty}) = 2^{-m}\sum_{i=1}^s e_{R_{\infty}/\pi}(M^{\sigma_i}_{\infty})\] where $\sigma_i$ is a global Serre weight with
    $L_i/L_{i-1} \iso \sigma_{i}.$ 
Since 
$$e(R_\infty/\pi R_\infty)=e(R_S^{\tau,\psi}/\pi)=\prod_{v|p}e(R_{\rhobar|_{G_{F_v}}}^{\square,0,\tau_v,\psi}/\pi)\prod_{v\in
  S,v\nmid p}e(R_{\rhobar|_{G_{F_v}}}^{\square,\tau_v,\psi,*}/\pi),$$
this
yields \[\prod_{v|p}e(R_{\rhobar|_{G_{F_v}}}^{\square,0,\tau_v,\psi}/\pi)=\sum_{i=1}^s\mu'_{\sigma_i}(\rhobar),\]which
by Remark \ref{rem: fixing det doesn't change e} is equivalent
to \[\prod_{v|p}e(R_{\rhobar|_{G_{F_v}}}^{\square,0,\tau_v}/\pi)=\sum_{i=1}^s\mu'_{\sigma_i}(\rhobar).\]
Since the $\sigma_i$ are precisely the Jordan--H\"older factors of
$(\otimes_{v|p,\cO}L_{\tau_v})\otimes_\cO\F$, we see that this is
equivalent
to \[\prod_{v|p}e(R_{\rhobar|_{G_{F_v}}}^{\square,0,\tau_v}/\pi)=\sum_{\{\sigma_v\}_{v|p}}
(\prod_{v|p} n_{0,\tau_v}(\sigma_v))\mu'_{\sigma}(\rhobar) ,\] where the
sum runs over tuples $\{\sigma_v\}_{v|p}$ of equivalence classes of
Serre weights, and $\sigma \cong \otimes_{v|p} \sigma_v.$ By the above
construction, these equations hold for all choices of types $\tau_v$
with $\det\tau_v=\widetilde{\varepsilon\det\rhobar|_{I_{F_v}}}$. 

In fact, we claim that these equations hold for all choices of types
$\tau_v$ with tame determinant. In order to see this, suppose that for
some $v$, $\det\tau_v$ is tame, but $\det\tau_v \neq \widetilde{\varepsilon\det\rhobar|_{I_{F_v}}}$ and so 
$\det\taubar_v\ne\varepsilonbar\det\rhobar|_{I_{F_v}}$. Then
$R_{\rhobar|_{I_{F_v}}}^{\square,0,\tau_v}=0$, so it suffices to prove
that $\mu'_{\sigma}(\rhobar)=0$ whenever $n_{0,\tau_v}(\sigma_v)\ne 0$ for all $v\mid p$ 
(as then the equation will collapse to $0=0$). Equivalently, we must show
that $M^\sigma_\infty=0$ if $n_{0,\tau_v}(\sigma_v)\ne 0$; but as in Corollary 
\ref{remark: all multiplicities are 0 unless det condition holds},
this is an easy consequence of local-global compatibility and the
assumption that $\det\taubar_v\ne\varepsilonbar\det\rhobar|_{I_{F_v}}$,
since if $M^\sigma_\infty\ne 0$ then $M^\sigma\ne 0$.

    Now, by Lemma \ref{lem: linear algebra lemma}(1) and Lemma \ref{lem: pot
      BT equations are well determined}, we see that
    the quantities $\mu'_\sigma(\rhobar)$ are uniquely determined by these
    equations as the $\tau_v$ run over all types with tame determinant. However, by Corollary \ref{cor: main result for pot BT}
    we see that setting
    $\mu'_\sigma(\rhobar)=\prod_{v|p}\mu_{\sigma_v}(\rhobar|_{G_{F_v}})$
    gives a solution to the equations, so we must have
    $\mu'_\sigma(\rhobar)=\prod_{v|p}\mu_{\sigma_v}(\rhobar|_{G_{F_v}})$,
    as required.
\end{proof}

The following corollary is the main result we need to apply
our techniques to the weight part of Serre's conjecture.
\begin{cor}
  \label{cor: Serre weights in the setting of the patching
    argument}With the above notation and assumptions, let
  $\sigma=\otimes_{v|p}\sigma_v$ be a global Serre weight. Then we
  have $S_{\tau^p}(U,\sigma)_\m\ne 0$ if and only if
  $\sigma_v\in\WBT(\rhobar|_{G_{F_v}})$ for all places $v|p$.
\end{cor}
\begin{proof}
  By Proposition \ref{key patching prop for BDJ}, we see that
  $\sigma_v\in\WBT(\rhobar|_{G_{F_v}})$ for all places $v|p$ if and
  only if $e_{R_\infty/\pi}(M_\infty^\sigma)\ne 0$. Since (by the patching construction) $M_\infty^\sigma$ is
  maximal Cohen--Macaulay over $R_\infty/\pi$, we see from Theorem 13.4
  of~\cite{MR1011461} that $e_{R_\infty/\pi}(M_\infty^\sigma)\ne 0$ if and only if $M^\sigma_\infty\ne 0$; but
  $M^\sigma_\infty\ne 0$ if and only if $M^\sigma\ne 0$, and
  $M^\sigma=S_{\tau^p}(U,\sigma)_\m$ by definition.
\end{proof}
\subsection{Proof of the BDJ conjecture} We are now ready to prove Theorem \ref{thm: intro version of
  BDJ}, which amounts to translating Corollary \ref{cor: Serre weights in the setting of the patching
    argument} into the original formulation of the BDJ conjecture.

\begin{para} 
Let $D$ and $\psi$ be as above and, as usual, denote by $\Sigma$
the set of finite places where $D$ is ramified. Let $V=(\prod_{v|p}\GL_2(\cO_{F_v}))V^p$ with
$V^p$ a compact open subgroup of $(D\otimes_\Q\A^{\infty,p})^\times.$ 
We assume that $\psi\circ \det|_{V^p}$ is trivial. 
Let $\sigma=\otimes_{v|p}\sigma_v$ be a global Serre weight for $D.$
We assume that $\sigma$ is compatible with $\psi$ as above, and extend $\sigma$ to a
representation of $V(\A_F^\infty)^\times.$ Let $\tau^p_0$ denote the trivial $\O$-representation 
of $V^p.$ We write $S(V,\sigma) = S_{\tau^p_0}(V,\sigma),$ and we remind the reader that 
the definition of $S(V,\sigma)$ depends on $\psi.$

Let $T$ be a finite set of finite places of $F,$ containing the set of places where 
$D,$ $\rhobar$ or $\psi$ are ramified and all the places dividing $p,$ and such that 
$(\O_D)_v^\times \subset V$ for $v \notin T.$  
We again denote by $ \m \subset \T^{T,\univ}$ the maximal ideal
associated to $\rhobar,$ as above.
The Hecke algebra $\T^{T,\univ}$ acts on $S(V,\sigma)$ and we may consider the 
localization $S(V,\sigma)_{\m}.$ Note that $S(V,\sigma)_{\m}$ does not depend on the choice 
of the set $T$ satisfying the above conditions. Thus, we may write $S(V,\sigma)_{\m}$ without 
making a choice of $T.$
\end{para}

\begin{defn}
  \label{defn: modular of some weight}We say that $\rhobar$ is modular
  for $D$ of weight $\sigma,$ if for some $V$ as above we have
  $S(V,\sigma)_\m\ne 0$.
\end{defn}

\begin{defn}
We say that $\rhobar$ is \emph{compatible with D} if for all $v \in \Sigma$, 
$\rhobar|_{G_{F_v}}$ is either irreducible, or is a twist of an
extension of the trivial character by the cyclotomic character.
\end{defn}

\begin{cor}
  \label{cor: main BDJ theorem} Let $p>2$ be
  prime, let $F$ be a totally real field, and let
  $\rhobar:G_F\to\GL_2(\Fpbar)$ be a continuous representation. Assume
  that $\rhobar$ is modular, that $\rhobar|_{G_{F(\zeta_p)}}$ is
  irreducible, and if $p=5$ assume further that the projective image
  of $\rhobar|_{G_{F(\zeta_p)}}$ is not isomorphic to $A_5$.

  For each place $v|p$ of $F$ with residue field $k_v$, let $\sigma_v$
  be a Serre weight of $\GL_2(k_v)$. Then $\rhobar$ is modular for $D$ of 
  weight $\sigma = \otimes_{v|p}\sigma_v$ if and only if $\rhobar$ is compatible with $D$ and $\sigma_v\in\WBT(\rhobar|_{G_{F_v}})$ for all $v|p$.
\end{cor}
\begin{proof} 
Suppose firstly that $\rhobar$ is compatible with $D$ and $\sigma_v\in\WBT(\rhobar|_{G_{F_v}})$  for all $v|p.$ 
Let $S$ be a set of primes containing all the primes dividing $v|p,$ and all primes where $\rhobar, \bar\psi$ or $D$ ramify. 
We take the set $\Sigma' \subset S$ to be empty. Since $\rhobar$ is compatible with $D,$ 
for each $v \in \Sigma,$ $\rhobar|_{G_{F_v}}$ has a lift of discrete series type, as in the proof of Corollaire
3.2.3 of \cite{breuildiamond}, and we can choose our types $\tau_v$ for $v \in S,$ $v\nmid p$ such that
  $R_{\rhobar|_{G_{F_v}}}^{\square,\tau_v,\psi,*}\ne 0.$ 

Choose a compact open subgroup $U \subset (D\otimes_\Q\A^\infty)^\times,$ and the 
lattices $L_{\tau_v} \subset \sigma(\tau_v)$ as in \ref{subsec:modularforms} above.
Let $V = (\prod_{v|p}\GL_2(\cO_{F_v}))V^p \subset U$ be a normal subgroup  such that 
$V$ acts trivially on $\sigma(\tau_v)$ for all $v \in S,$ $v \nmid p,$ and $\psi\circ\det|_{V^p}$ is trivial. 
Then 
\setcounter{equation}{4}
\begin{equation}\label{eqn:invariants}
S_{\tau^p}(U,\sigma)_{\m} = S_{\tau^p}(V,\sigma)_{\m}^{U/V} = (S(V,\sigma)_{\m}\otimes_{\O}L_{\tau^p})^{U/V}.
\end{equation}
Since $\sigma_v\in\WBT(\rhobar|_{G_{F_v}})$  for all $v|p,$ 
we have $S_{\tau^p}(U,\sigma)_\m\ne 0$ by Corollary 
\ref{cor: Serre weights in the setting of the patching argument}, and
$S(V,\sigma)_{\m}\neq 0,$ as required.

Conversely, suppose that $S(V,\sigma)_{\m}\neq 0$ for some $V.$ 
Let $S$ be a set of primes containing the primes $v|p,$ the primes where $\rhobar, \bar\psi$ or $D$ ramify, 
and the primes $v$ such that $(\O_D)_v^\times$ is not contained in $V.$ 
Let $S^p = \{v \in S: v\nmid p\}.$
Write $V^{S^p} = \prod_{v \notin S^p} V_v,$ and set 
$$ Y = \varinjlim_{W_{S^p}} S(W_{S^p}V^{S^p},\sigma)_{\m}\otimes_{\mathbb F} \bar{\mathbb F}_p$$
where $W_{S^p}$ runs over compact open subgroups of $\prod_{v \in S^p} V_v.$
Then $Y$ is naturally a smooth $\bar{\mathbb F}_p$-representation of the group $D_{S^p}^\times = \prod_{v \in S^p} D_v^\times.$ 
Note that $Y \neq 0$ as $S(V,\sigma)_{\m}\neq 0.$
Let $\bar \pi \subset Y$ be an irreducible subrepresentation. Then $\bar \pi = \otimes_{v \in S^p} \bar \pi_v$ 
where each $\bar \pi_v$ is an irreducible, admissible, smooth representation of $D_v^\times.$ 
Let $\Sigma' \subset S^p$ be the subset where $D$ is unramified at $v$ and $\bar\pi_v$ is finite dimensional. 

We now define the $(\O_D)_v^\times$-representations $L_{\tau_v},$ for $v \in S^p,$ used in the definition of our spaces of modular forms.  
If either $\bar\pi_v$ is infinite dimensional or $D$ is ramified at $v,$ then we take $\tau_v$ and $L_{\tau_v}$ to be the type 
and $(\O_D)_v^\times$-representation produced by applying Lemmas 
\ref{lem: mod p homs to types vanish implies zero} and \ref{lem: mod p homs to types vanish implies zero quaternion algebra version} 
respectively to the representation $\bar \pi_v.$
If $D$ is unramified at $v,$ and $\bar\pi_v$ is finite dimensional then, by \cite{MR1026328} Thm.~1.1,
$\bar\pi_v$ has the form $\bar\chi_v\circ\det,$ where $\bar\chi_v: F_v^\times \rightarrow \bar\F_p^\times$ is a continuous character. 
Since $\bar\chi_v$ has finite image, we may lift $\bar\chi_v|_{\O_{F_v}^\times}$ to a character $\chi_v: \O_{F_v}^\times \rightarrow \bar \Z_p^
\times.$ 
We take $\tau_v$ the $2$-dimensional scalar representation given by $\chi_v \circ \Art_{F_v}^{-1},$ 
and $L_{\tau_v} = \O$ with $(\O_D)_v^\times$ acting via $\chi_v^{-1}\circ\det.$
Since $Y$ has central character $\bar\psi,$ so does $\bar\pi$ and each $\bar\pi_v$ for $v \in S^p.$ 
Thus $\det \tau_v$ reduces to $\bar\psi$ mod $p.$ Since $p\neq 2,$ after replacing each $\tau_v$ by a twist 
by a character which has trivial reduction, we may assume that $\det\tau_v = \psi|_{I_{F_v}}$ for $v \in S^p.$
We take $\Sigma' \subset S^p$ to be the set of primes where $D$ is unramified and $\bar\pi_v$ is finite dimensional.

Set $\O_{D_{S^p}^\times} = \prod_{v\in S^p} (\O_D)_v^\times.$ By construction, $(Y\otimes L_{\tau^p})^{\O_{D_{S^p}^\times}} \neq 0.$ 
Hence there is an open normal subgroup $W_{S^p} \subset \O_{D_{S^p}^\times},$ such that 
$(S(W_{S^p}V^{S^p}, \sigma)_{\m}\otimes L_{\tau^p})^{\O_{D_{S^p}^\times}} \neq 0.$
Let $U = \prod_v U_v$ be as in \ref{subsec:modularforms}, so in particular $U_v \subset (\O_D)_v^\times.$  
By what we just saw, there exists an normal open subgroup $W = (\prod_{v|p}\GL_2(\cO_{F_v}))W^p \subset U$ 
such that $W$ acts trivially on $\sigma(\tau_v)$ for $v \in S^p\backslash\Sigma',$ $\tau_v\circ\Art_{F_v}$ 
is trivial on $W$ if $v \in \Sigma',$ and $\psi\circ\det|_{W^p}$ is trivial, and such that 
$(S(W, \sigma)_{\m}\otimes L_{\tau^p})^{U/W} \neq 0.$   
Then $S_{\tau^p}(U,\sigma)_{\m} \neq 0$ by (\ref{eqn:invariants}), applied with $W$ in place of $V.$
This implies that $\sigma_v\in\WBT(\rhobar|_{G_{F_v}})$  for all $v|p,$ 
by Corollary \ref{cor: Serre weights in the setting of the patching argument}. 

It remains to see that $\bar\rho$ is compatible with $D.$ Choose a representation of $\prod_{v\mid p} (\O_D)_v^\times$ 
on a finite free $\O$-module $\theta$ such that $\sigma$ is a Jordan-H\"older factor of $\theta\otimes_{\O}\mathbb F.$
Then $S_{\tau^p}(U,\sigma)_{\m} \neq 0$ implies $S_{\tau^p}(U,\theta)_{\m} \neq 0.$ This implies that for $v \in \Sigma,$ $\bar\rho|_{G_{F_v}}$ has a lift 
of discrete series type, and hence that $\bar\rho$ is compatible with $D,$ as in the proof of Corollaire 3.2.3 of \cite{breuildiamond}. 
\end{proof}
\setcounter{thm}{5}
\begin{remark}
  The above arguments strongly suggest that there should be versions
  of the Breuil--M\'ezard conjecture for quaternion algebras, and also
  for mod $l$ representations of the absolute Galois groups of
  $p$-adic fields, with $l\ne p$. The first problem is considered in~\cite{1309.0019}, and the
  second problem in~\cite{1309.1600}.
\end{remark}
Following \cite{bdj} we say that $\bar\rho$ is modular of weight $\sigma$ if 
there exists a quaternion algebra $D$ over $F$ which is indefinite at a single prime $v_0|\infty$ 
and split at all primes dividing $p,$ and such that $\rhobar$ is modular for $D$ of weight $\sigma.$ 
(In the definition of \cite{bdj} the prime $v_0$ is fixed. However, 
it follows from the Jacquet-Langlands correspondence that the condition does not in fact depend on the choice of $v_0$).
As a consequence of Corollary \ref{cor: main BDJ theorem} we have the following result on the BDJ conjecture 
as formulated in \cite{bdj} Conj.~3.14. 

\begin{cor}(Theorem \ref{thm: intro version of BDJ}) 
Let $p,F,\rhobar$ and $\sigma$ be as in \ref{cor: main BDJ theorem}. Then $\rhobar$ is modular of weight $\sigma$ 
if and only if $\sigma_v\in\WBT(\rhobar|_{G_{F_v}})$ for all $v|p$.
\end{cor}
\begin{proof} The ``only if''' direction follows immediately from Corollary \ref{cor: main BDJ theorem}, which also implies the 
converse once we show that there is a $D$ indefinite at a single prime $v_0|\infty$ 
and split at all primes dividing $p$ such that $\rhobar$ is compatible with $D.$ 
For this, we may take $D$ to be a quaternion algebra ramified at one infinite prime, and at one 
finite prime $v,$ such that $\rhobar(\Frob_v) = 1$ and $\mathbf{N}(v) = 1$ modulo $p.$
\end{proof}
\appendix
\section{Realising local representations globally}\label{sec:local to  global}
In this appendix we realise local representations globally, using the
potential automorphy techniques of \cite{BLGGT} and \cite{frankII}. We
will freely use the notation and terminology of \cite{BLGGT}, in
particular the notions of RACSDC, RAESDC and RAECSDC automorphic
representations $\pi$ of $\GL_n(\A_{F^+})$ and $\GL_n(\A_F)$, and the
associated $p$-adic Galois representations $r_{p,\imath}(\pi)$, which
are defined in Section 2.1 of \cite{BLGGT}.

\begin{aprop}\label{prop: local realisation over totally real field}
  Let $K/\Qp$ be a finite extension, and let
  $\rbar_K:G_K\to\GL_2(\Fpbar)$ be a continuous representation. Then
  there is a totally real field $L^+$ and a continuous irreducible
  representation $\rbar:G_{L^+}\to\GL_2(\Fpbar)$ such that
  \begin{itemize}
  \item for each place $v|p$ of $L^+$, $L^+_v\cong K$ and
    $\rbar|_{L^+_v}\cong\rbar_K$,
\item for each finite place $v\nmid p$ of $L^+$,
    $\rbar|_{G_{L^+_v}}$ is unramified,
  \item for each place $v|\infty$ of $L^+$, $\det\rbar(c_v)=-1$, where
    $c_v$ is a complex conjugation at $v$, and
  \item there is a non-trivial finite extension $\F/\F_p$ such that $\rbar(G_{L^+})=\GL_2(\F)$.
  \end{itemize}

\end{aprop}
\begin{proof}
  Choose a non-trivial finite extension $\F/\F_p$ such that
  $\rbar_K(G_K)\subset\GL_2(\F)$. We apply Proposition 3.2 of
  \cite{frankII} where, in the notation of that
  result,
  \begin{itemize}
  \item $G=\GL_2(\F)$,
  \item we let $F=E$ be any totally real
  field with the property that if $v|p$ is a place of $E$, then
  $E_v\cong K$,
\item if $v|p$, we let $D_v=\rbar_K(G_K)$,
\item if $v|\infty$, we let $c_v$ be the nontrivial conjugacy class of order 2 in $\GL_2(\F)$.
  \end{itemize}
This produces a representation $\rbar$ which satisfies all the
required properties, except possibly for the
  requirement that $\rbar$ be unramified outside $p$, which may be
  arranged by a further (solvable, if one wishes) base change. (Note that while
  this is not stated there, the
  commutative diagram in of Proposition 3.2(4) of
  \cite{frankII} is the obvious one, where the top arrow is the natural inclusion.)
\end{proof}

We remark that the reason that we have assumed that $\F\ne\F_p$ is
that we wish to apply Proposition A.2.1 of \cite{blggU2} to conclude
that $\rbar(G_{L^+})$ is adequate. We have the following now-standard potential
modularity result.
\begin{athm}\label{thm: potential modularity for GL2}
  Let $L^+$ be a totally real field, let $M/L^+$ be a finite
  extension, and let $p>2$ be a prime. Let
  $\rhobar:G_{L^+}\to\GL_2(\Fpbar)$ be an irreducible continuous totally odd
  representation. Then there exists a finite totally real Galois
  extension $F^+/L^+$ such that
  \begin{itemize}
  \item the primes of $L^+$ above $p$ split completely in $F^+$,
  \item $F^+$ is linearly disjoint from $M$ over $L^+$, and
  \item there is a weight $0$ RAESDC (regular, algebraic, essentially
    self-dual, cuspidal) automorphic representation $\pi$ of
    $\GL_2(\A_{F^+})$ such that $\rbar_{p,\imath}(\pi)\cong
    \rhobar|_{G_{F^+}}$, and $\pi$ has level potentially prime to $p$.
  \end{itemize}

\end{athm}
\begin{proof}
  This is a special case of Proposition 8.2.1 of \cite{0905.4266},
  once one knows that for each place $v|p$, $\rhobar|_{G_{L^+_v}}$
  admits a potentially Barsotti--Tate lift. (Such a lift will be of
  type A or B in the terminology of \cite{0905.4266}.) If $\rhobar|_{G_{L^+_v}}$
  is irreducible, this is proved in the course of the proof of
  Proposition 7.8.1 of \cite{0905.4266}, and if it is reducible then
  it is an immediate consequence of Lemma 6.1.6 of \cite{blggord}.
\end{proof}
Combining Proposition \ref{prop: local realisation over totally real
  field} and Theorem \ref{thm: potential modularity for GL2}, we obtain the following result.
\begin{acor}
  \label{cor: local modular realisation over totally real field} 
Let 
  $p>2$ be prime, let $K/\Qp$ be a finite extension, and let
  $\rbar_K:G_K\to\GL_2(\Fpbar)$ be a continuous representation. Then
  there is a totally real field $F^+$ and a RAESDC automorphic representation $\pi$ of
  $\GL_2(\A_{F^+})$ such that $\rbar_{p,\imath}(\pi)$ is absolutely
  irreducible, and:
  \begin{itemize}
  \item for each place $v|p$ of $F^+$, $F^+_v\cong K$ and $\rbar_{p,\imath}(\pi)|_{F^+_v}\cong\rbar_K$,
  \item for each finite place $v\nmid p$ of $F^+$,
    $\rbar_{p,\imath}(\pi)|_{G_{F^+_v}}$ is unramified,
  \item $\pi$ has level potentially prime to $p$, and
  \item there is a non-trivial finite extension $\F/\F_{p}$ such that $\rbar_{p,\imath}(\pi)(G_{F^+})=\GL_2(\F)$.
  \end{itemize}

\end{acor}
\begin{proof}
  This follows at once by applying Theorem \ref{thm: potential
    modularity for GL2} to the representation $\rbar$ provided by
  Proposition \ref{prop: local realisation over totally real field},
  taking the auxiliary field $M$ to be $(\overline{L})^{\ker\rbar}$.
\end{proof}

\begin{athm}\label{thm: local modular mod p realisation over CM}
  Suppose that $p>2$, that $K/\Qp$ is a finite extension, and let
  $\rbar_K:G_K\to\GL_2(\Fpbar)$ be a continuous representation. Then
  there is an imaginary CM field $F$ with maximal totally real
  subfield $F^+$ and a RACSDC automorphic representation $\pi$ of
  $\GL_2(\A_F)$ such that  $\rbar_{p,\imath}(\pi)$ is absolutely
  irreducible, and:
  \begin{itemize}
  \item each place $v|p$ of $F^+$ splits in $F$,
  \item for each place $v|p$ of $F^+$, $F^+_v\cong K$,
  \item for each place $v|p$ of $F^+$, there is a place $\tv$ of $F$
    lying over $v$ such that $\rbar_{p,\imath}(\pi)|_{G_{F_\tv}}$ is isomorphic to an
    unramified twist of $\rbar_K$,
  \item $\rbar_{p,\imath}(\pi)$ is unramified outside of the places dividing $p$,
  \item $\zeta_p\notin F$,
  \item $\rbar_{p,\imath}(\pi)({G_{F(\zeta_p)}})$ is adequate, and
  \item the projective image of $\rbar$ is not isomorphic to $A_4$.
    \end{itemize}

\end{athm}
\begin{proof}
  Let $\pi$ be the RAESDC automorphic representation of
  $\GL_2(\A_{F^+})$ provided by
  Corollary \ref{cor: local modular realisation over totally real
    field}. Choose a totally imaginary quadratic extension $F/F^+$
  which is linearly disjoint from $(\overline{F})^{\ker\rbar_{p,\imath}(\pi)}(\zeta_p)$
  over $F$ and in which all places of $F^+$ lying over $p$
  split. For each place $v|p$ of $F^+$, choose a place $\tv$ of $F$
  lying over $v$. Choose a finite order character $\theta:G_F\to\Qpbartimes$
  such that for each place $v|p$ of $F^+$,
  $\theta|_{G_{F_\tv}}=1$, and $\theta|_{G_{F_{\tv^c}}}=\varepsilon\det r_{p,\imath}(\pi)|_{G_{F^+_v}}$. (The existence of such
  a character is guaranteed by Lemma 4.1.1 of \cite{cht}.) Let
  $\theta_{F^+}$ denote $\theta$ composed with the transfer map
  $G_{F^+}^\ab\to G_F^\ab$, so that $\theta_{F^+}|_{G_F} = \theta \theta^c$.

  Choose a finite order character $\psi:G_F\to\Qpbartimes$ such
  that \[\psi\psi^c=\varepsilon^{-1}(\theta_{F^+}(\det r_{p,\imath}(\pi))^{-1})|_{G_F}\]
  and such that each $\psi|_{G_{F_\tv}}$ is unramified.  (The existence of
  such a character follows by applying Lemma 4.1.5 of
  \cite{cht} to
  $\varepsilon^{-1}(\theta_{G_{F^+}}\det\rbar_{p,\imath}(\pi))^{-1}|_{G_F}$,
  choosing the integers $m_\tau$ of \emph{loc.~cit.} to be $0$. Note
  that by virtue of the
  choice of $\theta$ above, the character $\varepsilon^{-1}(\theta_{F^+}(\det
  r_{p,\imath}(\pi))^{-1})|_{G_F}$ is crystalline at all places of $F$
  dividing $p$.)

Then the representation $r:=r_{p,\imath}(\pi)|_{G_F}\otimes\psi\theta^{-1}$
satisfies $r^c\cong r^\vee\varepsilon^{-1}$, so by base change there is
a RACSDC automorphic representation of $\GL_2(\A_F)$ satisfying all
the required properties, except possibly the property that
$\rbar_{p,\imath}(\pi)$ is unramified outside of $p$, which may be
arranged by a further solvable base change. (Note that the claims that
$\rbar_{p,\imath}(\pi)({G_{F(\zeta_p)}})$ is adequate and that the projective image of $\rbar$ is not isomorphic to $A_4$ are an immediate
consequence of Proposition \ref{prop:adequacy for n=2} and the fact
that $\rbar_{p,\imath}(\pi)(G_F)=\GL_2(\F)$ for some non-trivial
extension $\F/\F_p$.)
\end{proof}
\begin{acor}
  \label{cor: the final local-to-global result} Suppose that $p>2$,
  that $K/\Qp$ is a finite extension, and let
  $\rbar_K:G_K\to\GL_2(\Fpbar)$ be a continuous representation. Then
  there is an imaginary CM field $F$ and a continuous irreducible
  representation $\rbar:G_F\to\GL_2(\Fpbar)$ such that
  \begin{itemize}
  \item each place $v|p$ of $F^+$ splits in $F$, and has $F^+_v\cong K$,
  \item for each place $v|p$ of $F^+$, there is a place $\tv$ of $F$
    lying over $F^+$ with $\rbar|_{G_{F_\tv}}$ isomorphic to an
    unramified twist of $\rbar_K$,
  \item $[F^+:\Q]$ is even,
  \item $F/F^+$ is unramified at all finite places,
  \item $\zeta_p\notin F$,
  \item $\rbar$ is unramified outside of $p$,
  \item $\rbar$ is automorphic in the sense of Section \ref{sec: Galois repns}, 
  \item $\rbar(G_{F(\zeta_p)})$ is adequate, and
  \item the projective image of $\rbar$ is not isomorphic to $A_4$.
  \end{itemize}

\end{acor}
\begin{proof}
  This follows immediately from Theorem \ref{thm: local modular mod p realisation
    over CM} and the theory of base change between $\GL_2$ and unitary
  groups, cf.\ section 2 of \cite{blggU2} (note that we can
  make a finite solvable base change to ensure that $[F^+:\Q]$ is even
  and $F/F^+$ is unramified at all finite places without affecting the
  other conditions).
\end{proof}

\section{Errata for \cite{kisinfmc}}\label{erratakisinfmc}

In this appendix we give some errata for \cite{kisinfmc}.  
We are very grateful to Kevin Buzzard, Wansu Kim, Vytautas Paskunas, and Fabian Sander for pointing these out.
Unless otherwise mentioned all references are to \cite{kisinfmc}, and we freely use the notation of that paper.
\begin{apara} In (1.1.6), the multiplicity $\mu_{n,m}(\rhobar)$ when 
$\rhobar \sim \left(\smallmatrix \mu_{\lambda} & * \\ 0 & \mu_{\lambda}\endsmallmatrix \right)\otimes\omega^m$
(that is $n=p-2$ and $\lambda = \lambda'$) is not correctly defined in some cases. Namely, this multiplicity is defined to be $1$ if (the class of) $*$ is non-trivial. 
If $*$ is trivial, the multiplicity is not defined explicitly, but there is a speculation that it is $2.$ 
The actual multiplicities have been computed by Fabian Sander \cite{1212.4978}, and are $1$ when $*$ is non-trivial and ramified, 
$2$ when $*$ is non-trivial and unramified, and $4$ when $*$ is trivial. 
For the purposes of correcting the argument in \cite{kisinfmc}, $\mu_{n,m}(\rhobar)$ should be defined implicitly whenever $n=p-2$
and $\lambda = \lambda',$ as was done in the case when $*$ is trivial.

The mistake in the argument occurs in the last paragraph of (1.7.5), where it is claimed that ``we have just seen that (1.7.6) is a closed immersion''. 
The argument only shows that (1.7.6) is a homeomorphism onto its image, which is an isomorphism over generic points, but it can have some non-reduced 
fibres.

To correct this the last claim in the statement of (1.7.5), regarding the formal smoothness of $R^{\ord}$ should be deleted, 
and in the statement of (1.7.14) ``unless $\rhobar \sim \left(\smallmatrix \mu_{\lambda} & 0 \\ 0 & \mu_{\lambda'}\endsmallmatrix \right)\otimes\omega^m$'' should be replaced by ``unless $\rhobar \sim \left(\smallmatrix \mu_{\lambda} & * \\ 0 & \mu_{\lambda'}\endsmallmatrix \right)\otimes\omega^m$ and either the class of $*$ is trivial, or $\lambda = \lambda'.$'' This change has no effect on any of the arguments which follow and, in particular, does not effect the statement of the main theorems.
\end{apara} 

\begin{apara} In the definition of the functor in Lemma (1.7.4) add the condition that $G_{\Q_p}$ acts on $L_A\otimes_A\F$ via $\omega_1.$
\end{apara}

\begin{apara} In (2.2) add the condition that $\NN(v) \neq  -1 (p)$ for $v \in \Sigma.$ Otherwise the character $\gamma_v$ in (2.2.10) may fail to be unique. In the applications one can reduce to this case by base change, and even assume that $\NN(v) = 1 (p)$ for $v \in \Sigma.$
\end{apara}

\begin{apara} The proof of Proposition (2.2.15) implicitly uses that for $v \in \Sigma,$ the rings $\bar R_v^{\sq,\psi}/\pi$ are irreducible, and generically reduced. The proof of this is practically identical to the proof of Lemma (1.7.5), using the formal smoothness proved in Lemma 2.6.3 of 
\cite{kis04}, and the fact that the representation at a generic point of $\Spec \bar R_v^{\sq,\psi}/\pi$ cannot be scalar. 
\end{apara}

\begin{apara} Lemma (2.2.1) is not correct; there are some $\rhobar$ with small image which provide counterexamples. 
The mistake is in the first line of the second paragraph of the proof, where it is asserted that if $g'$ satisfies (2.2.2) then so does $gg'.$ 
Replacing $g$ by $gg'$ does not change the left hand side of (2.2.2), but may change the right hand side.

This lemma is used to show that the conditions imposed on the compact open subgroup $U = \prod_v U_v$ and the 
set of primes where $U_v$ is not maximal compact, can be satisfied in the application to the main theorem (2.2.18). 
More precisely, the conditions are (2.1.2): that 
$$ (U(\A_F^f)^\times\cap t D^\times t^{-1})/F^\times = 1.$$
for all $t \in (D\otimes_F\A_F^f)^\times,$ and (2.2)(4): that if $U_v$ is not maximal compact then 
$1 - \NN(v) \in \F^\times,$ and that the ratio of the eigenvalues $\rhobar(\Frob_v)$ is not in $\{1, \NN(v), \NN(v)^{-1}\}.$
We explain two ways to correct this, one which works when $p > 3,$ and one which works in general:

If $p > 3,$ one can require in all of \S 2 that $p$ split completely in $F$ (which is in any case assumed after (2.2.10) and in the main theorem), 
drop the condition (2.1.2), and require that $U_v$ is maximal compact for all $v$ (so that (2.2)(4) is vacuous). 
Indeed (2.1.2) is needed only to show in (2.1.4) that $S_{\sigma,\psi}(U_{\Delta},A)_{\mathfrak m}$ is a free $A[\Delta]$-module. 
The proof given there works provided that the finite group $(U(\A_F^f)^\times\cap t D^\times t^{-1})/F^\times$ has prime to $p$ order. 
But this condition can fail only if $F(\zeta_p)$ is a quadratic extension of $F,$ which implies $p=3,$ since $p$ splits completely in $F.$

To correct the argument for any $p,$ maintain the condition (2.1.2), but require in (2.2)(4) only that $1 - \NN(v) \in \F^\times,$ and that the ratio of the eigenvalues $\rhobar(\Frob_v)$ is not $\NN(v)^{\pm 1},$ which is the usual condition, for which the existence of $v$ is 
guaranteed by \cite{MR1605752}, Lemma 4.11. Let 
$R \subset S\backslash \Sigma_p$ be the set of primes such that the eigenvalues of $\rhobar(\Frob_v)$ are equal. Then 
$S_{\sigma,\psi}(U,\O)\otimes_{\Z_p}\Q_p$ is a $\mathbb{T}_{\sigma,\psi}(U)_{\m}[1/p]$-module of rank $2^{|R|}.$
\footnote{Alternatively, we could, as in the present paper, omit the operators $U_{\pi_w},$ $w \in S\backslash \Sigma_p$ from the definition 
of $\mathbb{T}_{\sigma,\psi}(U)_{\m}$ in (2.1.5), in which case the rank would be $2^{|R|}$ with $R = S\backslash \Sigma_p.$} 
We now explain the modifications necessary to the proof which involve, as in the present paper, keeping track of some factors 
of $2^{|R|}.$

Replace Lemma (2.2.11) with 
\end{apara}

\begin{alemma}\label{corrected2.2.11} Let $\mathfrak p \in \Spec \bar R_{\infty}$ be a prime in the support of $M_{\infty}.$ 
Then $$\dim_{\kappa(\mathfrak p)} M_{\infty}\otimes_{\bar R_{\infty}}\kappa(\mathfrak p) \geq 2^{|R|},$$
with equality if $\mathfrak p$ is a minimal prime of $\bar R_{\infty}.$ 
Moreover, the following conditions are equivalent. 
\begin{enumerate}
\item $M_{\infty}$ is a faithful $\bar R_{\infty}$-module 
\item $M_{\infty}$ is a faithful $\bar R_{\infty}$-module of rank $2^{|R|}$ at all generic points of $R_{\infty}.$
\item $e(\bar R_{\infty}/\pi \bar R_{\infty}) = 2^{-|R|}e(M_{\infty}/\pi M_{\infty}, \bar R_{\infty}/\pi \bar R_{\infty}).$
\item $e(\bar R_{\infty}/\pi \bar R_{\infty}) \leq 2^{-|R|}e(M_{\infty}/\pi M_{\infty}, \bar R_{\infty}/\pi \bar R_{\infty}).$
\end{enumerate}
If these conditions hold, and $\rho: G_{F,S} \rightarrow \GL_2(\O)$ is a deformation of $\bar\rho$ 
such that for $v \in \Sigma_p,$ $\rho|_{I_v}$ is an extension of $\gamma_v$ by $\gamma_v(1)$ if $v\nmid p,$ and $\rho|_{G_{F_v}}$ is 
potentially semi-stable of type $(k,\tau_v,\psi)$ if $v|p,$ then $\rho$ is modular, 
and arises from an eigenform in $S_{\sigma, \psi}(U,\O)\otimes_{\O}E.$
\end{alemma}
\begin{proof} Let $\O\lps \Delta_{\infty} \rps = \O\lps y_1,\dots, y_{h+j}\rps,$ as in the proof of Lemma (2.2.11). 
Since $M_{\infty}$ is finite flat over $\O\lps \Delta_{\infty} \rps,$ its support in $\Spec \bar R_{\infty}$ is a union of irreducible 
components, and any such irreducible component $Z \subset \Spec \bar R_{\infty}$ surjects onto $\Spec \O\lps \Delta_{\infty} \rps.$

To prove the first statement, it suffices to prove it in the case of minimal primes of the support of $M_{\infty},$ which follows once we show that $M_{\infty}$ is free of rank 
$2^{|R|}$ over the smooth locus of any $Z$ as above. The argument for this is analogous to that in the Lemma \ref{key patching lemma II} 
of the present paper: It suffices to show that a prime $\mathfrak p \in Z$ which maps to $(y_1,\dots,y_{h+j})$ in $\Spec \O\lps \Delta_{\infty}\rps$ 
is a smooth point of $Z,$ or equivalently that $\mathfrak p$ maps to a smooth point of $\Spec \bar R^{\sq,\psi}_v[1/p]$ for each $v|p.$ 

Let $\WD(\rho_{\mathfrak p}|_{G_{F_v}})$ denote the Weil-Deligne representation
attached to $\rho_{\mathfrak p}|_{G_{F_v}},$
and let $\pi$ be the automorphic representation of $D^\times$ corresponding to $\mathfrak p.$
If $\pi_v$ is a twist of an unramified principal series then $\WD(\rho_{\mathfrak p}|_{G_{F_v}})$ is a twist of a pure, unramified representation 
by \cite{MR2327298} and \cite{MR0332791}
\footnote{Patrick Allen has pointed out that the smoothness in this case may be deduced from the fact that $\pi$ and hence $\pi_v$ has a Whittaker model, without using the Ramanujan conjecture.}.
 If $\pi_v$ is a twist of a Steinberg representation, then $\WD(\rho_{\mathfrak p}|_{G_{F_v}})$ 
satisfies the Monodromy-Weight conjecture by \cite{MR2551990}. The description of the rings $\bar R^{\sq,\psi}_v$ in the appendix of 
\cite{kisinfmc} shows that this implies $\mathfrak p$ maps to a smooth point of $\Spec \bar R^{\sq,\psi}_v[1/p].$

Now (1)-(4) and the final claim follow as in the proof of (2.2.11).
\end{proof}

Replace Lemma (2.2.15) with 

\begin{alemma}\label{corrected2.2.15}
The $\bar R_{\infty}$-module $M^i_{\infty}/M^{i-1}_{\infty}$ is non-zero if and only if for each $v|p$ we have $\mu_{n_{i,v}, m_{i,v}}(\bar \rho|_{G_{F_v}}) \neq 0.$ 
Suppose this condition holds. Then  for any prime $\mathfrak p \in \Spec \bar R_{\infty}/\pi\bar R_{\infty}$ in the support 
of $M^i_{\infty}/M^{i-1}_{\infty},$ we have 
$$\dim_{\kappa(\mathfrak p)} M^i_{\infty}/M^{i-1}_{\infty}\otimes_{\bar R_{\infty}}\kappa(\mathfrak p) \geq 2^{|R|}.$$
Moreover, if for each $v|p$ we have 
$\bar\rho|_{G_{F_v}} \nsim \left(\smallmatrix \omega \chi & * \\ 0 & \chi \endsmallmatrix \right)$ 
for any character $\chi: G_{F_v} \rightarrow \F^\times, $ then 

$$ 2^{-|R|}e(M^i_{\infty}/M^{i-1}_{\infty}, \bar R_{\infty}/\pi \bar R_{\infty}) \geq  
e_{\Sigma}\prod_{v|p} \mu_{n_{i,v}, m_{i,v}}(\bar \rho|_{G_{F_v}}) := e_{\Sigma_p}$$
where 
$$ e_{\Sigma} : = \prod_{v \in \Sigma} e(\bar R_v^{\sq,\psi}/\pi \bar R_v^{\sq,\psi}). $$
\end{alemma}
\begin{proof} The first claim is proved in Lemma (2.2.15), which also proves the support of $M^i_{\infty}/M^{i-1}_{\infty}$ is 
all of $\Spec \bar R^i_{\infty}.$ The third claim follows from the second, just as in the proof of Lemma (2.2.15).

To prove the second claim, we remark that by \cite{gee061} \S 4.6, there is a smooth, irreducible $E$-representation (as always extending $E$ if necessary) 
$\tilde \sigma_i$ of $\prod_{v|p} \GL_2(\kappa(v))$ such that $\sigma_i$ is a Jordan-H\"older factor of the reduction of $\tilde \sigma,$ 
and $S_{\tilde \sigma_i,\psi}(U,\O)_{\m}\otimes \F \iso S_{\sigma_i,\psi}(U,\O)_{\m}$ (see Lemma 4.6.3 and \emph{cf.}\ the proof of 4.6.5, 4.6.6 of {\em loc.~cit})
\footnote{The use of these auxiliary smooth representations can be avoided by working with unitary groups as in the present paper.}
. 
Now the second claim follows by applying Lemma \ref{corrected2.2.11} with $\sigma = \tilde\sigma_i.$
\end{proof}

Finally, in the proofs of Corollary (2.2.17) and Lemma (2.3.1) all the expressions 
$e(M_{\infty}/\pi M_{\infty}, \bar R_{\infty}/\pi\bar R_{\infty})$ and $e(M_{\infty}^i/\pi M_{\infty}^i, \bar R_{\infty}/\pi\bar R_{\infty})$ 
should be multiplied by $2^{-|R|}.$
\bibliographystyle{amsalpha}
\bibliography{geekisin}

\end{document}